\documentclass[11pt,english,a4paper]{smfart}
\usepackage[english]{babel}
\usepackage{xcolor}
\usepackage{pbsi}
\usepackage[T1]{fontenc}
\usepackage{stmaryrd}
\usepackage{mathabx}
\usepackage[mathscr]{euscript}
\usepackage[pagebackref,colorlinks=true,linkcolor=blue,citecolor=blue,urlcolor=red, hypertexnames=true]{hyperref}
 \usepackage[abbrev,backrefs,nobysame]{amsrefs}
 \usepackage{caption}
 \usepackage{enumerate}
\usepackage{amsmath,amssymb,bm,amsfonts}
\usepackage[all]{xy}
\entrymodifiers={+!!<0pt,\fontdimen22\textfont2>}

\usepackage{mathrsfs}

\usepackage{graphicx}

\usepackage{color}

\newtheorem{theorem}{Theorem}[section]
\newtheorem{lemma}[theorem]{Lemma}

\newtheorem{corollary}[theorem]{Corollary}
\newtheorem{proposition}[theorem]{Proposition}

\newtheorem*{remark*}{\it Remark}

\renewcommand*{\thesubsection}{\thesection\alph{subsection}}

\newcommand{\flba}[2]{
\xymatrix@C15pt{#1\ar@{|->}[r]&#2}}
\newcommand{\flcourte}[2]{
\xymatrix@C12pt{#1\ar[r]&#2}}

{\theoremstyle{definition}}

{\theoremstyle{definition}}

\theoremstyle{definition}

\newtheorem{question}[theorem]{Question}
\newtheorem{fact}[theorem]{Fact}
\newtheorem{claim}[theorem]{Claim}

{\theoremstyle{definition}\newtheorem{remark}[theorem]{Remark}}

\def\D{\ensuremath{\mathbb D}}

\def\T{\ensuremath{\mathbb T}}
\def\R{\ensuremath{\mathbb R}}
\def\Z{\ensuremath{\mathbb Z}}

\def\C{\ensuremath{\mathbb C}}
\def\Q{\ensuremath{\mathbb Q}}
\def\N{\ensuremath{\mathbb N}}

\def\g{{{\mathcal G}}}
\def\u{{{\mathcal U}}}
\def\b{{{\mathcal B}}}
\def\en{E_{N}}

\def\bth{\begin{theorem}}
\def\blm{\begin{lemma}}
\def\bpr{\begin{proposition}}
\def\bpf{\begin{proof}}
\def\epf{\end{proof}}
\def\epr{\end{proposition}}
\def\elm{\end{lemma}}
\def\eth{\end{theorem}}
\def\bco{\begin{corollary}}
\def\eco{\end{corollary}}
\def\be{\begin{enumerate}}
\def\ee{\end{enumerate}}
\def\bea{\begin{enumerate}[\rm (a)]}
\def\beun{\begin{enumerate}[\rm (1)]}
\def\bei{\begin{enumerate}[\rm (i)]}

\newcommand{\pss}[2]{\ensuremath{{\langle #1,#2\rangle}}}

\newcommand{\ds}{\displaystyle}
\newcommand{\ba}[1]{\overline{#1}}

\newcommand{\ti}[1]{\widetilde{#1}}

\newcommand{\sbt}{\,\begin{picture}(-1,0)(-1,-2)\circle*{3}\end{picture}\ }

\newcommand{\gd}{G_{\delta }}
\newcommand{\fs}{F_{\sigma}}

\newcommand{\bmx}{{\mathcal{B}}_{M}(X)}
\newcommand{\bbx}{{\mathcal{B}}_{1}(X)}
\newcommand{\bbh}{{\mathcal{B}}_{1}(H)}
\newcommand{\bh}{{\mathcal{B}}(H)}
\newcommand{\wot}{\texttt{WOT}}

\newcommand{\sot}{\texttt{SOT}}
\newcommand{\sote}{\texttt{SOT}\mbox{$^{*}$}}
\newcommand{\bx}{{\mathcal B}(X)}
\newcommand{\tbh}{{{\mathcal T}_1(H)}}
\newcommand{\tbx}{{{\mathcal T}_1(X)}}

\newcommand{\bbc}{\mathcal B_1(c_0)}

\numberwithin{equation}{section}

\author[S. Grivaux]{Sophie Grivaux}
\address[S. Grivaux]{CNRS, Laboratoire Paul Painlev\'e, UMR 8524\\
Universit\'{e} de Lille\\
Cit\'e Scientifique, B\^atiment M2\\
59655 Villeneuve d'Ascq Cedex 
(France)}
\email{sophie.grivaux@univ-lille.fr}
\author[\'{E}. Matheron]{\'{E}tienne Matheron}
\address[\'{E}. Matheron]{Laboratoire de Math\'{e}matiques de Lens\\ Universit\'{e} d'Artois\\ Rue Jean Souvraz SP 18\\ 62307 Lens (France)}
\email{etienne.matheron@univ-artois.fr}
\author[Q. Menet]{Quentin Menet}
\address[Q. Menet]{Service de Probabilit\'e et Statistique, D\'epartement de Math\'ematique\\ Universit\'{e} de Mons\\ Place du Parc 20\\ 7000 Mons (Belgium)}
\email{quentin.menet@umons.ac.be}

\begin{document}

\title[Typical $\ell_p\,$-$\,$space contractions]{Does a typical $\ell_p\,$-$\,$space contraction have a non-trivial invariant subspace?}

 \thanks{This work was supported in part by
the project FRONT of the French
National Research Agency (grant ANR-17-CE40-0021) and by the Labex CEMPI (ANR-11-LABX-0007-01). The third author is a Research Associate of the Fonds de la Recherche Scientifique - FNRS}

\begin{abstract} 
Given a Polish topology $\tau$ on $\bbx$, the set  of all contraction operators on $X=\ell_p$, $1\le p<\infty$ or $X=c_0$, we prove several results related to the following question: does a typical $T\in\bbx$ in the Baire Category sense has a non-trivial invariant subspace? In other words, is there  a dense $G_\delta$ set $\mathcal G\subseteq (\bbx,\tau)$ such that every $T\in\mathcal G$ has a non-trivial invariant subspace? 
We mostly focus on the Strong Operator Topology and the Strong$^*$ Operator Topology.%The two main Polish topologies that we consider on $\bbx$ are the Strong and Strong$^*$ Operator Topologies.
% ... or: \emph{How \emph{not} to solve the Invariant Subspace Problem, using Baire Category methods}.
\end{abstract}

\maketitle

% \emph{Warning.} J'ai chang\'{e} la d\'{e}finition de la macro \verb=\sote=.
\par\bigskip
\section{Introduction}\label{Section introduction}

Unless otherwise specified, all the Banach spaces -- and hence all the Hilbert spaces -- considered in this paper are complex, infinite-dimensional, and \emph{separable}.  %By the word ``subspace'', we will always mean a closed linear subspace (possibly finite-dimensional) of the Banach space under consideration. A subspace $E$ of a Banach space $X$ is \emph{non-trivial} if $E\neq \{ 0\}$ and $E\neq X$.
%\par\medskip
%We will denote by $\D$ the open unit disk in the complex plane, and by $\ba{\,\D}$ the closed unit disk. If $(x_{i})_{i\in I}$ is any family of vectors of $X$, we denote by $[\,x_{i}\;;\;i\in I\,]$ its closed linear span in $X$.
%\par\medskip
If $X$ is a  Banach space, we denote by $\bx$ the space of all bounded linear operators on $X$ endowed with its usual norm. For any $M>0$,  
\[
\bmx:=\{T\in\bx\;;\;\Vert T\Vert\le M\}
\]
is the closed ball of radius $M$ in $\bx$. In particular, $\bbx$ is the set of all \emph{contractions} of the Banach space $X$.
% \subsection{Context and motivation}
\par\smallskip
In this paper, we will be interested in \emph{typical} properties of Banach space contractions. The word ``typical'' is to be understood in the Baire category sense: given a Baire topological space $\mathcal  Z$ and any property (P) of elements of $\mathcal Z$, we say  that \emph{a typical $z\in\mathcal Z$ satisfies \emph{(P)}} if the set
 $\{z\in\mathcal Z\;;\;z\ \textrm{satisfies (P)}\}$ is comeager in $\mathcal Z$, \mbox{\it i.e.} this set contains a dense $\gd$ subset of $\mathcal Z$.  So, given a Banach space $X$, we need topologies on $\bbx$ which turn $\bbx$ into a Baire space. The operator norm topology does so, but it appears to be too strong to get interesting results by Baire category arguments; in particular, the fact that $(\bbx,\Vert\,\cdot\,\Vert)$ is ``usually'' non-separable is a real disadvantage. However, there are quite a few weaker topologies on $\bx$ whose restriction to each ball $\bmx$ is {Polish}, \mbox{\it i.e.} completely metrizable and {separable}. We will mostly focus on two of them: the \emph{Strong Operator Topology} (\sot) and the \emph{Strong$^{\,*}$ Operator Topology} 
(\sote). Recall that \sot\ is just the topology of pointwise  convergence, and that \sote\ is the topology of pointwise convergence for operators and their adjoints: a net $(T_{i})$ in $\bx$ converges to $T\in\bx$ with respect to \sot\  if and only if $T_{i}x\to Tx$ in norm for every $x\in X$, and $T_i\to T$ with respect to \sote\ if and  only  if $T_{i}x\to Tx$ for every $x\in X$ and $T_{i}^{*}x^{*}\to T^{*}x^{*}$ for every $x^{*}\in X^{*}$. The \emph{Weak Operator Topology} (\wot) will also make an occasional appearance ($T_i\to T$ %a net $(T_{i})\subseteq\bx$ converges to $T\in\bx$ 
with respect to \wot\ 
if and only if $T_{i}x\to Tx$ weakly for every $x\in X$).
Even though these topologies behave badly when considered on the whole space $\bx$, %; in particular, they are not metrizable. Nevertheless, 
each closed ball $\bmx$ is indeed Polish when endowed with \sot, %they become much nicer when restricted to closed balls of $X$. Indeed, for any $M>0$, $\bmx$ is a Polish space when endowed with \sot; 
 and the same holds true for \sote\ if one assumes additionally that $X^*$ is separable. (For \wot,  it is safer to assume $X$ is reflexive; and then $\mathcal B_M(X)$ is even compact and metrizable.)%(See Section \ref{Section2} below for a few more details on these two facts.)
 %The meaning of the word ``typical'' is the usual one: given any property (P) of operators on $X$ and $M>0$, we say that  \emph{a typical operator} $T\in(\mathfrak{X}_{M},\tau)$ \emph{satisfies} (P) if the set
 %$
%\{T\in\mathfrak{X}_{M}\;;\;T\ \textrm{satisfies (P)}\}
%$ is comeager in $(\mathfrak{X}_{M},\tau)$, \mbox{\it i.e.} this set contains a dense $\gd$ subset of $(\mathfrak{X}_{M},\tau)$. %given by the following definition:
%\begin{definition}\label{Definition Operateur typique}
% Let (P) be any property of operators on $X$. Let $M>0$, and let $(\mathfrak{X}_{M},\tau)$ be a Polish space of operators on $X$. We say that  \emph{a typical operator} $T\in(\mathfrak{X}_{M},\tau)$ \emph{satisfies} (P) if the set
% \[
%\{T\in\mathfrak{X}_{M}\;;\;T\ \textrm{satisfies the property (P)}\}
%\]
%is comeager in $(\mathfrak{X}_{M},\tau)$, i.e. this set contains a dense $\gd$ subset of 
%$(\mathfrak{X}_{M},\tau)$.
% \end{definition}

\par\smallskip The study of typical properties of  {contractions} %(\mbox{\it i.e.}  belonging to $\bbx$) 
was initiated, in a Hilbertian setting, by Eisner \cite{E} and Eisner-M\'{a}trai \cite{EM}. Given a Hilbert space $H$ (complex, infinite-dimensional and separable), Eisner studied in \cite{E} properties of typical operators $T\in\bbh$ for the Weak Operator Topology and proved that a typical $T\in(\bbh,\texttt{WOT})$ is unitary. The situation turns out to be completely different if one considers the Strong Operator Topology:  indeed, it was shown in \cite{EM} that a typical $T\in(\bbh,\sot)$ is unitarily equivalent to the backward shift of infinite multiplicity acting on $\ell_{2}\bigl(\Z_{+} ,\ell_2\bigr)$. So the behaviour of typical contractions on the Hilbert space is essentially fully understood in the \sot\ case, and rather well so in the \texttt{WOT} case. The general picture in the \sote\ case is more complicated, and this was studied in some detail in \cite{GMM}.%: typical contractions for \sote\  exhibit some interesting dynamical features (they were studied in detail in\cite{GMM}).
%\par\medskip
\par\smallskip
Let $X$ be a Banach space, and let $\tau$ be a topology on $\bx$ turning the ball $\bbx$ into a Polish space. If ${\mathcal{G}}$ is a subset of $\bx$ such that 
$
{\mathcal{G}}_{1}:={\mathcal{G}}\cap\bbx
$ is a $\gd$ subset of $\bbx$ for the topology $\tau$, then $({\mathcal{G}}_{1},\tau)$ is also a Polish space in its own right. Our initial goal in this paper was to determine, for certain suitable subsets ${\mathcal G}\subseteq\bx$, %Polish spaces of operators of that kind, 
whether a typical contraction  $T$ from ${\mathcal G}$ %$T\in\mathfrak X_M$ 
has a non-trivial invariant subspace, \mbox{\it i.e.} %whether, for a typical $T\in (\mathfrak X_1,\tau)$, there exists 
a closed linear subspace $E\subseteq X$ with $E\neq\{ 0\}$ and $E\neq X$ such that $T(E)\subseteq E$. %study of whether a typical $T\in\bbx$ for one of the natural Polish topologies has a non-trivial invariant subspace 
\par\smallskip
This is of course motivated by the famous \emph{Invariant Subspace Problem}, which asks, for a given Banach space $X$, whether \emph{every} operator $T\in\bx$ has a non-trivial invariant subspace. The Invariant Subspace Problem was solved in the negative by Enflo in the 70's \cite{En} for some peculiar Banach space; and then by Read \cite{R1}, who subsequently exhibited in several further papers examples of operators without invariant subspaces on some classical Banach spaces like $\ell_{1}$ and $c_{0}$. %or on $\oplus_{\ell_{2}}\mathcal J$, the $\ell_{2}\,$-$\,$direct sum  of countably many copies of the James space $\mathcal J$. 
See \cite{R2}, \cite{R3}, and also \cite{GR2} for a unified approach to these constructions. On the other hand, there are also lots of positive results. To name a few: the classical Lomonosov Theorem from  \cite{L} (any operator whose commutant contains a non-scalar operator commuting with a non-zero compact operator has a non-trivial invariant subspace);  the Brown-Chevreau-Pearcy Theorem  from \cite{BCP2} (any Hilbert space contraction whose spectrum contains the whole unit circle has a non-trivial invariant subspace); and more recently, the construction by  Argyros and Haydon \cite{AH} of Banach spaces $X$ for which the Invariant Subspace Problem has a \emph{positive} answer. We refer to the books \cite{RR} and \cite{CP}, as well as to the survey \cite{CE}, for a comprehensive overview of the known methods yielding the existence of non-trivial invariant subspaces for various classes of operators on Banach spaces.    %Moreover, many classes of operators are known to admit non-trivial invariant subspaces. Among these, one may single out the class of operators whose commutant contains an operator which is not a multiple of the identity and commutes with a non-zero compact operator (this is the famous Lomonosov Theorem from \cite{L}). Various functional calculi also permit to construct, in different settings, non-trivial invariant subspaces. One of the most remarkable results making use of this kind of method is the Brown-Chevreau-Pearcy Theorem \cite{BCP2}, which states that any contraction on a complex Hilbert space  whose spectrum contains the whole unit circle has a non-trivial invariant subspace. We refer to the books \cite{RR} and \cite{CP}, as well as to the survey \cite{CE}, for a comprehensive overview of known methods yielding the existence of non-trivial invariant subspaces for various classes of operators on Banach spaces. 

\par\smallskip
A problem closely related to the Invariant Subspace Problem is the Invariant \emph{Subset} Problem, which asks (for a given Banach space $X$) whether {every} operator on $X$ admits a non-trivial invariant closed set; equivalently, if there exists at least one $x\neq 0$ in $X$ whose orbit under the action of $T$ is not dense in $X$. Again, this question was solved in the negative by Read in \cite{R4}, who constructed an operator on the space $\ell_{1}$ with no non-trivial invariant closed set. Further examples of such operators on a class of Banach spaces including $\ell_1$ and $c_{0}$ %and $\bigoplus_{\ell_{2}}\mathcal J$ 
were exhibited in \cite{GR2}.  One may also  add that the Invariant Subset Problem has been a strong motivation for  the study of  \emph{hypercyclic} operators. Recall that an operator $T\in\mathcal L(X)$ is  said  to be hypercyclic if there exists some vector $x\in  X$ whose orbit under the action of $T$ is  dense in $X$; in which case $x$ is said to be hypercyclic for $T$.   With this terminology, the Invariant  Subset Problem for $X$ asks whether there exists an operator $T\in\mathcal L(X)$ such that every $x\neq 0$ is hypercyclic  for  $T$. We  refer to   the books  \cite{BM}  and  \cite{GEP}  for more  information  on  hypercyclicity and other dynamical properties of linear operators.
\par\smallskip
Despite considerable efforts, the Invariant Subspace Problem and the Invariant Subset Problem remain stubbornly open in the reflexive setting, and in particular for the Hilbert space. Therefore, it seems quite natural to address the \textit{a priori} more tractable ``generic'' version of the problem; and this is meaningful even on spaces like $\ell_1$ or $c_0$, since the mere existence of counterexamples says nothing about the typical character of the property of having a non-trivial invariant subspace.%The spaces constructed by Argyros and Haydon in \cite{AH} provide examples of non-reflexive spaces as well as of reflexive spaces $X$ such that every operator $T\in\bbx$ has a non-trivial invariant subspace, but no example of a reflexive space supporting an operator without non-trivial invariant subspaces has been found to this day.
%\par\smallskip
%In order to  understand which geometric properties of the underlying Banach space $X$ may lead to the existence of an operator on $X$ without non-trivial subspaces, 

%; that is, to investigate  whether typical operators from certain Polish spaces of  contractions have a non-trivial invariant subspace, and for which reason. 
We will mostly focus on operators on $X=\ell_{p}$, $1\le p<\infty$ or $X=c_{0}$. This already offers a wide range of possibilities: the spaces $\ell_{1}$ and $c_{0}$ are known to support operators without non-trivial invariant subspaces (and even without non-trivial invariant closed sets), whereas the spaces $\ell_{p}$ for $1<p<\infty$ are all reflexive -- so it is not known if they support operators without non-trivial invariant subspaces -- but have rather different geometric properties in the three cases $p=2$, $1<p<2$ and $p>2$. The main part of our work will consist in the study of the spaces $({\mathcal{B}}_{1}(X),\tau)$ for one of the topologies $\tau=\sot$ or $\tau=\sote$ (when $\tau=\sote$, we have  to exclude $X=\ell_1$ since it has %to suppose that $1<p<\infty$, since $\ell_{1}$ has 
a non-separable dual).

\par\smallskip
It follows from the topological $0\;$-$\,1$ law that on any of the spaces $X=\ell_p$ or $c_0$, either a typical contraction has a non-trivial invariant subspace, or a typical contraction does not have a non-trivial invariant subspace (see Proposition \ref{0-1}). In view of the complexity of the known counterexamples to the Invariant Subspace Problem on $\ell_1$ and $c_0$, and of the very serious difficulty encountered when trying to construct such a counterexample on a reflexive space, it 
seems reasonable to expect that the correct alternative is the first one, \mbox{\it i.e.} a typical contraction on $X$ does have a non-trivial invariant subspace; but our efforts to prove that have remained unsuccesful, except in the ``trivial'' case $X=\ell_2$ (see below) and for $X=\ell_1$. All is not lost, however, since we do obtain a few significant results related to our initial goal. In the next section,  we describe in some detail the contents of the paper.%However, we have been unable to prove it.

\subsection{Notations}
The following notations will be used throughout the paper. 
\par\smallskip
- We denote by $\D$ the open unit disk in $\C$, and by $\T$ the unit circle

\smallskip
- If $(x_{i})_{i\in I}$ is any family of vectors in a Banach space $X$, we denote by $[\,x_{i}\;;\;i\in I\,]$ its closed linear span in $X$. 

\smallskip
- When $X=\ell_{p}$, $1\le p<\infty$ or $X=c_{0}$, we denote by $(e_{j})_{j\ge 0}$ the canonical basis of $X$, and by $(e_{j}^*)_{j\ge 0}\subseteq X^*$ the associated sequence of coordinate functionals. %A basis of open neighborhoods of $T_{0}\in\bbx$ for the \sot\ is given by the family of sets
%\[
%\mathfrak{U}_{\varepsilon ,N}=\{T\in\bbx\;;\;\forall\,j=0,\dots,N,\ \Vert(T-T_{0})e_{j}\Vert<\varepsilon \},\quad \varepsilon >0,\ N\ge 0.
%\]
For any set $I$ of nonnegative integers, we denote by $P_I:X\to [e_i\, ;\, i\in I]$ the canonical projection of $X$ onto $[e_i\, ;\, i\in I]$. When $I=[0,N]$ for some $N\geq 0$, we may also write $P_N$ instead of $P_{[0,N]}$. For every $N\ge 0$, we set $E_N:= [e_0,\dots ,e_N]$ and  $F_N:=[e_j;\; j>N]$.
%\par\smallskip

\smallskip
- If $T$ is a bounded operator on a (complex) Banach space $X$, we denote by $\sigma(T)$ its spectrum, by $\sigma_e(T)$ its essential spectrum, by $\sigma_p(T)$ its point spectrum -- \mbox{\it i.e.} the eigenvalues of $T$ -- and by $\sigma_{ap}(T)$ its approximate point spectrum, \mbox{\it i.e.} the set of all $\lambda\in\C$ such that $T-\lambda$ is not bounded below.

\section{Main results}
In Section \ref{Section2}, we review some properties of typical contractions of $\ell_{2}$ for the topology \sot, and we prove some general properties of \sot$\,$-$\,$typical contractions of the spaces $\ell_{p}$ and $c_{0}$. We mention in particular the following result (labeled as Proposition \ref{aameliorer}). 

\begin{proposition}\label{theoreme zero}
 Let $X=\ell_{p}$, $1\leq p<\infty$,  or $X=c_{0}$. A typical $T\in(\bbx,\emph{\sot})$ has the following spectral properties: $T-\lambda $ has dense range for every $\lambda \in\C$, and $\sigma (T)=\sigma _{{ap}}(T)=
 \ba{\,\D}$.
\end{proposition}

%Here and afterwards, $\D$ denotes the open unit disk in $\C$, and $\ba{\,\D}$ is the closed unit disk. Recall also that $\sigma(T)$ denotes the spectrum of the operator $T$, and that $\sigma_{ap}(T)$ is the \emph{approximate point spectrum of $T$}, \mbox{\it i.e.} the set of all $\lambda \in\C$ such that $T-\lambda $ is not bounded below.
\par\smallskip 
When $X=\ell_{2}$, the Eisner-M\'{a}trai result from \cite{EM} mentioned above implies that a typical operator $T$ in $({\mathcal{B}}_{1}(\ell_{2}),\sot)$ satisfies much stronger properties: $T^{*}$ is a non-surjective isometry, $T-\lambda $ is surjective for every $\lambda \in\D$, and every $\lambda \in\D$ is an eigenvalue of $T$ with infinite multiplicity. It immediately follows that an \sot$\,$-$\,$typical contraction on $\ell_{2}$ has an enormous amount of invariant subspaces.	
\par\smallskip
A natural step towards the understanding of the behaviour of \sot$\,$-$\,$typical contractions on other $\ell_{p}\,$-$\,$spaces is to determine whether they enjoy similar properties. Somewhat surprisingly, it turns out that \sot$\,$-$\,$typical contractions on $\ell_{1}$ behave very much like \sot$\,$-$\,$typical contractions on $\ell_{2}$ in this respect. Indeed, we prove in Section \ref{Section3} the following result (labeled as Theorem \ref{l1}):

\begin{theorem}\label{Premier th}
 Let $X=\ell_{1}$. A typical $T\in(\bbx,\emph{\sot})$ satisfies the following properties: $T^*$ is a non-surjective isometry, $T-\lambda $ is surjective for every $\lambda \in\D$, and every $\lambda \in\D$ is an eigenvalue of $T$ with infinite multiplicity. In particular, a typical $T\in(\bbx,\emph{\sot})$ has non-trivial invariant subspaces.
\end{theorem}

The situation is radically different when $X=\ell_{p}$ for $1<p<\infty$ and $p\neq 2$, or when $X=c_{0}$. It is not hard to show that in this case, the set of all co-isometries is nowhere dense in $(\bbx,\sot)$, so that a statement analogous to the first part of Theorem \ref{Premier th} cannot be true (see Proposition \ref{lppasisom}). Moreover, we prove in Section \ref{Section4} that, at least when $X=\ell_{p}$ for $p>2$ or $X=c_0$, the typical behaviour regarding the eigenvalues is in fact to have no eigenvalue at all.  More precisely, we obtain the following results (see Theorem \ref{Valeurs propres c0} and Theorem \ref{Valeurs propres lp}). %or when $X=c_{0}$, 
%an \sot$\,$-$\,$typical contraction on $X$ has no eigenvalue 
%(this is Theorem \ref{Valeurs propres lp}):

\begin{theorem}\label{Second th} If $X=c_0$, then a typical $T\in(\bbx,\emph{\sot})$ has no eigenvalue. If $X=\ell_p$ with $p>2$ then, for any $M>1$, a typical $T\in(\bbx,\emph{\sot})$ is such that $(MT)^*$ is hypercyclic.
\end{theorem}

The conclusion in the $\ell_p\,$-$\,$case is indeed stronger than asserting that $T$ has  no  eigenvalue, since  it is well known that  the adjoint of a hypercyclic operator  cannot have  eigenvalues. 

Our proofs in the $c_0\,$-$\,$case and in the $\ell_p\,$-$\,$case are quite different. In the $c_0\,$-$\,$case, we make use of the so-called \emph{Banach-Mazur} game. In the $\ell_p\,$-$\,$case, we observe that the result is already known \emph{for the topology} $\sote$  by \cite{GMM}, and then we prove the following rather unexpected fact, which might be useful in other situations: if $p>2$, then any $\sote$-$\,$comeager subset of $\mathcal B_1(\ell_p)$ is also \sot$\,$-$\,$comeager (see Theorem \ref{bizarre?}). Our proof of this latter result heavily relies on properties of the $\ell_{p}\,$-$\,$norm which are specific to the case $p>2$, and so it does not extend to the case $1<p<2$. In fact, we have not been able to determine whether an \sot$\,$-$\,$typical contraction on $X=\ell_{p}$, $1<p<2$ has eigenvalues. We do not know either if hypercyclicity of $(MT)^*$ is typical in the $c_0\,$-$\,$case.%$X=c_0$. %We don't know either if a typical 
\par\smallskip
Theorem \ref{Second th} might suggest that in fact, an \sot$\,$-$\,$typical contraction on $\ell_{p}$, $p>2$ does \emph{not} have a non-trivial invariant subspace. If it were true, this would solve the Invariant Subspace Problem for 
$\ell_p$, presumably with rather ``soft'' arguments; so we would not bet anything on it. To support this feeling, a result of M\"uller \cite{M2} can be brought to use to prove that for $X=\ell_p$ with $1<p<\infty$, a typical $T\in(\bbx,{\sot})$ has a non-trivial invariant 
\emph{closed convex cone} (this is Proposition \ref{Mullererie}). When $X=c_0$, we show that a typical $T\in(\bbx,{\sot})$ has a non-zero \emph{non-supercyclic} vector -- \mbox{\it i.e.} a vector $x\neq 0$ whose scaled orbit $\{ \lambda T^n x;\; n\geq 0\,,\, \lambda\in\C\}$ is not dense in $X$ -- and hence a non-trivial invariant star-shaped closed set (this is Proposition \ref{conec0}).

\smallskip The following table summarizes most of our knowledge about  \sot$\,$-$\,$typical contractions on $\ell_p$ or $c_0$.\\
{\footnotesize
\begin{center}
\begin{tabular}{c|c|c|c|c|c|}
\cline{2-6}
  & $p=1$ &$1<p<2$  & $p=2$&$2<p<\infty$ & $c_0$\\
\hline
%\multicolumn{1}{|c|}{}&&&&&\\
\multicolumn{1}{|c|}{$\sigma(T)%=\sigma_{ap}(T)=\overline{\mathbb{D}}
$}     & $\overline{\mathbb{D}}$ &   $\overline{\mathbb{D}}$ & $\overline{\mathbb{D}}$ & $\overline{\mathbb{D}}$ & $\overline{\mathbb{D}}$\\
\multicolumn{1}{|c|}{}&{\tiny (Prop~\ref{aameliorer})}&{\tiny (Prop~\ref{aameliorer})}&{\tiny (Eisner-M\'atrai)}&{\tiny (Prop~\ref{aameliorer})}&{\tiny (Prop~\ref{aameliorer})}\\
%\multicolumn{1}{|c|}{}&&&&&\\
\hline
%\multicolumn{1}{|c|}{}&&&&&\\
\multicolumn{1}{|c|}{$T^*$ is an isometry}&  {Yes}      &  No  &   {Yes} & No &{No}\\
\multicolumn{1}{|c|}{} &{\tiny (Thm~\ref{l1})}&{\tiny (Prop~\ref{lppasisom})}&{\tiny (Eisner-M\'atrai)}&{\tiny (Prop~\ref{lppasisom})}&{\tiny (Prop~\ref{lppasisom})}\\
%\multicolumn{1}{|c|}{}&&&&&\\
\hline
%\multicolumn{1}{|c|}{}&&&&&\\
\multicolumn{1}{|c|}{$T$ has a non-trivial}     & Yes &   Yes & Yes & Yes & ? \\
\multicolumn{1}{|c|}{invariant closed cone}&{\tiny (Thm~\ref{l1})}&{\tiny (Prop~\ref{Mullererie})}&{\tiny (Eisner-M\'atrai)}&{\tiny (Prop~\ref{Mullererie})}&\\
%\hline
%&&&&&\\
%$T-\lambda$ has dense range for every $\lambda\in \mathbb{C}$ & Yes &   Yes      & Yes & Yes & Yes \\
%{\tiny (Theorem~\ref{aameliorer})}&&&&&\\
%\multicolumn{1}{|c|}{}&&&&&\\
\hline
\multicolumn{1}{|c|}{$T$ has a non-zero}     & Yes &   Yes & Yes & Yes & Yes \\
\multicolumn{1}{|c|}{non-supercyclic vector}&{\tiny (Thm~\ref{l1})}&{\tiny (Prop~\ref{Mullererie})}&{\tiny (Eisner-M\'atrai)}&{\tiny (Prop~\ref{Mullererie})}&{\tiny (Prop~\ref{conec0})}\\
%\multicolumn{1}{|c|}{}&&&&&\\
\hline
\multicolumn{1}{|c|}{$\sigma_{p}(T)$}   &$\D$& $ ?$ & $\D$ & $\emptyset$ & $\emptyset $ \\
\multicolumn{1}{|c|}{}&{\tiny (Thm~\ref{l1})}&&{\tiny (Eisner-M\'atrai)}&{\tiny (Thm~\ref{Valeurs propres lp})}&{\tiny(Thm~\ref{Valeurs propres c0})}\\
%\multicolumn{1}{|c|}{}&&&&&\\
\hline
%\multicolumn{1}{|c|}{}&&&&&\\
\multicolumn{1}{|c|}{$T$ has a non-trivial }    & Yes &   ? & Yes& ? & ? \\
\multicolumn{1}{|c|}{invariant subspace}&{\tiny (Thm~\ref{l1})}&&{\tiny (Eisner-M\'atrai)}&&\\
%\multicolumn{1}{|c|}{}&&&&&\\
\hline
\end{tabular}
\end{center}

\smallskip
\captionof{table}{Properties of \sot$\,$-$\,$typical $T\in\bbx$}
}

\medskip

In Section \ref{Section5}, we change the setting and consider a subclass ${\mathcal{T}}_{1}(X)$ of contractions on $X=\ell_{p}$, $1\le p<\infty$ or $X=c_{0}$ defined as follows: if $(e_{j})_{j\ge 0}$ denotes the canonical basis of $X$, then ${\mathcal{T}}_{1}(X)$ consists of all  operators $T\in\bbx$ which are ``triangular plus~1'' with respect to $(e_{j})$, with positive entries on the first subdiagonal:
\[
{\mathcal{T}}_{1}(X)=\bigl\{T\in\bbx\;;\;Te_{j}\in[\,e_{0},\dots,e_{j+1}\,]\ \textrm{and}\ \pss{Te_{j}}{e_{j+1}}>0\ \textrm{for every}\ j\ge 0\bigr\}.
\]
%Here and afterwards, if $(x_{i})_{i\in I}$ is any family of vectors in $X$, we denote by $[\,x_{i}\;;\;i\in I\,]$ its closed linear span in $X$. 

The class ${\mathcal{T}}_{1}(X)$ is rather natural for at least two reasons. Firstly, in the Hilbertian setting, any cyclic operator $T\in{\mathcal{B}}_{1}\bigl(\ell_{2}\bigr)$ is unitarily equivalent to some operator in ${\mathcal{T}}_{1}\bigl(\ell_{2}\bigr)$. Secondly, all the operators without non-trivial invariant subspaces constructed by Read in  \cite{R1}, \cite{R2}, \cite{R3} belong to the classes ${\mathcal{T}}_{1}(\ell_{1})$ or ${\mathcal{T}}_{1}\bigl(c_{0}\bigr)$. The extensions considered in \cite{GR2} are also of this form, as well as the operators with few non-trivial invariant subspaces constructed in \cite{GR1}. Thus, ``{Read's type operators}'' on $X=\ell_{p}$ or $c_{0}$ belong to ${\mathcal{T}}_{1}(X)$, and hence it is  quite natural to investigate whether an \sot$\,$-$\,$typical or \sote-$\,$typical operator from ${\mathcal{T}}_{1}(X)$ has a non-trivial invariant subspace. Again, our results here are rather partial. We prove the following ``local'' version of the Eisner-M\'{a}trai result from \cite{EM} (see Corollary \ref{T1l2}).

\begin{theorem}\label{Septieme th}
 Let $X=\ell_{2}$. A typical $T\in({\mathcal{T}}_{1}(X),\emph{\sot})$ is unitarily equivalent to the backward shift of infinite multiplicity acting on $\ell_{2}\bigl(\Z_{+},\ell_{2}\bigr)$. In particular, it has eigenvalues, and hence non-trivial invariant subspaces.
\end{theorem}

In view of the results of Section \ref{Section3}, it is not too surprising that we are also able to settle the case $X=\ell_{1}$ (see Proposition \ref{T1l1}):

\begin{theorem}\label{Huitieme th}
 Let $X=\ell_{1}$. A typical $T\in({\mathcal{T}}_{1}(X),\emph{\sot})$ has all the properties listed in \emph{Theorem \ref{Premier th}}. In particular, it has a non-trivial invariant subspace.
\end{theorem}
        
%These results say that counterexamples to the Invariant Subspace Problem on $X=\ell_1$ or $\ell_2$ cannot be found by a ``soft" Baire category argument in the space $\mathfrak T_1(X)$; which is rather natural in view of the complexity of the known constructions on $\ell_1$. 

\par\smallskip
We move over in Section \ref{Section6} to the study of typical contractions on $X=\ell_{p}$, $1<p<\infty$ with respect to the topology \sote. The spectral properties of \sote-$\,$typical contractions are rather different from those of 
\sot$\,$-$\,$typical contractions, at least for $p=2$, since for example a typical $T\in(\mathcal B(\ell_2),\sote)$ has no eigenvalue (see  \cite{EM} or \cite{GMM}). Regarding invariant subspaces, the situation is exactly the same as in the \sot$\,$-$\,$case: we can prove that an \sote-$\,$typical contraction on $X=\ell_p$ has a non-trivial  invariant subspace only in the Hilbertian case $p=2$ (recall that $p=1$ is not considered in the \sote$\,$ setting); and to do that, we have to use a quite non-trivial result, namely the Brown-Chevreau-Pearcy Theorem from \cite{BCP2}. We are also able to prove (using the result of M\"uller \cite{M2}) that for any  $1<p<\infty$, a typical $T\in(\mathcal B_1(\ell_p),{\sote})$ has a non-trivial invariant closed cone. These facts are summarized in  the following  table. \\

{\footnotesize
\begin{center}
\begin{tabular}{c|c|c|c|}
\cline{2-4} &$1<p<2$  & $p=2$&$2<p<\infty$ \\
%&&&\\
\hline%\hline
%\multicolumn{1}{|c|}{}&&&\\
\multicolumn{1}{|c|}{$\sigma(T)%=\sigma_{ap}(T)=\overline{\,\mathbb{D}}
$ } & $\overline{\mathbb{D}}$ & $\overline{\mathbb{D}}$ & $\overline{\mathbb{D}}$ \\
\multicolumn{1}{|c|}{}&{\tiny (Prop~\ref{Proposition 6.1})}&{\tiny (Prop~\ref{Proposition 6.1})}&{\tiny (Prop~\ref{Proposition 6.1})}\\
%\multicolumn{1}{|c|}{}&&&\\
\hline
%\multicolumn{1}{|c|}{}&&&\\
\multicolumn{1}{|c|}{$T^*$ is an isometry}&   {No} & {No} & {No}\\
\multicolumn{1}{|c|}{}&{\tiny (Prop~\ref{Proposition 6.0})}&{\tiny (Prop~\ref{Proposition 6.0})}&{\tiny (Prop~\ref{Proposition 6.0})}\\
%\multicolumn{1}{|c|}{}&&&\\
\hline
%\multicolumn{1}{|c|}{}&&&\\
\multicolumn{1}{|c|}{$T$ has a non-trivial}   & Yes & Yes & Yes \\
\multicolumn{1}{|c|}{invariant closed cone}&{\tiny (Cor~\ref{Mul7})}&{\tiny (Cor~\ref{Mul7})}&{\tiny (Cor~\ref{Mul7})}\\
%\hline
%&&&&&\\
%$T-\lambda$ has dense range for every $\lambda\in \mathbb{C}$ & Yes &   Yes      & Yes & Yes & Yes \\
%\multicolumn{1}{|c|}{}&&&\\
%{\tiny (Theorem~\ref{aameliorer})}&&&&&\\
\hline
%\multicolumn{1}{|c|}{}&&&\\
\multicolumn{1}{|c|}{$\sigma_{p}(T)$}   & {$\emptyset$} & {$\emptyset$}  & {$\emptyset$}  \\
\multicolumn{1}{|c|}{}&{\tiny (Prop~\ref{Proposition 6.0})} &{\tiny (Prop~\ref{Proposition 6.0})}&{\tiny (Prop~\ref{Proposition 6.0})}\\
%\multicolumn{1}{|c|}{}&&&\\
\hline
%\multicolumn{1}{|c|}{}&&&\\
\multicolumn{1}{|c|}{$T$ has a non-trivial  }   & ? &   Yes & ? \\
\multicolumn{1}{|c|}{invariant subspace}&&{\tiny (Cor~\ref{invl2})}&\\
%\multicolumn{1}{|c|}{}&&&\\
\hline
\end{tabular}
\end{center}

\smallskip
\captionof{table}{Properties of \sote-$\,$typical $T\in\bbx$}
}

\medskip
Since \sote-$\,$typical contractions of a Hilbert space are by far not so well understood as \sot$\,$-$\,$typical contractions, it is interesting to investigate whether they satisfy some of the standard criteria implying the existence of a non-trivial invariant subspace. What about the Lomonosov Theorem, for instance? In this direction, we obtain the following result (this is Theorem \ref{Theorem 6.2}):

\begin{theorem}\label{Sixieme th}
 Let $X=\ell_{2}$. A typical $T\in(\bbx,\emph{\sote})$ does not commute with any non-zero compact operator.
\end{theorem}

\par\smallskip
It should be apparent that this paper rises  more questions than it answers. We collect some of them in Section \ref{Questions}.
%After reading this introduction, one might conclude that our results are rather fragmented, and that we stay somehow close to the surface without understanding what is really going on. To say it in a more optimistic way: this paper certainly rises  more questions than it answers. We collect some of them in Section \ref{Questions}.

%- Let $H$ be a Hilbert space.
% For every $T\in\bh$, we  denote  by  $\{  T\}'$ the commutant of  $T$:	
%\[
%\{T\}':=\{A\in\bbh\;;\;AT=TA\}.
%\]
%Also, 
%we denote by $\Vert T\Vert_{e}$ the essential norm of $T$, that is the distance of $T$ to the algebra ${\mathcal{K}}(H)$ of compact operators:
%\[
%\Vert T\Vert_{e}=\inf_{K\in{\mathcal{K}}(H)}\Vert T-K\Vert;
%\]
%and by $r_{e}(T)$ the essential spectral radius of T:
%\[
%r_{e}(T)=\max\,\{\vert z\vert \;;\;z\in\sigma _{e}(T)\}.
%\]

\section{Some general properties of \sot$\,$-$\,$typical contractions}\label{Section2}
For the convenience of the reader, we begin this section by giving a quick proof of the fact that when $X$ is a (separable) Banach space, the closed balls $(\bmx,\sot)$ are indeed Polish spaces. The proof for \sote\ when $X^*$ is separable would be essentially the same.
\begin{lemma}\label{Lemme polonais}
 For any {\rm (}separable{\rm )} Banach space $X$ and any $M>0$, the ball $\bmx$ is a Polish space when endowed with the topology \emph{\sot}.
\end{lemma}
\begin{proof}
 Let $\mathbf Q$ be a countable dense subfield of $\C$, and let $Z\subseteq X$ be countable dense subset of $X$ which is also a vector space over $\mathbf Q$. Consider the set 
\[\mathcal L_M:=\bigl\{ L:Z\to X \;;\; L\;\hbox{is $\mathbf Q$-linear and $\forall z\in Z\;:\; \Vert L(z)\Vert\leq M\Vert z\Vert$}\bigr\}.\]
This is a closed subset of $X^Z$ with respect to the product topology, hence a Polish space. Moreover, it is easily checked that the map $T\mapsto T_{| Z}$ is a homeomorphism from $(\bmx,\sot)$ onto $\mathcal L_M$. So $(\bmx,\sot)$ is Polish.
\end{proof}
\par\medskip

Next, we prove a general result showing that all the properties we will be considering in  this paper are either typical or ``atypical''. If $X$ is a (separable) Banach space, let us denote by ${\rm Iso}(X)$ the group of all surjective linear isometries $J:X\to X$. 
We say that a set $\mathcal A\subseteq\bbx$ is \emph{${\rm Iso}(X)\,$-$\,$invariant} if $J\mathcal AJ^{-1}=\mathcal A$ for every $J\in{\rm Iso}(X)$.

\bpr\label{0-1} Let $X=\ell_p$, $1\leq p<\infty$ or $X=c_0$. If $\mathcal A\subseteq(\bbx,\emph{\sot})$ has the Baire property and is ${\rm Iso}(X)\,$-$\,$invariant, then $\mathcal A$ is either meager or comeager in $\bbx$.
\epr

\begin{proof}[Proof of Proposition \ref{0-1}] Recall first that ${\rm Iso}(X)$ becomes a Polish group when endowed with \sot. This is well known and true for any (separable) Banach space $X$; but we give a quick sketch of proof for convenience of the reader. First,  ${\rm Iso}(X)$ is a $\gd$ subset of $(\bbx,\sot)$, because an operator $J\in\bbx$ belongs to ${\rm Iso}(X)$ if and only if it is an isometry with dense range (the first condition obviously defines a closed set, and the second one is easily seen to define a $\gd$ set). So $({\rm Iso}(X),\sot)$ is a Polish space. Next, multiplication is continuous on ${\rm Iso}(X)\times {\rm Iso}(X)$ since it is continuous on $\bbx\times \bbx$. Finally,  if $(J_n)$ is a sequence in ${\rm Iso}(X)$ such that $J_n\to J\in {\rm Iso}(X)$, then $\Vert J_n^{-1}x-J^{-1}x\Vert = \Vert x-J_nJ^{-1}x\Vert\to 0$ for every $x\in X$, so $J_n^{-1}\to J^{-1}$; hence the map $J\mapsto J^{-1}$ is continuous on ${\rm Iso}(X)$.

The Polish group ${\rm Iso}(X)$ acts continuously by conjugacy on $\bbx$. By the \emph{topological $0\,$-$\,1$ law} (see \cite[Theorem 8.46]{Ke}), it is enough to show that the action is topologically transitive, \mbox{\it i.e.} that for any pair $(\mathcal U,\mathcal V)$ of non-empty open subsets of $(\bbx,\sot)$, one can find $J\in{\rm Iso}(X)$ such that $(J\mathcal U J^{-1})\cap \mathcal V\neq\emptyset$.

Choose $A\in\mathcal U$ and $B\in\mathcal V$. For any integer $N\geq 0$, set $A_N:= P_NAP_N$ and $B_N:=P_NBP_N$, considered as operators on $E_N=[e_0,\dots ,e_N]$. Denote by $\widetilde{A}_N$ and $\widetilde{B}_N$ 
the copies of $A_N$ and $B_N$ living on $\widetilde{E}_N:=[e_{N+1},\dots , e_{2N+1}]$. If $N$ is large enough, then the operator 
$T_N:= A_N\oplus \widetilde{B}_N$ -- considered as an operator on $X$ -- belongs to $\mathcal U$, and $S_N:= B_N\oplus\widetilde{A}_N$ belongs to $\mathcal V$. Let us fix such an integer $N$, and let $J\in{\rm Iso}(X)$ be the isometry exchanging $e_j$ and $e_{N+1+j}$ for $j=0,\dots ,N$ and such that $Je_k=e_k$ for all $k>2N+1$. Then $JT_NJ^{-1}=S_N$; so we have indeed $(J\mathcal UJ^{-1})\cap \mathcal V\neq\emptyset$.
\epf
\bco If $X=\ell_p$ or $c_0$ then, either a typical $T\in\bbx$ has a non-trivial invariant subspace, or a typical $T\in\bbx$ does not have a non-trivial invariant subspace. %Likewise, either a typical $T\in\bbx$ has eigenvalues, or a typical $T\in\bbx$ does not have eigenvalues
\eco
\bpf Let $\mathcal A$ be the set of all $T\in\bbx$ having a non-trivial invariant subspace. It is clear that $\mathcal A$ is ${\rm Iso}(X)\,$-$\,$invariant; so we just have to show that $\mathcal A$ has the Baire property. Now, if $T\in\bbx$ then
\[ T\in\mathcal A\iff \exists (x,x^*)\in X\times B_{X^*}\;:\;\bigl( x\neq0\;, \;x^*\neq 0\quad{\rm and}\quad \forall n\in\Z_+\;:\; \pss{x^*}{T^nx}=0\bigl).\]
Since the relation under brackets defines a Borel (in fact, $F_\sigma$) subset of the Polish space $\bbx\times X\times (B_{X^*}, w^*)$, this shows that $\mathcal A$ is an \emph{analytic} subset of $\bbx$. In particular, $\mathcal A$ has the Baire property. (See \cite{Ke} for background on analytic sets.)
\epf

As mentioned in the introduction, we know what is actually true only for $X=\ell_2$ and $X=\ell_1$. Likewise, one can say that either a typical $T\in\bbx$ has eigenvalues, or a typical $T\in\bbx$ does not have eigenvalues; but we do not know what is true for $X=\ell_p$, $1<p<2$.

\par\medskip

When $X$ is a Hilbert space (so that we call it $H$), the \sot$\,$-$\,$typical properties of contractions are essentially fully understood, thanks to the following result of Eisner and M\'{a}trai \cite{EM}. Let us denote by $\ell_2(\Z_+, \ell_2)$ the infinite direct sum of countably many copies of 
$\ell_2$, and let $B_\infty$ be the canonical backward shift acting on $\ell_2(\Z_+, \ell_2)$, \mbox{\it i.e.} the operator defined by $B_\infty (x_0, x_1, x_2, \dots):=(x_1,x_2, x_3,\dots )$ for every $(x_0,x_1,x_2,\dots )\in\ell_2(\Z_+,\ell_2)$.%\mbox{\it i.e.} the direct sum of countably many copies of the canonical backward shift $B$ on $\ell_2$.
\bth\label{EM}A typical $T\in(\bbh,\emph{\sot})$ is unitarily equivalent to $B_\infty$.
\eth

\smallskip
This theorem is proved in \cite{EM} by showing that a typical $T\in(\bbh,{\sot})$ has the following properties: $T^*$ is an isometry, $\dim\ker(T)=\infty$, and $\Vert T^nx\Vert\to 0$ as $n\to\infty$ for every $x\in H$. The result then follows from the Wold decomposition theorem.

\smallskip

%Theorem \ref{EM} has the following consequences:
\bco\label{vrac} A typical $T\in(\bbh,\emph{\sot})$ has the following properties.
\bea
\item\label{isomH}%[\rm (a)] 
$T^*$ is a non-surjective isometry.
\item\label{surjH} $T-\lambda$ is surjective for every $\lambda\in\D$.
\item\label{eigenH} Every $\lambda\in\D$ is an eigenvalue of $T$ with infinite multiplicity. %In particular, $\sigma_p(T)$ contains $\D$.
\item\label{spectrumH}%[\rm (b)] 
$\sigma(T)=\sigma_e(T)=\sigma_{ap}(T)=\overline{\,\D}$.
\item\label{orbit0H} $\Vert T^nx\Vert\to 0$ for every $x\in H$.
\ee
\eco

\smallskip In particular, it follows immediately from (\ref{eigenH}) that an \sot$\,$-$\,$typical contraction on $H$ has a wealth of invariant subspaces. So we may state

\bco A typical $T\in(\bbh,\emph{\sot})$ has a non-trivial  invariant subspace.
\eco

\smallskip
One may also note that property (\ref{isomH}) by itself implies as well that $T$ has a non-trivial invariant subspace.  Indeed, by a classical result of Godement \cite{Go}, any Banach space isometry has a non-trivial invariant subspace.  %(if the underlying space has dimension greater than $1$). 
Hence, if $X$ is a reflexive Banach space, then any operator $T\in\bx$ such that $T^*$ is an isometry has a non-trivial invariant subspace. 
\par\medskip

\par\medskip
In the next three propositions, we show that something remains of Corollary \ref{vrac}
in a non-Hilbertian setting.

\bpr\label{Tnto0} Let $X$ be a Banach space. A typical $T\in(\bbx,\emph{\sot})$ is such that $\Vert T^nx\Vert\to 0$ as $n\to\infty$ for all $x\in X$.
\epr
\bpf The key point is that if $T\in\bbx$ is a contraction and $x\in X$, then the sequence $\Vert T^nx\Vert$ is non-increasing, and hence $T^nx\to 0$ if and only if $\underline{\lim}\, \Vert T^nx\Vert=0$. Let us denote by ${\mathcal C}_0$ the set of all $T\in\bbx$ such that 
$\forall x\in X\;:\; \Vert T^nx\Vert\to 0$. Let also $Z$ be a countable dense subset of $X$.  By what we have just said, the following equivalence holds true for any $T\in\bbx$:
\[ T\in{\mathcal C}_0\iff \forall z\in Z\;\forall K\in\N\;\,\exists n\;:\; \Vert T^nz\Vert<1/K.\]
Since the map $T\mapsto T^n$ is $\sot$-continuous on $\bbx$ for each $n\in\N$, this shows that ${\mathcal C}_0$ is a $G_\delta$ subset of $(\bbx,\sot)$. Moreover, ${\mathcal C}_0$ is dense in $\bbx$ because it contains every operator $T$ with $\Vert T\Vert<1$.
\epf

\smallskip Recall that a Banach space $X$ is said to have the \emph{Metric Approximation Property} (for short, the {MAP}) if, for any compact set $K\subseteq X$ and every $\varepsilon >0$, one can find a finite rank operator $R\in\bx$ with $\Vert R\Vert\leq 1$ such that 
$\Vert Rx-x\Vert<\varepsilon$ for all $x\in K$. Equivalently, $X$ has the MAP if and only if there exists a net $(R_i)\subseteq \bbx$ consisting of finite rank operators such that $R_i\xrightarrow{\sot} Id$ (and since $X$ is assumed to be separable, one can in fact take a \emph{sequence} $(R_i)$ with this property). See \mbox{e.g.} \cite{LT} for more on this notion.

\bpr\label{invertible} Let $X$ be a Banach space which has the \emph{MAP}. Then, the set of all $T\in\bbx$ such that $0\in\sigma_{ap}(T)$ is a dense $\gd$ subset of $(\bbx,\emph{\sot})$. In particular, a typical $T\in(\bbx,\emph{\sot})$ is non-invertible.
\epr
\bpf Set $\g:=\{ T\in\bbx;\; 0\in\sigma_{ap}(T)\}$. Then, 
\[ T\in\g\iff \forall K\in\N\;\exists x\in S_X\;:\; \Vert Tx\Vert<1/K,\]
so $\g$ is a $G_\delta$ subset of $(\bbx,{\sot})$. Moreover, $\g$ contains all finite rank operators in $\bbx$, and the latter are \sot$\,$-$\,$dense in $\bbx$ because $X$ has the MAP.%; so $\g'$ is comeager in $\bbx$ and hence $\mathfrak H$ is comeager.
\epf
In the particular case of $\ell_{p}$ and $c_{0}$, we have the following proposition, which provides some spectral properties of \sot$\,$-$\,$typical contractions on these spaces:
\begin{proposition}\label{aameliorer}
 Assume that $X=\ell_{p}$, $1\le p<\infty$ or that $X=c_{0}$. A typical $T\in(\bbx,\emph{\sot})$ has the following properties: $T-\lambda $ has dense range for every $\lambda \in\C$, and $\sigma (T)=\sigma _{{ap}}(T)=\ba{\,\D}$.
\end{proposition}
\bpf Let us set 
\[ \g_1:= \bigl\{ T\in\bbx;\; \hbox{$T-\lambda$ has dense range for every $\lambda\in\C$}\bigr\},\]
and 
\[ \g_2:=\bigl\{ T\in\bbx;\; \sigma(T)=\sigma_{ap}(T)=\overline{\,\D}\bigr\}.\]%Part (a) does not require any assumption on $X$ (except separability). 

\noindent By the Baire Category Theorem, it is enough to show that $\g_1$ and $\g_2$ are both dense $\gd$ subsets of $(\bbx,\sot)$.

\smallskip
Let $Z$ be a countable dense subset of $X$. If $T\in\bbx$, then
\[ T\not\in\g_1\iff \exists\lambda\in \C\;\exists K\in\N\; \exists z\in Z\;:\; \forall z'\in Z\;:\; \Vert (T-\lambda) z'-z\Vert \geq1/K.\]
This shows that $\bbx\setminus\g_1$ is the projection of an $\fs$ subset of $\C\times \bbx$, where $\bbx$ is endowed with the topology \sot. Since $\C$ is $\sigma$-compact, it follows that $\bbx\setminus \g_1$ is $F_\sigma$ in $(\bbx,{\sot})$, \mbox{\it i.e.} $\g_1$ is $\gd$.

Similarly, if $T\in\bbx$ then (since $\sigma_{ap}(T)\subseteq\sigma(T)\subseteq\overline{\,\D}$), we have that $T\not\in\g_2$ if and only if $T-\lambda$ is bounded below for some $\lambda\in\overline{\,\D}$, \mbox{\it i.e.}
\[ T\not\in\g_2\iff \exists\lambda\in \overline{\,\D}\;\exists K\in\N\;\forall z\in Z\;:\; \Vert (T-\lambda) z\Vert\geq1/K\,\Vert z\Vert;\]
and since $\overline{\,\D}$ is compact, it follows that $\g_2$ is a $\gd$ subset of $(\bbx,{\sot})$.

\smallskip A straightforward adaptation of the proof of \cite[Proposition 2.3]{GMM} or \cite[Proposition 2.23]{BM} shows that for any $M>1$, the set of all hypercyclic operators in $\bmx$ is a dense $\gd$ subset of $(\bmx,\sot)$. In particular, an \sot$\,$-$\,$typical $R\in\bmx$ is such that $a R +b$ has dense range for any $a,b\in\C$ with $(a,b)\neq (0,0)$. Since the map $T\mapsto M T$ is a homeomorphism from $(\bbx,{\sot})$ onto $(\bmx,{\sot})$, it follows that $\g_1$ is dense in $(\bbx,{\sot})$.

\smallskip Now, let us show that $\g_2$ is dense in $\bbx$. Recall that $(e_j)_{j\geq 0}$ is the canonical basis of $X$, that for any $N\geq 0$,  $P_N:X\to [e_0,\dots ,e_N]$  is the canonical projection map onto $E_N= [e_0,\dots ,e_N]$, and that $F_N=[e_j;\; j>N]$. 
Choose an operator $B_N\in {\mathcal B}_1( F_N)$ such that $B_N-\lambda$ is not bounded below for any $\lambda\in\overline{\,\D}$ (for example, the backward shift built on $(e_j)_{j\geq N+1}$ will do). If $A\in\bbx$ is arbitrary, then $T_N:=P_NA_{| E_N} \oplus B_N$ satisfies $\Vert T_N\Vert\leq 1$ and is such that $T-\lambda$ is not bounded below for any $\lambda\in\overline{\,\D}$; that is, $T_N\in\g_2$. Since $T_N\xrightarrow{\sot} A$ as $N\to\infty$, this shows that $\g_2$ is indeed dense in $(\bbx,\sot)$.
\epf

Proposition \ref{aameliorer} says in particular that a typical $T\in\mathcal B_1(\ell_p)$ has the largest possible spectrum. However, it is worth mentioning that there are also lots of operators $T\in\mathcal B_1(\ell_p)$ whose spectrum is rather small:
\begin{proposition}\label{Operateurs spectre T}
 Let $X=\ell_{p}$, $1\le p<\infty$. 
 The set of operators $T\in\bbx$ such that $\sigma (T)\subseteq \T$ is dense in $(\bbx,\emph{\sot})$.
\end{proposition}

\smallskip Perhaps surprisingly, the proof of this result is not straightforward. For technical reasons, it will be convenient to identify the space $X$ with $Y:=\ell_{p}(\Z)$ endowed with its canonical basis $(f_{k})_{k\in\Z}$. 
\par\smallskip
For each $N\ge 0$, let $Y_{N}:=[\,f_{k}\;;\;\vert k\vert \le N\,]$. %, and let $\Pi_N$ be the canonical projection of $Y$ onto $Y_N$. 
To any operator $A\in{\mathcal{B}}(Y_{N})$ and any bounded sequence of positive numbers $\omega =(\omega _{k})_{k\in\Z}$ -- such a sequence $\omega$ will be called a \emph{weight sequence} -- we associate the operator $S_{A,\omega }$ on $Y$ defined by 
\[
S_{A,\omega}f_{k}:=
\begin{cases}
 Af_{k}+\omega _{k-(2N+1)}f_{k-(2N+1)}\quad &\textrm{if}\ \vert k\vert \le N, \\
 \omega _{k-(2N+1)}f_{k-(2N+1)}&\textrm{if}\ \vert k\vert >N.
\end{cases}
\]
Some of the properties of these operators, related to linear dynamics, can be found in \cite[Proposition 2.14 and Remark 2.15]
{GMM}. 
\par\smallskip In the next two lemmas, the integer $N$ and the operator $A\in{\mathcal B}(Y_N)$ are fixed. The first lemma gives an estimate on the norm of $S_{A,\omega }$.
%Each operator $S_{A,\omega }$ is bounded on $Y$, and we have the following estimate on its norm:
\begin{lemma}\label{Estimation norme lp}
 For any $\varepsilon >0$, there exists $\delta >0$ depending only on $\varepsilon$ and $\Vert A\Vert$ such that, for any weight sequence $\omega $ satisfying $\sup\limits_{-(3N+1)\le k\le N}\omega _{k}<\delta $, one has
 \[
\Vert S_{A,\omega }\Vert\le\max\,\biggl( (\Vert A\Vert^{p}+\varepsilon ^{p})^{1/p},\sup_{k\not\in[-(3N+1),N]}\omega _{k} \biggr). 
\]
\end{lemma}
\begin{proof}[Proof of Lemma \ref{Estimation norme lp}]
For any set of integers $I$ we denote by $P_I$ the canonical projection of $Y=\ell_p(\Z)$ onto $[f_k\,; \; k\in I\,]$. For every $y=\sum_{k\in\Z}y_{k}f_{k}$, we have
\begin{align*}
\Vert S_{A,\omega }\,y\Vert^{p}=\biggl\Vert A\,P _{[-N,N]}y+\sum_{k=N+1}^{3N+1}\omega _{k-(2N+1)}&y_{k}f_{k-(2N+1)}\biggr\Vert^{p}\\
&+\sum_{k\not\in[N+1,3N+1]}\omega _{k-(2N+1)}^{p}\vert y_{k}\vert^{p}
\end{align*}
Let $c_{1},c_{2}\in\R_+$ be such that  $(a+b)^{p}\le a^{p}+c_{1}b^{p}+c_{2}a^{p-1}b$ for every $a,b\ge 0$. Using H\"older's inequality, we se that for any $u,v\in\ell_p$, we have 
\[ \Vert u+v\Vert^p\leq \Vert u\Vert^p+c_1\, \Vert v\Vert^p + c_2 \,\Vert u\Vert^{p-1}\,\Vert v\Vert.\]
So we get
\begin{align*}
 \bigl\Vert \,S_{A,\omega }\,y\,\bigr\Vert ^{p}\le\bigl\Vert\,&A\,\bigr\Vert^{p}\,\bigl\Vert\,P _{[-N,N]}y\,\bigr\Vert^{p}+c_{1}\sup_{N+1\le k\le 3N+1}\omega _{k-(2N+1)}^{p}\,\bigl\Vert\,P _{[N+1,3N+1]}y\,\bigr\Vert^{p}\\
 &+c_{2}\,\bigl\Vert\,A\,\bigr\Vert^{p-1}\bigl\Vert\,P _{[-N,N]}\,y\,\bigr\Vert^{p-1}\sup_{N+1\le k\le 3N+1}\omega _{k-(2N+1)}\,\bigl\Vert\,P _{[N+1,3N+1]}\,y\,\bigr\Vert\\
 &+\sup_{k\in[-N,N]}\omega _{k-(2N+1)}^{p}\,\bigl\Vert\,P _{[-N,N]}\,y\,\bigr\Vert^{p}\\
 &+\sup_{k\not\in[-N,3N+1]}\omega _{k-(2N+1)}^{p}\,\bigl\Vert\,(I-P _{[-N,3N+1]})\,y\,\bigr\Vert^{p}.
\end{align*}
Since $\sup_{-N\le k\le 3N+1}\omega _{k-(2N+1)}=\sup_{-(3N+1)\le k\le N}\omega _{k}<\delta $ and since
\[\bigl\Vert\,P _{[-N,N]}\,y\,\bigr\Vert^{p-1}\,\bigl\Vert\,P _{[N+1,3N+1]}\,y\,\bigr\Vert\le \bigl\Vert\,P _{[-N,3N+1]}\,y\,\bigr\Vert^{p},\] this yields that

\begin{align*}
 \bigl\Vert \,S_{A,\omega }\,y\,\bigr\Vert ^{p}&\le \Bigl (\bigl\Vert \,A\,\bigr\Vert ^{p}+\delta ^{p} \Bigr)\,\bigl\Vert \,P _{[-N,N]}\,y\,\bigr\Vert ^{p}+
 \bigl ( c_{1}\delta ^{p}+c_{2}\,\bigl\Vert\,A\,\bigr\Vert^{p-1}\delta \bigr)\,\bigl\Vert \,P _{[-N,3N+1]}\,y\,\bigr\Vert ^{p} \\
  &\quad\quad +\sup_{k\not\in[-(3N+1),N]}\omega _{k}^{p}\;\bigl\Vert \,(I-P _{[-N,3N+1]})\,y\,\bigr\Vert ^{p}\\
  &\le\Bigl ( \bigl\Vert \,A\,\bigr\Vert ^{p}+(c_{1}+1)\delta ^{p}+c_{2}\,\bigl\Vert\,A\,\bigr\Vert^{p-1}\delta \Bigr)\,\bigl\Vert \,P _{[-N,3N+1]}y\,\bigr\Vert ^{p}\\
  &\quad \quad +\sup_{k\not\in[-(3N+1),N]}\omega _{k}^{p}\;\bigl\Vert \,(I-P _{[-N,3N+1]})\,y\,\bigr\Vert ^{p}.
\end{align*}
Choosing $\delta $ such that $(c_{1}+1)\delta ^{p}+c_{2}\,\bigl\Vert\,A\,\bigr\Vert^{p-1}\delta<\varepsilon ^{p}$, we obtain that
\[
\bigl\Vert \,S_{A,\omega }\,\bigr\Vert \le\max\biggl ( \bigl (\, \bigl\Vert \,A\,\bigr\Vert ^{p}+\varepsilon ^{p}\bigr)^{1/p},\sup_{k\not\in[-(3N+1),N]}\omega _{k}  \biggr),
\]
as claimed.
\end{proof}
\par\medskip
The second lemma gives a description of the point spectrum of these operators $S_{A,\omega }$. For any weight sequence $\omega $ and any $\lambda $ in $\C$, let us define two subsets $\Lambda _{\omega ,\lambda }^{-}$ and $\Lambda _{\omega ,\lambda }^{+}$ of $[-N,N]$ as follows:
\begin{align*}
 \Lambda _{\omega ,\lambda }^{-}&:=\biggl \{ k\in[-N,N]\;;\;\sum_{i\ge 1}
 \ \Bigl \vert\, \dfrac{\omega _{k}\cdots\omega _{k-(i-1)(2N+1)}}{\lambda ^{i}}\,\Bigr\vert^{p}<\infty \biggr\} \ \quad\text{with}\quad \Lambda _{\omega ,0}^{-}=\emptyset, \\
  \Lambda _{\omega ,\lambda }^{+}&:=\biggl \{ k\in[-N,N]\;;\;\sum_{i\ge 1}
 \ \Bigl \vert\, \dfrac{\lambda ^{i}}{\omega _{k+(i-1)(2N+1)}\cdots\omega _{k}}\,\Bigr\vert^{p}<\infty \biggr\} .
\end{align*}
\begin{lemma}\label{Estimation lambda}
 Let $\lambda \in\C$. Then $\lambda $ is an eigenvalue of $S_{A,\omega }$ if and only if there exists  a non-zero  vector $u\in Y_{N}$ such that $\emph{supp}(u)\subseteq\Lambda _{\omega ,\lambda }^{-}$ and 
 $\emph{supp}\bigl((A-\lambda) u \bigr)\subseteq\Lambda _{\omega ,\lambda }^{+}$. 
 In particular, if $\Lambda _{\omega ,\lambda }^{-}=\emptyset$ then $\lambda $ is not an eigenvalue of $S_{A,\omega }$.  %either $\Lambda _{\omega ,\lambda }^{-}$ or $\Lambda _{\omega ,\lambda }^{+}$ is empty, $\lambda $ is not an eigenvalue of $S_{A,\omega }$.
\end{lemma}

\noindent Here, we denote by $\textrm{supp}(y)$ the \emph{support} of a vector $y=\sum_{k\in\Z}y_{k}f_{k}$, \mbox{\it i.e.} the set of all $k\in\Z$ such that $y_{k}\neq 0$.

\begin{proof}[Proof of Lemma \ref{Estimation lambda}]
The number $\lambda $ is an eigenvalue of $S_{A,\omega }$ if and only if there exists a non-zero vector $y\in Y$ such that 
\begin{align}
 &(A-\lambda )\,P _{[-N,N]}\,y+\sum_{\vert k\vert\le N}\omega _{k}y_{k+(2N+1)}f_{k}=0\tag{1}\label{Eq1}\\
 &\omega _{k+(2N+1)}y_{k+(2N+1)}=\lambda \,y_{k}\qquad\qquad\textrm{for every}\ k\ge N+1\tag{2}\label{Eq2}\\
 &\omega _{-k+(2N+1)}y_{-k+(2N+1)}=\lambda \,y_{-k}\qquad\,\textrm{for every}\ k\ge N+1\tag{3}\label{Eq3}
\end{align}
It follows from these expressions that $\lambda =0$ is never an eigenvalue of $S_{A,\omega }$. We suppose for the rest of the proof that $\lambda \neq 0$. We deduce from equations (\ref{Eq1}), (\ref{Eq2}), and (\ref{Eq3}) that
\begin{align*}
 y_{k+(2N+1)}&=-\dfrac{1}{\omega _{k}}\ \pss{f_{k}^{*}}{\,(A-\lambda )\,P _{[-N,N]}y }\qquad \textrm{for every}\ \vert k\vert \le N, 
 \intertext{(given $P _{[-N,N]}y$, these equations determine the values of $y_{j}$ for $j\in[N+1,3N+1]$).}
 y_{k+i(2N+1)}&=\dfrac{\lambda^{i}}{\omega _{k+i(2N+1)}\dots\omega _{k+(2N+1)}}\,y_{k}\qquad \textrm{for every}\ k\ge N+1\ \textrm{and every}\ i\ge 1,
 \intertext{(the values of $y_{j}$ for $j\in[N+1,3N+1]$ determine the values of $y_{j}$ for $j>3N+1$)}
 \intertext{and}
 y_{k-i(2N+1)}&=\dfrac{\omega _{k}\dots \omega _{k-(i-1)(2N+1)}}{\lambda ^{i}}\,y_{k}\quad \textrm{for every}\ k\le N\ \textrm{and every}\ i\ge 1
 \end{align*}
 (the values of $y_{j}$ for $j\in[-N,N]$ determine the values of $y_{j}$ for $j<-N$).
 \par\medskip
Hence, if $u$ is a vector of $Y_{N}$, there exists a vector $y\in Y$ with $P_{[-N,N]}\,y=u$ and $y\in\ker (S_{A,\omega }-\lambda )$ if and only if the series
\[
\sum_{i\ge 1}
 \ \Bigl \vert \, \dfrac{\omega _{k}\cdots\omega _{k-(i-1)(2N+1)}}{\lambda ^{i}}\,\Bigr\vert ^{p}\cdot\vert u_{k}\vert ^{p}\quad \textrm{and}\quad 
 \sum_{i\ge 1}
 \ \Bigl \vert \, \dfrac{\lambda ^{i}}{\omega _{k+(i-1)(2N+1)}\cdots\omega _{k}}\,\Bigr\vert ^{p}\cdot\big\vert \pss{f_{k}^{*}}{(A-\lambda )u} \bigr\vert^{p}
\]
are convergent for every $k\in[-N,N]$. This is equivalent to the conditions that, for every $k\in[-N,N]$, either $u_{k}=0$ or $k\in\Lambda _{\omega ,\lambda }^{-}$, and either $\pss{f_{k}^{*}}{(A-\lambda )u}=0$ or $k\in\Lambda _{\omega ,\lambda }^{+}$. Lemma \ref{Estimation lambda} follows.
\end{proof}

\par\smallskip With the two lemmas above at our disposal, we can now proceed to the 
\begin{proof}[Proof of Proposition \ref{Operateurs spectre T}.] Let ${\mathcal U}$ be a non-empty open set in $({\mathcal B}_1(Y), \sot)$. We want to show that ${\mathcal U}$ contains an operator $T$ such that $\sigma(T)\subseteq\T$.
\par\smallskip Choose an integer $N$, an operator $A\in{\mathcal B}(Y_N)$ with $\Vert A\Vert <1$, and $\varepsilon >0$, such that any operator $T\in{\mathcal B}_1(Y)$  satisfying $\Vert (T-A) f_k\Vert<\varepsilon$ for all $k\in [-N,N]$ belongs to the open set ${\mathcal U}$. Then, choose a weight sequence $\omega=(\omega_k)_{k\in\Z}$ such that $\omega_k$ is extremely small if $-(3N+1)\leq k\leq N$, and otherwise $\omega_k=1$. By Lemma \ref{Estimation norme lp} and since $\Vert A\Vert <1$, the operator $T:= S_{A,\omega}$ satisfies $\Vert T\Vert\leq 1$, provided that the weights $\omega_k$ are sufficiently small for $-(3N+1)\leq k\leq N$. Moreover, $Tf_k$ is very close to $Af_k$ for $k\in [-N,N]$. Hence $T\in {\mathcal U}$.

It remains to show that $\sigma( T)\subseteq \T$. First, observe that since all but finitely many weights $\omega_k$ are equal to $1$, the set $\Lambda_{\omega,\lambda}^-$ is obviously empty for any $\lambda\in\D$. By Lemma \ref{Estimation lambda}, it follows that $T$ has no eigenvalue in $\D$. Let us show that, in fact, $\sigma(T)\cap\D=\emptyset$. Let $\lambda\in\D$ be arbitrary. Since all but finitely many weights $\omega_k$ are equal to $1$, we see that the operator $T$ is a finite rank perturbation of $S^{2N+1}$, where $S$ is the unweighted bilateral shift on $Y=\ell_p(\Z)$. Hence, $T-\lambda$ is a finite rank perturbation of $S^{2N+1}-\lambda$. Now, $S^{2N+1}$ is a surjective isometry, so $S^{2N+1}-\lambda$ is invertible because $\vert \lambda\vert<1$. Therefore, $T-\lambda$ is a Fredholm operator with 
index $0$. But $T-\lambda$ is also one-to-one since $T$ has no eigenvalue in $\D$. So $T-\lambda$ is invertible, \mbox{\it i.e.} $\lambda\not\in\sigma(T)$. This terminates the proof of Proposition \ref{Operateurs spectre T}.
\end{proof}

\section{\sot$\,$-$\,$typical contractions on $\ell_{1}$}\label{Section3}
We consider in this section the case of the space $X=\ell_{1}$, which turns out to behave much like $\ell_{2}$ as far as the properties considered in Corollary \ref{vrac} are concerned:
\bth\label{l1} Let $X=\ell_1$. A typical $T\in(\bbx, \emph{\sot})$ has the following properties.
\bea
\item\label{isoml1} $T^*$ is a non-surjective isometry.
\item\label{surjl1} $T-\lambda$ is surjective for every $\lambda\in\D$.
\item\label{eigenl1} Every $\lambda\in\D$ is an eigenvalue of $T$ with infinite multiplicity.
\item\label{spectruml1} $\sigma(T)=\sigma_{ap}(T)=\overline{\,\D}$.
\item\label{orbit0l1} $\Vert T^nx\Vert\to 0$ for every $x\in X$. 
\ee
\eth
\bpf We already know that (\ref{orbit0l1}) holds true on any Banach space $X$; and (\ref{spectruml1}) follows from (\ref{eigenl1}) because $\sigma_{ap}(T)$ is a closed set containing $\sigma_p(T)$. So we just need to prove (\ref{isoml1}), (\ref{surjl1}) and (\ref{eigenl1}).

\smallskip
(\ref{isoml1}) Since $X=\ell_1$ has the MAP, we know by Proposition \ref{invertible} that a typical $T\in(\bbx,\sot)$ is such that $T^*$ is non-invertible; so it is enough to show that a typical $T\in(\bbx,\sot)$ is such that $T^*$ is an isometry. Let us denote by ${\mathcal I}_*$ the set of all $T\in\bbx$ such that $T^*$ is an isometry; we are going to show that ${\mathcal I}_*$ is a dense $\gd$ subset of $(\bbx,\sot)$.% and by $\mathfrak H$ the set of all non-invertible $T\in\bbx$. We have to show that $\g\cap\mathfrak H$ is comeager in $(\bbx,\sot)$; and by the Baire category Theorem, it is enough to check separately that $\g$ and $\mathfrak H$ are comeager.

The proof that ${\mathcal I}_*$ is $\gd$ works on any separable Banach space $X$. This relies on the following fact.
\begin{fact}\label{lsc} The map $(T,x^*)\mapsto \Vert T^*x^*\Vert$ is lower-semicontinuous on $(\bbx,\sot)\times (B_{X^*},w^*)$.
\end{fact}
\begin{proof}[Proof of Fact \ref{lsc}] This  is  clear  since $\Vert T^*x^*\Vert=\sup_{x\in B_X} \vert \langle x^*, Tx\rangle\vert$ and each map 
\[
(T,x^*)\mapsto \langle x^*, Tx\rangle,\quad x\in  X
\] 
is continuous on $(\bbx,\sot)\times (B_{X^*},w^*)$. 
\end{proof}

If $T\in\bbx$, then 
\[ T\not\in{\mathcal I}_*\iff \exists x^*\in B_{X^*}\;\exists \alpha\in\Q\;:\; \Vert T^*x^*\Vert\leq \alpha <\Vert x^*\Vert.\]
The condition $\Vert T^*x^*\Vert\leq \alpha$ defines a closed subset of $(\bbx,\sot)\times (B_{X^*},w^*)$ by Fact~\ref{lsc}; and the condition $\Vert x^*\Vert>\alpha$ defines an open subset of $(B_{X^*},w^*)$, and hence an $F_\sigma$ set in \((B_{X^*},w^*)\) because $(B_{X^*},w^*)$ is metrizable. So we see that $\bbx\setminus{\mathcal I}_*$ is the projection along $B_{X^*}$ of an $F_\sigma$ subset of $(\bbx,\sot)\times (B_{X^*},w^*)$. Since $(B_{X^*}, w^*)$ is compact, this shows that ${\mathcal I}_*$ is $\gd$.

\smallskip Now, let us show that ${\mathcal I}_*$ is \sot$\,$-$\,$dense in $\bbx$ when $X=\ell_1$. Given an arbitrary $A\in\bbx$, set $T_N:= P_NAP_N+ B_N (I-P_N)$, where $B_N:F_N\to X$ is the operator defined by $B_N e_{N+1+k}=e_k$ for every $k\geq 0$. Then $\Vert T_Ne_j\Vert\leq 1$ for every $j\in\Z_+$, and hence $\Vert T_N\Vert\leq 1$ since we are working on $X=\ell_1$. Moreover, if $x^*\in X^*$ is arbitrary, then $\Vert  T_N^*x^*\Vert\geq \vert \langle x^*, T_Ne_{N+1+k}\rangle\vert=\vert\langle x^*,e_k\rangle\vert$ for every $k\geq 0$, and hence $\Vert T^*x^*\Vert \geq \Vert x^*\Vert$ since $X^*=\ell_\infty$. So we have 
$T_N\in{\mathcal I}_*$ for all $N\geq 0$; and since $T_N\xrightarrow{\sot} A$ as $N\to\infty$, we conclude that ${\mathcal I}_*$ is indeed dense in $(\bbx,\sot)$.

\smallskip
To prove (\ref{surjl1}) and (\ref{eigenl1}), observe that if $T\in\bbx$ is such that $T^*$ is an isometry, then $T^*-\lambda$ is bounded below for every $\lambda\in\D$. 
%, and hence $T^*-\lambda$ has closed range. 
By (\ref{isoml1}) and since the adjoint of an operator $R$ is bounded below if and only if $R$ is surjective, %has closed range if and only if $R^*$ has closed range, 
it follows that a typical $T\in(\bbx,\sot)$ is such that $T-\lambda$ is surjective for every $\lambda\in\D$; which is (\ref{surjl1}). Since a typical $T\in(\bbx,\sot)$ is also such that $T-\lambda$ is non-invertible for every $\lambda\in\D$ (by Proposition \ref{aameliorer}), it follows that a typical $T\in\bbx$ is such that $T-\lambda $ is not one-to-one for every $\lambda \in\D$.  However, in order to obtain (\ref{eigenl1}), we still have to prove that for a typical $T\in (\bbx,\sot)$, the multiplicity of every $\lambda \in\D$ as an eigenvalue of $T$ is infinite. The proof of this statement relies on the following lemma:
\begin{lemma}\label{Infini}
 Let $X=\ell_{1}$. Let $\ba{\alpha }=(\alpha _{i})_{i\ge 0}$ be a sequence of positive numbers with $\alpha _{0}=\alpha _{1}=1$ and $\sum_{i\ge 0}\alpha _{i}<\infty$. Let also $\ba{N}=(N_{i})_{i\ge 0}$ be a strictly increasing sequence of nonnegative integers with $N_{0}=0$. For every $q\ge 0$, let ${\mathcal{D}}_{q}(\ba{\alpha },\ba{N})$ be the subset of $\bbx$ defined as follows:
 \begin{align*}
  T\in{\mathcal{D}}_{q}(\ba{\alpha },\ba{N})\iff\forall\,0\le i_{0}<&\dots<i_{l}\le q\quad \textrm{with}\quad \biggl\Vert T\biggl ( \sum_{j=0}^{l}\alpha _{j}e_{N_{i_{j}}}\biggr) \biggr\Vert\le\alpha _{l+1} \\
  &\exists\,i>q\;:\;\biggl\Vert T\biggl ( \sum_{j=0}^{l}\alpha _{j}e_{N_{i_{j}}}+\alpha _{l+1}e_{N_{i}}\biggr) \biggr\Vert<\alpha _{l+2}.
 \end{align*}
Then ${\mathcal{D}}_{q}(\ba{\alpha },\ba{N})$ is a dense open subset of $(\bbx,\emph{\sot})$.
\end{lemma}
\begin{proof}[Proof of Lemma \ref{Infini}]
 The set ${\mathcal{D}}_{q}(\ba{\alpha },\ba{N})$ is clearly \sot$\,$-$\,$open in $\bbx$. Let 
 $T_{0}\in\bbx$, and let ${\mathcal{U}}$ be an \sot$\,$-$\,$open neighborhood of $T_{0}$ in $\bbx$. We choose $r\ge q$ such that any operator $T\in\bbx$ satisfying $Te_{n}=T_{0}e_{n}$ for every $0\le n\le N_{r}$ belongs to ${\mathcal{U}}$. Next, we enumerate the (finite) set $\Sigma _{q}$ consisting of all sequences $(i_{0},\dots,i_{l})$ with
 \[
i_{0}<\dots<i_{l}\le q\quad \textrm{and}\quad \biggl\Vert T_{0}\biggl ( \sum_{j=0}^{l}\alpha _{j}e_{N_{i_{j}}}\biggr) \biggr\Vert\le\alpha _{l+1} 
\]
as $\Sigma _{q}=\{\sigma _{1},\dots,\sigma _{s}\}$, and write for each $k=1,\dots,s$ the sequence $\sigma _{k}$ as $(i_{k,0},\dots,i_{k,l_{k}})$ for a certain integer $l_{k}\ge 0$. We have thus
\[
\biggl\Vert T_{0}\biggl ( \sum_{j=0}^{l_{k}}\alpha _{j}e_{N_{i_{k,j}}}\biggr) \biggr\Vert\le\alpha _{l_{k	}+1}\quad \textrm{for every}\quad k=1,\dots,s. 
\]
Let now $i_{1},\dots,i_{s}$ be integers with $r<i_{1}<\dots<i_{s}$, and define an operator $T$ on $X$ by setting
\[
Te_{n}:=
\begin{cases}
 T_{0}e_{n}&\textrm{for every}\ 0\le n\le N_{r},\\
 -\dfrac{1}{\alpha _{l_{k}+1}}\,T_{0}\ds\biggl (\sum_{j=0}^{l_{k}}\alpha _{j}e_{N_{i_{k,j}}} \biggr )&\textrm{if}\ n=N_{i_{k}}\ \textrm{for some}\ 1\le k\le s,\\
 0&\textrm{in all other cases.}
\end{cases}
\]
Since $\Vert Te_{n}\Vert\le 1$ for every $n\ge 0$ and $X=\ell_{1}$, we see that $\Vert T\Vert\le 1$. Also, 
\[
\textrm{for every}\ k=1,\dots,s,\quad\quad  T\biggl ( \sum_{j=0}^{l_{k}}\alpha _{j}e_{N_{i_{k,j}}}\biggr) =T_{0}\biggl ( \sum_{j=0}^{l_{k}}\alpha _{j}e_{N_{i_{k,j}}}\biggr) 
\]
since $i_{k,j}\le q\le r$ and $T$ and $T_{0}$ coincide on $[\,e_{n}\;;\;0\le n\le N_{r}\,]$. Hence 
\[
T\biggl ( \sum_{j=0}^{l_{k}}\alpha _{j}e_{N_{i_{k,j}}}+\alpha _{l_{k}+1}e_{N_{i_{k}}}\biggr)=0\quad \quad \textrm{for every}\ k=1,\dots,s. 
\]
This implies that the operator $T$ belongs to ${\mathcal{D}}_{q}(\ba{\alpha },\ba{N})$. We have thus proved that ${\mathcal{D}}_{q}(\ba{\alpha },\ba{N})$ is \sot$\,$-$\,$dense in $\bbx$.
\end{proof}
The reason why these sets ${\mathcal{D}}_{q}(\ba{\alpha },\ba{N})$ are introduced is given in the next lemma:
\begin{lemma}\label{Vecteur du noyau}
 Let $\ba{\alpha }$ and $\ba{N}$ satisfy the assumptions of \emph{Lemma \ref{Infini}}. If $T\in\bbx$ belongs to the set $\bigcap_{q\ge 0}{\mathcal{D}}_{q}(\ba{\alpha },\ba{N})$, there exists a non-zero vector $x\in[\,e_{N_{i}}\;;\;i\ge 0\,]$ such that $Tx=0$.
\end{lemma}
\begin{proof}[Proof of Lemma \ref{Vecteur du noyau}]
 Since $\alpha _{0}=\alpha _{1}=1$ and $N_{0}=0$, we have $\Vert T(\alpha _{0}e_{N_{0}})\Vert\le\alpha _{1}$. Since $T\in {\mathcal{D}}_{q}(\ba{\alpha },\ba{N})$, setting $i_{0}=0$, we obtain that there exists $i_{1}>0$ such that 
 \[\bigl\Vert T(\alpha _{0}e_{N_{i_{0}}}+\alpha _{1}e_{N_{i_{1}}})\bigr\Vert<\alpha _{2}.\]
 An induction argument along the same line then shows the existence of a strictly increasing sequence of nonnegative integers $(i_{j})_{j\ge 0}$ such that
 \[
\biggl\Vert T\biggl (\sum_{j=0}^{l}\alpha _{j}e_{N_{i_{j}}} \biggr)\biggr\Vert <\alpha _{l+1}\quad  \textrm{for every}\quad l\ge 0.
\]
The vector $x=\sum_{j\ge 0}\alpha _{j}e_{N_{i_{j}}}$ belongs to $[\,e_{N_{i}}\;;\;i\ge 0\,]$, is non-zero since $\alpha _{j}>0$ for every $j\ge 0$, and satisfies $Tx=0$.
\end{proof}
Thanks to Lemma \ref{Infini} and \ref{Vecteur du noyau}, it is now not difficult to prove that an \sot$\,$-$\,$typical operator $T\in\bbx$ satisfies $\dim\ker T=\infty$. Indeed, let $(\ba{N}_{n})_{n\ge 1}$ be an infinite sequence consisting of strictly  increasing sequences $\ba{N}_{n}=(N_{i,n})_{i\ge 0}$ of nonnegative integers with $N_{0,n}=0$, with the property that the sets $\{N_{i,n}\;;\;i\ge 1\}$, $n\ge 1$, are pairwise disjoint. Let also $\ba{\alpha }=(\alpha _{i})_{i\ge 0}$ be any sequence of positive numbers with $\alpha _{0}=\alpha _{1}=1$ and $\sum_{i\ge 0}\alpha _{i}<\infty$. By Lemma \ref{Infini}, the set
\[
{\mathcal{G}}:=\bigcap_{n\ge 1}\bigcap_{q\ge 0}{\mathcal{D}}_{q}(\ba{\alpha },\ba{N}_{n})
\]
is a dense $G_{\delta }$ subset of $(\bbx,\sot)$. The set ${\mathcal{U}}_{0}\subseteq\bbx$ consisting of operators $T\in\bbx$ such that $Te_{0}\neq 0$ being open and dense in $(\bbx,\sot)$, ${\mathcal{G}}\cap{\mathcal{U}}_{0}$ is also a dense $G_{\delta }$ subset of $(\bbx,\sot)$. By Lemma \ref{Vecteur du noyau}, if $T$ belongs to ${\mathcal{G}}\cap{\mathcal{U}}_{0}$, there exists, for each $n\ge 1$, a non-zero vector $x_{n}\in[\,e_{N_{i,n}}\;;\;i\ge 0\,]$ such that $Tx_{n}=0$. Since $Te_{0}\neq 0$, 
$x_{n}$ is not colinear to $e_{0}$, and the assumption that the sets $\{N_{i,n},\;\,i\ge 1\}$, $n\ge 1$, are pairwise disjoint implies that the vectors $x_{n}$, $n\ge 1$, are linearly independent. So $\dim\ker T=\infty$.
\par\medskip
At this point, we know that a typical $T\in\bbx$ is such that $T-\lambda$ is surjective for every $\lambda\in\D$ and $\dim\ker(T)=\infty$. These two properties imply that $\dim\ker(T-\lambda)=\infty$ for every $\lambda\in\D$. Indeed, $T-\lambda$ is semi-Fredholm for every $\lambda\in\D$ since it is surjective. By the continuity of the Fredholm index, ${\rm ind}(T-\lambda)$ does not depend on $\lambda\in\D$. Now, we have ${\rm ind}(T)=\infty$ since $T$ is surjective and $\dim\ker(T)=\infty$. So ${\rm ind}(T-\lambda)=\infty$ for every $\lambda\in\D$; and since $T-\lambda$ is surjective, it follows that $\dim\ker(T-\lambda)=\infty$. (See \mbox{e.g.} \cite[p.\,229-244]{K} for a treatment of the Fredholm index for operators on Banach spaces.) This terminates the proof of assertion (\ref{eigenl1}), and hence the proof of Theorem \ref{l1}.
%It now remains to deduce from this that if $T\in \mathfrak G\cap\mathfrak U_0$, then $\dim\ker(T-\lambda )=\infty$ for every $\lambda \in\D$. Since $T$ is surjective and $\dim\ker T=\infty$, $T$ is semi-Fredholm with index $\textrm{ind}\,(T)=\infty$. For every $\lambda \in\D$, the operator $T-\lambda $ is semi-Fredholm (since it is surjective), and the continuity of the index function on the set of semi-Fredholm operators on $X$ implies that $\textrm{ind}\,(T-\lambda )=\textrm{ind}\,(T)=\infty$ for every $\lambda \in\D$. As $T-\lambda $ is surjective, $\textrm{ind}\,(T-\lambda )=\dim\ker(T-\lambda )$, and thus $\dim\ker(T-\lambda )=\infty$. 
\epf
As an immediate consequence of property (\ref{eigenl1}) in Theorem \ref{l1}, we obtain:
\begin{corollary}\label{Cor th l1}
 Let $X=\ell_{1}$. A typical $T\in(\bbx,\emph{\sot})$ has a non-trivial invariant subspace.
\end{corollary}

%%%%%%%%%%%%%%%%%%%%%%%%%%%%%%%%%%%%%%%%%%%%%%%%%%%%%%%%%%%%%%%%%%%%%%%%%%%%%%%%%%%%%
%%%%%%% Fin de cette insertion %%%%%%%%%%%%%%%%%%%%%%%%%%%%%%%%%%%%%%%%%%%%%%%%%%%%%%
%%%%%%%%%%%%%%%%%%%%%%%%%%%%%%%%%%%%%%%%%%%%%%%%%%%%%%%%%%%%%%%%%%%%%%%%%%%%%%%%%%%%%

\section{\sot$\,$-$\,$typical contractions on $\ell_{p}$ and $c_{0}$}\label{Section4}

We have proved in the previous sections that a typical contraction $T\in(\bbx,\sot)$ for $X=\ell_{2}$ or $X=\ell_{1}$ has the property that $T^{*}$ is an isometry on $X^{*}$. We begin this section by observing that if $X=\ell_{p}$ with $1<p<\infty$ and $p\neq 2$, or if $X=c_{0}$, this is no longer true. Thus one cannot follow this route to show that a typical $T\in(\bbx,\sot)$ has a non-trivial invariant subspace.

\begin{proposition}\label{lppasisom} 
 Let $X=c_0$ or $\ell_p$, with $1<p<\infty$ and $p\neq 2$. The set of all $T\in\bbx$ such that $T^*$ is an isometry is nowhere dense in $(\bbx,\emph{\sot})$. In particular, a typical $T\in(\bbx,\emph{\sot})$ is such that $T^*$ is not an isometry.
\end{proposition}

\bpf The dual space of $X$ is equal to $\ell_q$ for some $1\leq q<\infty$ with $q\neq 2$. Now, the isometries of $\ell_q$ (or more generally the isometries of $L_q$ spaces) are completely described by a classical result of Lamperti \cite{L}. All we need to know is that if $S$ is an isometry of $\ell_q$, then $S$ maps functions with disjoint supports to functions with disjoint supports: if $f,g\in\ell_q$ and $f\cdot g=0$, then $(Sf)\cdot (Sg)=0$ (see \mbox{e.g.} \cite[Corollary 8.4]{Caro} for a proof of this fact).

So if $T\in\bbx$ is such that $T^*$ is an isometry, then 
$(T^*e_0^*)\cdot( T^*e_1^*)=0$. In other words, the set \[\bigl\{ T\in\bbx;\;\hbox{$T^*$ is an isometry}\bigr\}\] is contained in 
\[ {\mathcal F}:= \bigl\{ T\in\bbx;\; \forall j\in\Z_+\;:\; \langle e_0^*, Te_j\rangle=0\;{\rm or}\;\, \langle e_1^*,Te_j\rangle=0\bigr\} .\] 
The set ${\mathcal F}$ is obviously \sot$\,$-$\,$closed in $\bbx$, and it is easily seen that its complement is $\sot\,$-$\,$dense in $\bbx$. In fact, it is clear that for any fixed $j\geq 0$, the open set \[ {\mathcal O}_j:=\bigl\{ T\in\bbx;\; \langle e_0^*,Te_j\rangle\neq 0\;{\rm and}\;\, \langle e_1^*,Te_j\rangle\neq 0\bigr\}\]
is already $\sot\,$-$\,$dense in $\bbx$. This concludes the proof. %$\langle T^*e_0^*,e_n\rangle\,\langle T^*e_1*, e_n\rangle=0$ for all $n\in\Z_+$.
\epf
\par\medskip
In view of the results from Sections \ref{Section2} and \ref{Section3} above, it is also a natural question to ask whether, for $X=\ell_{p}$ with $1<p<\infty$ and $p\neq 2$ or $X=c_{0}$, a typical $T\in(\bbx,\sot)$ admits eigenvalues or not. This question turns out to be surprisingly difficult, and we are only able to answer it for $X=c_0$ and $X=\ell_p$ with $p>2$.

\begin{theorem}\label{Valeurs propres c0} Let $X=c_0$. A typical $T\in(\bbx,\emph{\sot})$ has no eigenvalue.
\end{theorem}

In the $\ell_p\,$-$\,$case, we obtain the following stronger result. 

\begin{theorem}\label{Valeurs propres lp}
 Let $X=\ell_{p}$ with $p> 2$. For any $M>1$, a typical $T\in(\bbx,\emph{\sot})$ is such that $(MT)^*$ is hypercyclic. In particular, a typical $T\in(\bbx,\emph{\sot})$ has no eigenvalue.
\end{theorem}

Recall that the adjoint of a hypercyclic operator cannot have any eigenvalue. Indeed, suppose that $S$ is a hypercyclic operator acting on a Banach space $Y$, and suppose that $S^{*}y^{*}=\lambda y^{*}$ for some $y^{*}\in Y^{*}$ and $\lambda \in\C$. Then $\pss{y^{*}}{S^n y}=\lambda^n \pss{y^{*}}{y}$ for every $y\in Y$ and every $n\ge 0$. It follows that the set $\{\pss{y^{*}}{S^n y}\; ;\; n\ge 0\}$ is never dense in $\C$, and thus the vector $y$ cannot be hypercyclic for the operator $S$.

\par\smallskip
The reason why we have to assume, in the statement of Theorem \ref{Valeurs propres lp}, that $p>2$, lies in the well-known fact that the $\ell_{p}\,$-$\,$norms have very different convexity and smoothness properties when $p>2$ and when $1<p<2$. Many classical inequalities on $\ell_{p}\,$-$\,$norms, like the Clarkson inequalities for instance, reverse when one moves over from the case $p>2$ to the case $1<p<2$. Also, if $p>2$, the 
$p\,$-th power $\Vert\,\cdot\,\Vert_{p}^{p}$ of the $\ell_{p}\,$-$\,$norm is $\mathcal{C}^{2}\,$-$\,$smooth, while it is only $\mathcal{C}^{1}\,$-$\,$smooth when $1<p<2$ (see \cite[Chapter 5]{DGZ} for more on smoothness of $L^{p}\,$-$\,$norms).
More prosaically, we will  also use our assumption that $p>2$ through the following classical inequality (see \mbox{e.g.} \cite[Lemma 2.1]{Kan}):
\begin{lemma}\label{Lemme Kan}
 Let $u$ and $v$ be complex numbers with $v\neq 0$. If $2<p<\infty$, we have 
 \[
\vert u+v\vert^{p}+\vert u-v\vert^{p}>2\,\vert u\vert^{p}+p\,\vert u\vert^{p-2}\, \vert v\vert^{2}.
\]
{\rm (}If $0<p<2$ and $u\neq 0$, the reverse inequality holds.{\rm )}
\end{lemma}

\smallskip
The proofs of Theorems \ref{Valeurs propres c0} and \ref{Valeurs propres lp} use rather different ideas. However, it is possible to give a direct proof of a weaker yet apparently non-trivial statement that follows the same pattern in the $c_0\,$-$\,$case and in the $\ell_p\,$-$\,$case. (The second part  of the statement follows from Proposition \ref{invertible}.)

\smallskip
%Combining Theorems  \ref{Valeurs propres c0} and \ref{Valeurs propres lp} with ,  we get the following
\begin{corollary}\label{Injection lp}
 Let $X=\ell_{p}$ with $p> 2$ or $X=c_{0}$. A typical $T\in(\bbx,\emph{\sot})$ is {one-to-one}; and hence a typical $T\in\bbx$ is not surjective.
\end{corollary}

\smallskip Accordingly we will first give this direct  proof of Corollary \ref{Injection lp}, and then we will prove Theorem \ref{Valeurs propres c0} and Theorem \ref{Valeurs propres lp}.

%We now start with the proof of Proposition \ref{Injection lp}.
%\par\medskip
\subsection{Direct proof of Corollary \ref{Injection lp}} To show that a typical $T\in\bbx$ is one-to-one when $X=\ell_p$ for some $p>2$ or $X=c_0$, we are going to write down explicitely a $\gd$ subset $\mathcal G$ of $(\bbx, \sot)$ such that every $T\in\mathcal G$ is one-to-one and $\mathcal G$ is dense in $\bbx$. 

\smallskip The definition of $\mathcal G$ is rather technical looking, but it does not require $p>2$ in the $\ell_p$ case; so assume temporarily that $X=\ell_p$, $1\leq p<\infty$ or $X=c_0$. For every integer $k\geq 1$, let $\mathcal A_{k}\subseteq\bbx$ be defined in the following way:
 \[
T\in{\mathcal A}_k\iff \exists \varepsilon>0,\ \exists\, 0<\delta_1\leq \delta_2,\  \exists\, N>k,\  \exists\, r_1, r_2\in\Z_{+} \quad\textrm{such that}\]
\begin{enumerate}[\rm (i)]
\item%[\rm (i)] 
$\left\Vert e_{r_1}^*TP_{(N,\infty)}\right\Vert<\varepsilon \cdot 10^{-(k+1)}$\quad  \textrm{and}\quad  $\left\Vert e_{r_2}^*TP_{(N,\infty)}\right\Vert<\varepsilon \cdot10^{-(k+1)}$;

\smallskip
\item%[\rm (ii)] 
$\left\vert\langle e_{r_1}^*, Tx\rangle\right\vert > \left\vert\varepsilon x_k\right\vert-\vert \delta_1 x_N\vert  -\varepsilon\cdot 10^{-(k+1)}$\quad \textrm{and}\quad  $ \left\vert\langle e_{r_2}^*, Tx\rangle\right\vert >\vert\delta_2 x_N\vert-\varepsilon\cdot 10^{-(k+1)}$  for every $x\in B_{E_N}$. 
\end{enumerate}

Now set 
\[ \mathcal G:=\bigcap_{k\ge 1}\stackrel{\;\;\scriptscriptstyle\circ}{\mathcal A}_{k},\]
where $\stackrel{\;\;\scriptscriptstyle\circ}{\mathcal A}_{k}$ is the {\sot}$\,$-$\,$interior of $\mathcal A_{k}$ relative to $\bbx$.

\smallskip
\begin{lemma}\label{oneonecriterion} Any operator $T\in\bigcap_{k\ge 1}\mathcal A_{k}$ is one-to-one. %We deduce that any operator $T\in\bigcap_{k\ge 1}\u_{k}$ is one-to-one. Since $\bigcap_{k\ge 1}\stackrel{\scriptscriptstyle\circ}{\u}_{k}$ is a dense $G_{\delta }$ subset of $(\bbx,\sot)$ by Lemma \ref{Lemme fond Injection lp}, it follows that the set of one-to-one operators is comeager in $(\bbx,\sot)$. Proposition \ref{Injection lp} is proved.
\end{lemma}
\bpf Indeed, let $T\in\bigcap_{k\ge 1}\mathcal A_{k}$, and suppose that $x=\sum_{j\ge 0}x_{j}e_{j}\in X$ satisfies $\Vert x\Vert=1$ and $Tx=0$. Since $\Vert x\Vert=1$, there exists $k\ge 0$ such that $\vert x_{k}\vert>5\cdot10^{-(k+1)}$.
\par\medskip
Since $T\in\mathcal A_{k}$, we may choose $\varepsilon $, $\delta _{1}$, $\delta _{2}$, $N$, $r_{1}$, and $r_{2}$ such that 
 conditions (i) and (ii) above are satisfied. On the one hand, we have
\begin{align*}
 0=\vert\pss{e_{r_{1}}^{*}}{Tx}\vert&\ge\vert\pss{e_{r_{1}}^{*}}{TP_{[0,N]}x}\vert-\Vert e_{r_{1}}^{*}TP_{(N,\infty)}\Vert\\
 &>\varepsilon \,\vert x_{k}\vert-\delta _{1}\vert x_{N}\vert-2\varepsilon\cdot 10^{-(k+1)}>3\varepsilon\cdot 10^{-(k+1)}-\delta _{1}\vert x_{N}\vert .
\end{align*}
On the other hand,
\begin{align*}
 0=\vert \pss{e_{r_{2}}^{*}}{Tx}\vert &\ge\vert \pss{e_{r_{2}}^{*}}{TP_{[0,N]}x}\vert -\Vert e_{r_{2}}^{*}TP_{(N,\infty)}\Vert\\
 &>\delta_{2}\vert x_{N}\vert -2\varepsilon\cdot 10^{-(k+1)}.
\end{align*}
Hence $3\varepsilon\cdot 10^{-(k+1)}<\delta _{1}\vert x_{N}\vert \le \delta _{2}\vert x_{N}\vert <2\varepsilon\cdot 10^{-(k+1)}$, which is a contradiction. 
\epf

%The main step in the proof of Proposition \ref{Injection lp} is the following technical  looking lemma:
\smallskip
\begin{lemma}\label{Lemme fond Injection lp}
Assume that $X=\ell_{p}$ for some $p>2$ or $X=c_{0}$. Then, for every integer $k\ge 1$, %, let $\u_{k}\subseteq\bbx$ be defined in the following way:
% \[
%T\in{\mathcal A}_k\iff \exists \varepsilon>0,\ \exists\, 0<\delta_1\leq \delta_2,\  \exists\, N>k,\  \exists\, r_1, r_2\in\Z_{+} \quad\textrm{such that}\]
%\begin{enumerate}[\rm (i)]
%\item%[\rm (i)] 
%$\left\Vert e_{r_1}^*TP_{(N,\infty)}\right\Vert<\varepsilon \cdot 10^{-(k+1)}$\quad  \textrm{and}\quad  $\left\Vert e_{r_2}^*TP_{(N,\infty)}\right\Vert<\varepsilon \cdot10^{-(k+1)}$;
%\smallskip
%\item%[\rm (ii)] 
%$ \forall x\in B_{E_N}\;: \left\vert\langle e_{r_1}^*, Tx\rangle\right\vert > \left\vert\varepsilon x_k\right\vert-\vert \delta_1 x_N\vert  -\varepsilon\cdot 10^{-(k+1)}$\quad \textrm{and}\quad  $ \left\vert\langle e_{r_2}^*, Tx\rangle\right\vert >\vert\delta_2 x_N\vert-\varepsilon\cdot 10^{-(k+1)}.$
%\end{enumerate}
%\par\medskip
%\noindent Then, 
the \emph{\sot}$\,$-$\,$interior of $\mathcal A_{k}$ relative to $\bbx$ is dense in $(\bbx,\emph{\sot})$.
\end{lemma}
\begin{proof}
 Fix \(k\ge 1\). Let $\u$ be a non-empty open subset of $(\bbx,\sot)$, and pick a finite rank operator $B\in\u$. There exist $\varepsilon \in(0,1/4)$ and an integer $N>k+1$ such that 
$\u$ contains all operators $T\in\bbx$ satisfying $\Vert(T-B)e_{j}\Vert<4\varepsilon $ for every $j=0,\dots,N-1$. We also fix an integer $M>N$ such that $\textrm{Ran}(B)\subseteq E_{M}$.
\par\smallskip
% Recall that for every integer $M\ge 0$,  we denote by $P_{M}$ the canonical projection of $X$ onto 
% $E_M=[\,e_{0},\dots,e_{M}]$. 
%Since $P_{M}B\xrightarrow{\textrm{\sot}}B$ as $M\rightarrow \infty$, we can suppose without loss of generality that the range of $B$, which we denote by $\textrm{Ran}(B)$, satisfies $\textrm{Ran}(B)\subseteq E_{M}$ for some $M>N$. We then have that if $\Vert(T-B)e_{j}\Vert<4\varepsilon $ for every $j=0\,\dots,N-1 $, then
%$T$ belongs to $\u$.
%\par\smallskip
Our aim is to exhibit an operator $A\in\bbx$ belonging to the \sot$\,$-$\,$interior $\stackrel{\;\;\scriptscriptstyle\circ}{\mathcal A}_{k}$ of $\mathcal A_{k}$ in $\bbx$, and satisfying $\Vert(A-B)e_{j}\Vert<4\varepsilon $ for every $j=0,\dots,N-1$. We will have to treat separately in our discussion the cases $X=\ell_{p}$ and $X=c_{0}$.
\par\medskip
\quad \emph{Case 1.} Suppose that $X=\ell_{p}$ for some $p>2$, and let $\delta >0$. We define an operator $A$ in the following way:
\[
Ae_{j}:=
\begin{cases}
 (1-2\varepsilon )\,Be_{j}&\textrm{if}\ 0\le j<N\ \textrm{and}\ j\neq k,\\
 (1-2\varepsilon )\,Be_{k}+\varepsilon\, e_{M+1}\quad &\textrm{if}\ j=k,\\
 \delta (e_{M+1}+e_{M+2})&\textrm{if}\ j=N,\\
 0&\textrm{if}\ j>N.
\end{cases}
\]
By the Intermediate Value Theorem, we can now choose $\delta >0$ in such a way that we have $\Vert A\Vert=1$. Obviously, $\Vert(A-B)e_{j}\Vert<4\varepsilon $ for every $0\le j< N$. What remains to be shown is that if $T\in\bbx$ is sufficiently close to $A$ for the topology \sot, then $T\in\mathcal A_{k}$.
\par\medskip
For every $x=\sum_{j\ge 0}x_{j}e_{j}\in X$, we have
\[
\pss{e_{M+1}^{*}}{Ax}=\varepsilon x_{k}+\delta x_{N}\quad \textrm{and}\quad \pss{e_{M+2}^{*}}{Ax}=\delta x_{N}.
\]
Moreover, $AP_{(N,\infty)}=0$, and so $A$ belongs to $\mathcal A_{k}$, the data witnessing this being $\varepsilon $, $\delta _{1}=\delta _{2}=\delta $, $N$, $r_{1}=M+1$, and $r_{2}=M+2$. If $T$ is sufficiently close to $A$ with respect to \sot, it will satisfy property (ii) of the definition of $\mathcal A_{k}$ for the same choices of $\varepsilon $, $\delta_{1} $, $\delta _{2}$, $N$, $r_{1}$, and $r_{2}$ (since the ball \(B_{E_{N}}\) is compact, the condition given by property (ii) is \sot-open). Let us now show that if $T$ is sufficiently close to $A$ for \sot, it also satisfies property (i).
\par\medskip
Let $x\in E_{N}$ be a norming vector for $A$ (\mbox{\it i.e.} $\Vert Ax\Vert=\Vert x\Vert=1$). Since $\Vert A_{| E_{N-1}}\Vert\le 1-\varepsilon <1$, $x$ does not belong to $E_{N-1}$, so that we have $x_{N}\neq 0$. Hence $\pss{e_{r_{2}}^{*}}{Ax}\neq 0$. Moreover, we also have $\pss{e_{r_{1}}^{*}}{Ax}\neq 0$. Indeed, suppose that $\pss{e_{r_{1}}^{*}}{Ax}=0$, and consider the vector $x'=x-2x_{N}e_{N}$. It satisfies $\Vert x'\Vert=\Vert x\Vert=1$, and 
\[
Ax'=Ax-2x_{N}\delta (e_{M+1}+e_{M+2})=P_{[0,M]}Ax-2x_{N}\delta e_{M+1}-x_{N}\delta e_{M+2}.
\]
So 
\begin{align*}
 \Vert Ax'\Vert^{p}&=\Vert P_{[0,M]}Ax\Vert^{p}+\vert 2\delta x_{N}\vert ^{p}+\vert \delta x_{N}\vert ^{p},
 \intertext{while}
 \Vert Ax\Vert^{p}&=\Vert P_{[0,M]}Ax\Vert^{p}+\vert \delta x_{N}\vert ^{p}=1.
\end{align*}
Hence $\Vert Ax'\Vert>1$, which is a contradiction. So $\pss{e_{r_{1}}^{*}}{Ax}\neq 0$. Let
\[
c:=\min\,\bigl (\vert \, \pss{e_{r_{1}}^{*}}{Ax}\,\vert ,\vert \,\pss{e_{r_{2}}^{*}}{Ax}\,\vert \bigr ),
\]
and let $\eta \in(0,1)$. Any operator $T\in\bbx$ which is sufficiently close to $A$ for the topology \sot\ satisfies the following properties:
\[
\vert \pss{e_{r_{i}}^{*}}{Tx}\vert \ge c/2,\ i=1,2,\quad \textrm{and}\quad \Vert Tx\Vert^{p}\ge 1-\eta .
\]
The inequality from Lemma \ref{Lemme Kan} now comes into play: for every $y\in F_{N}=[\,e_{j}\;;\;j> N\,{]}$, for any  $\lambda >0$, and for
every $j\ge 0$, we have
\[
\vert\pss{e_{j}^{*}}{Tx+\lambda Ty}\vert^{p}+\vert\pss{e_{j}^{*}}{Tx-\lambda Ty}\vert^{p}\ge 2 \vert\pss{e_{j}^{*}}{Tx}\vert^{p}+p\lambda ^{2}\,
\vert\pss{e_{j}^{*}}{Tx}\vert^{p-2}\,\vert \pss{e_{j}^{*}}{Ty}\vert^{2}.
\]
Summing over $j$, we obtain that
\[
\Vert Tx+\lambda Ty\Vert^{p}+\Vert Tx-\lambda Ty\Vert^{p}\ge 2\Vert Tx\Vert^{p}+p\lambda ^{2}\,\sum_{j\ge 0} \vert\pss{e_{j}^{*}}{Tx}\vert^{p-2}\,\vert \pss{e_{j}^{*}}{Ty}\vert^{2}.
\]
Hence we have for
$i=1,2$,
\[
\Vert Tx+\lambda Ty\Vert^{p}+\Vert Tx-\lambda Ty\Vert^{p}\ge 2\Vert Tx\Vert^{p}+p\lambda ^{2}\, \vert \pss{e_{r_{i}}^{*}}{Tx}\vert^{p-2}\,\vert \pss{e_{r_{i}}^{*}}{Ty}\vert^{2}.
\]
So if $T$ is \sot$\ $-$\,$close to $A$,
\[
\Vert Tx+\lambda Ty\Vert^{p}+\Vert Tx-\lambda Ty\Vert^{p}\ge 2(1-\eta) +p\,\lambda ^{2}\,(c/2)^{p-2}\,\vert\pss{e_{r_{i}}^{*}}{Ty}\vert^{2}
\]
for every $y\in F_{N}$ and every $\lambda >0$.
\par\medskip
On the other hand, since $\Vert T\Vert\le 1$, and since the vectors $x$ and $y$ have disjoint supports, we also have
\begin{align*}
 \Vert Tx+\lambda Ty\Vert^{p}+\Vert Tx-\lambda Ty\Vert^{p}&\le \Vert x+\lambda y\Vert^{p}+\Vert x-\lambda y\Vert^{p}\\
 &=2(\Vert x\Vert^{p}+\lambda^p\Vert y\Vert^{p})\\
 &=2+2\lambda ^{p}\Vert y\Vert^{p}.
 \end{align*}
{Hence}
\[ \vert\pss{e_{r_{i}}^{*}}{Ty}\vert^{2}\le\dfrac{2}{p\,(c/2)^{p-2}}\Bigl (\dfrac{\eta }{\lambda ^{2}}+\lambda ^{p-2} \Bigr) 
\]
for every $\lambda >0$, and every $y\in F_{N}$ such that $\Vert y\Vert\le 1$. If we choose first $\lambda >0$ sufficiently small, and then $\eta \in(0,1)$ sufficiently small, we obtain that for all $T\in(\bbx,\sot)$ sufficiently close to $A$,
\[
\Vert e_{r_{i}}^{*}TP_{(N,\infty)}\Vert<\varepsilon\cdot 10^{-k},\quad i=1,2
\]
which is exactly property (i) from the definition of $\mathcal A_{k}$. This terminates the proof of Lemma \ref{Lemme fond Injection lp} in the $\ell_{p}\,$-$\,$case.
\par\medskip
\emph{Case 2.}\quad Suppose that $X=c_{0}$. In this case, we define the operator $A$ in the following way:
\[
Ae_{j}:=
\begin{cases}
 Be_{j}&\textrm{if}\ 0\le j<N\ \textrm{and}\ j\neq k,\\
 Be_{k}+\varepsilon\,e_{M+1}&\textrm{if}\ j=k,\\
 (1-\varepsilon )\,e_{M+1}+e_{M+2}\quad &\textrm{if}\ j=N,\\
 0&\textrm{if}\ j>N.
\end{cases}
\]
Then $\Vert A\Vert=1$, and clearly \(\Vert(A-B)e_{j}\Vert<4\varepsilon \) for every \(j=0,\dots,N-1\). Moreover, for every $x=\sum_{j\ge 0}x_{j}e_{j}\in X$,
\[
\pss{e_{M+1}^{*}}{Ax}=\varepsilon x_{k}+(1-\varepsilon )x_{N}\quad \textrm{and}\quad \pss{e_{M+2}^{*}}{Ax}=x_{N}.
\]
Also $AP_{(N,\infty)}=0$, and so \(A\) belongs to \(\mathcal A_{k}\) for the choices of the associated data as 
$\varepsilon $, $\delta _{1}=1-\varepsilon $, $\delta _{2}=1$, $N$, $r_{1}=M+1$, and $r_{2}=M+2$. If $T$ is sufficiently close to $A$ with respect to \sot, (ii) is clearly satisfied with the same choices of $\varepsilon $, $\delta _{1}$, $\delta _{2}$, $N$, $r_{1}$, and $r_{2}$, and what remains to be shown is that (i) is also satisfied.
\par\medskip
The argument for this is much more direct than in the previous case, because if we set $u_{1}:=e_{k}+e_{N}$ and $u_{2}:=e_{N}$, we have $\pss{e_{r_{i}}^{*}}{Au_{i}}=1=\Vert u_{i}\Vert$, $i=1,2$.
\par\medskip
Let $\eta \in(0,1)$. If $T$ is sufficiently close to $A$ for  \sot\ then $\vert\pss{e_{r_{i}}^{*}}{Tu_{i}}\vert>1-\eta $, and hence $\Vert e_{r_{i}}^{*}TP_{[0,N]}\Vert>1-\eta $, $i=1,2$. The definition of the $c_{0}\,$-$\,$norm implies that 
\[
\Vert e_{r_{i}}^{*}T\Vert=\Vert e_{r_{i}}^{*}TP_{[0,N]}\Vert+\Vert e_{r_{i}}^{*}TP_{(N,\infty)}\Vert,
\]
and thus $\Vert e_{r_{i}}^{*}TP_{(N,\infty)}\Vert\le \eta $, $i=1,2$. Choosing $\eta =\varepsilon\cdot 10^{-k}$, we obtain that any $T\in(\bbx,\sot)$ sufficiently close to $A$ satisfies (i) with the choices of $\varepsilon $, $\delta _{1}$, $\delta _{2}$, $N$, $r_{1}$, and $r_{2}$ as above, and so $T\in\mathcal A_{k}$.
\end{proof}
\begin{proof}[Proof of Corollary \ref{Injection lp}] Assuming that $X=\ell_p$, $p>2$ or $X=c_0$, it follows from the above two lemmas that the set $\mathcal G=\bigcap_{k\ge 1}\stackrel{\;\;\scriptscriptstyle\circ}{\mathcal A}_{k}$ is a dense $G_{\delta }$ subset of $(\bbx,\sot)$ consisting of one-to-one operators. 
\end{proof}
%\begin{remark}\label{remark helas}
 %We did not see how to extend the arguments used in the proof of Proposition \ref{Injection lp} to show that a typical operator \(T\in (\mathcal{B}_{1}(X),\sot)\) on \(X=\ell_{p}\) with \(p>2\), or on \(X=c_{0}\), has no eigenvalue. In order to prove this statement in the \(\ell_{p}\)$\,$-$\,$case, we will have to work substantially more.
%\end{remark}

\subsection{Proof of Theorem \ref{Valeurs propres c0}} We will make use of the so-called \emph{Banach-Mazur game}, which is often a very  effective tool for showing that a given set is  comeager. 
% The proofs of Corollary \ref{Injection lp} and Theorem \ref{Valeurs propres lp} could also be presented in terms of a Banach-Mazur game.
\par\smallskip
Let $\mathcal E$ be a Polish space, and let $\mathcal A\subseteq \mathcal E$. The Banach-Mazur game $\mathbf G(\mathcal A)$ is an infinite game with two players, denoted by I and II. The two players play alternatively open sets $\mathcal U_0\supseteq \mathcal U_1\supseteq \mathcal U_2\supseteq \cdots$;  so, player I plays the open sets  of  even indices $\mathcal U_{2k}$ and player II plays the open sets of odd indices $\mathcal U_{2k+1}$.  Then, player II wins the run if $\bigcap_{n\geq 0} \mathcal U_n\subseteq \mathcal  A$. The key result concerning this game is the following (see \mbox{e.g.} \cite[p.\, 51]{Ke} for a proof): the set $\mathcal A$ is comeager in $\mathcal E$ if and only if player II has a winning strategy in $\mathbf G(\mathcal A)$. Moreover, nothing changes if the open sets $\mathcal U_0, \mathcal U_1, \mathcal U_2, \dots$ are required to be picked from a given basis for the topology of $\mathcal E$. 

\smallskip In our setting, the space $\mathcal E$ is $(\bbc, \sot)$, and $\mathcal A$ is the set of all operators $T\in \bbc$ having no eigenvalue. The basic open sets will be of the form
\[ \mathcal U({N,A,\varepsilon}):= \bigl\{ T\in\bbc;\; \Vert Te_j-Ae_j\Vert <\varepsilon\;\,\hbox{for $j=0,\dots ,N$}\bigr\},\]
where $N\in\N$, $A\in\mathcal B_1(E_N)%\mathcal B_1(E_N, c_0)
$ and $0<\varepsilon\leq 1$. So player I will play open sets  of the form $\mathcal U_{2k}=\mathcal U(N_{2k}, A_{2k}, \varepsilon_{2k})$, and player II will play open sets of the form $\mathcal U_{2k+1}=\mathcal U(N_{2k+1}, A_{2k+1}, \varepsilon_{2k+1})$.

\smallskip %Player I starts the game by playing a basic open set $\mathcal U_0=\mathcal U(A_0, N_0,\varepsilon_0)$. 
Let $(\alpha_k)_{k\geq 0}$ be  a sequences of positive  real numbers, with $\alpha_k\leq 1$. Let also $0<c\leq1/2$ and $C>1$. We are going to describe a strategy for player II depending on $(\alpha_k)$,  $c$ and $C$, and then to show  that  this  strategy is winning if  $(\alpha_k)$, $c$ and $C$ are suitably chosen.

\smallskip Assume that player I has just played a basic open set $\mathcal U_{2k}=\mathcal U(A_{2k}, N_{2k}, \varepsilon_{2k})$. Let us set 
%\[ \tau_k:=\frac{\varepsilon_{2k}}8\,\frac1{(k+2)^2-1/2};\]
\[ \tau_k:=\alpha_k \varepsilon_{2k},\]
and choose a $\tau_k\,$-$\,$net $\Lambda_k=\{ \lambda_1, \dots ,\lambda_{L_k}\}$ for the unit circle $\T$, with $\lambda_1:=1$. Let us also choose %an integer $M_k\geq N_{2k}$ such that 
%\[ \forall j\leq N_{2k}\;:\; \Vert P_{M_k}A_{2k}e_j-A_{2k}e_j\Vert <\frac{\varepsilon_{2k}}2,\]
%and 
an integer $R_k\geq 2$ such that 
\[  \left(\frac{1-\tau_k/2}{1-\tau_k}\right)^{R_k-2}>\frac{C}{\varepsilon_{2k}}\cdot\]

Now, set 
%\[ N_{2k+1}:= M_k+(L_k+1)R_k -1,\]
\[ N_{2k+1}:= N_{2k}+(L_k+1)R_k -1,\]

\smallskip\noindent and define an operator $A_{2k+1}\in\mathcal B(E_{N_{2k+1}})%\mathcal B(E_{N_{2k+1}}, c_0)
$ as follows :

%\medskip
%\be
%\item[-] If $j\leq N_{2k}$ and $j\neq k$, then
\[ A_{2k+1} e_j:= A_{2k}e_j\qquad\hbox{if $j\leq N_{2k}$ and $j\neq k$};\]
%\[ A_{2k+1} e_j:= P_{[0,M_k]} A_{2k}e_j\qquad\hbox{if $j\leq N_{2k}$ and $j\neq k$};\]
%\item[-] if $j=k$, then 

\[ A_{2k+1}e_k:= A_{2k}e_k+\sum_{i=1}^{L_k} \lambda_i  \frac{\varepsilon_{2k}}2\, e_{N_{2k}+iR_k};\]
%\[ A_{2k+1}e_k:= P_{[0,M_k]} A_{2k}e_k+\sum_{i=1}^{L_k} \lambda_i  \frac{\varepsilon_{2k}}2\, e_{M_k+iR_k};\]
%\item[-] if $j=M_k +iR_k$ with $1<i\leq R_k$, then 

\[ A_{2k+1} e_{N_{2k}+iR_k}:=\lambda_i\left(1-\frac{\varepsilon_{2k}}2\right) e_{N_{2k}+iR_k}\qquad\hbox{for every $1<i\leq L_k$};\]
%\ee

%\[ A_{2k+1} e_{M_k+R_k}:=\lambda_1 \left(1-\frac{\varepsilon_{2k}}2\right) e_{M_k+R_k} +e_{M_k+R_k+1};\]
\[ A_{2k+1} e_{N_{2k}+R_k}:=\lambda_1 \left(1-\frac{\varepsilon_{2k}}2\right) e_{N_{2k}+R_k} +e_{N_{2k}+R_k+1};\]

%\[  A_{2k+1} e_{M_k+R_k+s}:=e_{M_k+R_k+s+1}\qquad\hbox{for every $1\leq s<R_k-1$};\]
\[  A_{2k+1} e_{N_{2k}+R_k+s}:=e_{N_{2k}+R_k+s+1}\qquad\hbox{for every $1\leq s<R_k-1$};\]

\[ A_{2k+1}e_j:=0\qquad\hbox{otherwise}.\]

\smallskip
Observe that in fact $A_{2k+1}\in \mathcal B_1(E_{N_{2k+1}})%\mathcal B_1(E_{N_{2k+1}},c_0)
$, \mbox{\it i.e.} $\Vert A_{2k+1}\Vert\leq 1$. Indeed, if $x=\sum_{j\geq 0} x_je_j\in E_{N_{2k+1}}$, then
\begin{align*}
\Vert A_{2k+1} x\Vert=
\max\left\{ \max_{0\leq j\leq N_{2k}} \right.&\vert \pss{e_j^*}{A_{2k}x}\vert \,,\,\\
 &\left. \max_{1\leq i\leq L_k}\left\vert \frac{\varepsilon_{2k}}2\, x_k+\left(1-\frac{\varepsilon_{2k}}2\right) x_{N_{2k}+iR_k}\right\vert\,,\, \max_{1\leq s<R_k-1} \vert x_{N_{2k}+R_k+s}\vert \right\} \\ \leq \Vert x\Vert.&\end{align*}
 
 Finally, set
 \[ \varepsilon_{2k+1}:=c\,\tau_k\,\varepsilon_{2k}.\]
 
\smallskip Then, the strategy of player II is to play the open set  $\mathcal U(N_{2k+1}, A_{2k+1}, \varepsilon_{2k+1})$. Note that this is indeed allowed: since we are assuming that $c\leq 1/2$ and since $\tau_k\leq 1$, we have $\Vert A_{2k+1} e_j-A_{2k}e_j\Vert \leq \varepsilon_{2k}/2$ for $j=0,\dots ,N_{2k}$, and hence $\mathcal U(N_{2k+1}, A_{2k+1}, \varepsilon_{2k+1})\subseteq \mathcal U(N_{2k}, A_{2k}, \varepsilon_{2k})$. %Hence, II can play 
 %the open set  
 
 \smallskip Before proving that this strategy is winning for II if $(\alpha_k)$, $c$ and $C$ are suitably chosen, we point out the following fact.
 \begin{fact}\label{trucmuche} Let $(\mathcal U_n)_{n\geq 0}$ be a run in the Banach-Mazur game where player {\rm II} has followed the strategy described above. If $T\in\bigcap_{n\geq 0} \mathcal U_n$ then, for any $k\geq 0$, every $1\leq i\leq L_k$ and every $1\leq s<R_k-1$, we have
 \[ \Vert e_{N_{2k}+iR_k}^* TP_{\Z_+\setminus\{ k, N_{2k}+iR_k\}}\Vert \leq 2\varepsilon_{2k+1}\quad{\rm and}\quad \Vert e_{N_{2k}+R_k+s}^*TP_{\Z_+\setminus\{ N_{2k}+R_k+s-1\}}\Vert\leq \varepsilon_{2k+1}.\]
 \end{fact}
 \begin{proof}[Proof of Fact \ref{trucmuche}]  If $x\in c_0$ and $\Vert x\Vert \leq 1$, then $\Vert e_k+e_{N_{2k}+iR_k}+ \omega P_{\Z_+\setminus\{ k, N_{2k}+iR_k\}}x\Vert =1$ for every $\omega\in\T$, by  definition of the $c_0\,$-$\,$norm; so $\Vert T(e_k+e_{N_{2k}+iR_k}+ \omega P_{\Z_+\setminus\{ k, N_{2k}+iR_k\}}x)\Vert \leq 1$ for every $\omega\in\T$ and hence 
 \[ \left\vert\pss{ e_{N_{2k}+iR_k}^*}{T(e_k+e_{N_{2k}+iR_k})}\right\vert +\left\vert \pss{e_{N_{2k}+iR_k}^*}{TP_{\Z_+\setminus\{ k, N_{2k}+iR_k\}} x)}\right\vert\leq 1.\] On the other hand, 
   \begin{align*}
 \left\vert\pss{ e_{N_{2k}+iR_k}^*}{T(e_k+e_{N_{2k}+iR_k})}\right\vert&\geq
 \bigl\vert
 \langle e_{N_{2k}+iR_k}^* A_{2k+1}(e_k+e_{N_{2k}+iR_k})\rangle\bigr\vert \\
 &\phantom{e_{N_{2k}+iR_k}^* }-\bigl\Vert (T-A_{2k+1})(e_k+e_{N_{2k}+iR_k})\bigr\Vert\\
 &=1-\bigl\Vert (T-A_{2k+1})(e_k+e_{N_{2k}+iR_k})\bigr\Vert\\
 &\geq 1-2\varepsilon_{2k+1}.
 \end{align*}
Thus, we see that $\vert \pss{e_{N_{2k}+iR_k}^*}{TP_{\Z_+\setminus\{ k, N_{2k}+iR_k\}} x)}\vert\leq 2\varepsilon_{2k+1}$ for every $x\in B_{c_0}$, \mbox{\it i.e.} $\Vert e_{N_{2k}+iR_k}^* TP_{\Z_+\setminus\{ k, N_{2k}+iR_k\}}\Vert \leq 2\varepsilon_{2k+1}$.

\smallskip Likewise, if $x\in B_{c_0}$ then, on the one hand, 
\[ \left\vert \pss{e_{N_{2k}+R_k+s}^*}{Te_{N_{2k}+R_k+s-1}}\right\vert +\left \vert \pss{e_{N_{2k}+R_k+s}^*}{TP_{\Z_+\setminus\{ N_{2k}+R_k+s-1\}}x}\right\vert\leq 1\]
and, on the other hand,
\begin{align*} \left\vert \pss{e_{N_{2k}+R_k+s}^*}{Te_{N_{2k}+R_k+s-1}}\right\vert &\geq \left\vert \pss{e_{N_{2k}+R_k+s}^*}{A_{2k+1}e_{N_{2k}+R_k+s-1}}\right\vert\\
&\phantom{e_{N_{2k}+R_k+s}^*}   -\left\Vert (T-A_{2k+1})e_{N_{2k}+R_k+s-1} \right\Vert\\
&\geq 1-\varepsilon_{2k+1};
\end{align*}
and this shows that $\Vert e_{N_{2k}+R_k+s}^*TP_{\Z_+\setminus\{ N_{2k}+R_k+s-1\}}\Vert\leq \varepsilon_{2k+1}$.
 \epf
 
 We can now prove
 \begin{lemma}\label{winning} Assume that $c<1/24$, and choose $\eta\geq 16c$ such that $8c+\eta <1$. If $C>4/\eta$ and $\prod_{k=0}^\infty \frac{1-\alpha_k}{1+4\alpha_k}\geq 8c+\eta$, then the strategy described above is winning for player~{\rm II}.
 \end{lemma}
 \bpf
 Let $(\mathcal U_n)_{n\geq 0}$ be a run in the game where player {\rm II} has followed the  strategy above, and let $T\in\bigcap_{n\geq 0} \mathcal U_n$. We have to show that $T$ has no eigenvalue. Towards a contradiction, assume that $Tx=\lambda x$ for some $x\in c_0$ such that $\Vert x\Vert =1$ and some $\lambda\in\C$. Note that $\vert \lambda\vert\leq 1$ since $\Vert T\Vert\leq1$. In what follows, we write   $x=\sum\limits_{j\geq 0} x_j e_j$.
 
% Since $\Vert T\Vert\leq 1$, it is enough to prove that if $x\in c_0$ satisfies $\Vert x\Vert =1$ and if $\lambda\in\C$ satisfies $\vert \lambda\vert\leq 1$, then $Tx\neq \lambda x$. Towards a contradiction, assume that . Write. %We distinguish 2 cases. In what follows, we fix $\eta$ such that \[ \eta\geq 16c\qquad{\rm and}\qquad 8c+\eta<1.\]
 
 \begin{claim}\label{cl1} If $k\geq 0$ is such that $\vert x_k\vert \geq 8c+\eta$, then $1-\tau_k\leq \vert \lambda\vert $.
 \end{claim}
 \begin{proof}[Proof of Claim \ref{cl1}]
Assume that $\vert\lambda\vert<1-\tau_k$ for some $k\geq 0$ such that $\vert x_k\vert \geq 8c+\eta$. Looking at the $(N_{2k}+R_k)$-th coordinate of $Tx-\lambda x$, we may write 
 \begin{align*}
 0&=\left\vert \pss{e_{N_{2k}+R_k}^*}{Tx}- \lambda \, x_{N_{2k}+R_k}\right\vert\\
 &\geq \left\vert \pss{e_{N_{2k}+R_k}^*}{TP_{\{ k, N_{2k}+R_k\}}x}- \lambda\, x_{N_{2k}+R_k}\right\vert-\left\vert  \pss{e_{N_{2k}+R_k}^*}{TP_{\Z_+\setminus\{ k, N_{2k}+R_k\}}x} \right\vert\\
 &\geq \,\left\vert \pss{e_{N_{2k}+R_k}^*}{A_{2k+1}(x_ke_k+x_{N_{2k}+R_k}e_{N_{2k}+R_k})}-\lambda\, x_{N_{2k}+R_k}\right\vert-\left\Vert (T-A_{2k+1})x_ke_k\right\Vert\\
 & \phantom{e_{N_{2k}+R_k}^*}-\left\Vert (T-A_{2k+1})x_{N_{2k}+R_k}e_{N_{2k}+R_k}\right\Vert-\left\vert  \pss{e_{N_{2k}+R_k}^*}{TP_{\Z_+\setminus\{ k, N_{2k}+R_k\}}x} \right\vert\\
 &\geq \left\vert \frac{\varepsilon_{2k}}2\, x_k+\left(1-\frac{\varepsilon_{2k}}2-\lambda\right)\, x_{N_{2k}+R_k}\right\vert-4\varepsilon_{2k+1},
 \end{align*}
 where we have used Fact \ref{trucmuche}. So $\left\vert \frac{\varepsilon_{2k}}2\, x_k+\left(1-\frac{\varepsilon_{2k}}2-\lambda\right)\, x_{N_{2k}+R_k}\right\vert\leq 4\varepsilon_{2k+1}$. Since $\vert x_k\vert  \geq 8c+\eta$, $\varepsilon_{2k+1}=c\tau_k\varepsilon_{2k}\leq c\,\varepsilon_{2k}$ and $\left\vert1-\frac{\varepsilon_{2k}}2-\lambda \right\vert\leq 2$, it follows that 
 \[ \vert x_{N_{2k}+R_k}\vert \geq \frac{(8c+\eta)\varepsilon_{2k}/2-4c\, \varepsilon_{2k}}{2}\geq \frac{\eta}4\,\varepsilon_{2k}.\]
 
 %Note that $1-\frac{\varepsilon_{2k}}2-\lambda\neq 0$, since otherwise we would have $\vert x_k\vert \leq 8\varepsilon_{2k+1}/\varepsilon_{2k}\leq 8c$. Moreover, 
 Now, for any $1\leq s<R_k-1$, we have
 \begin{align*}
 0&=\left\vert \pss{e_{N_{2k}+R_k+s}^*}{Tx}-\lambda\, x_{N_{2k}+R_k+s}  \right\vert\\
 &\geq \left\vert \pss{e_{N_{2k}+R_k+s}^*}{Tx_{N_{2k}+R_k+s-1} e_{N_{2k}+R_k+s-1}}-\lambda\, x_{N_{2k}+R_k+s}  \right\vert\\
 &\phantom{e_{N_{2k}+R_k+s}^*}-\left\Vert {e_{N_{2k}+R_k+s}^*}{TP_{\Z_+\setminus\{ N_{2k}+R_k+s-1\}}}\right\Vert\\
 &\geq  \left\vert \pss{e_{N_{2k}+R_k+s}^*}{A_{2k+1}x_{N_{2k}+R_k+s-1} e_{N_{2k}+R_k+s-1}}-\lambda\, x_{N_{2k}+R_k+s}  \right\vert\\
 &\phantom{e_{N_{2k}+R_k+s}^*}-\left\Vert (T-A_{2k+1})e_{N_{2k}+R_k+s-1}\right\Vert-\left\Vert {e_{N_{2k}+R_k+s}^*}{TP_{\Z_+\setminus\{ N_{2k}+R_k+s-1\}}}\right\Vert\\
  &\geq  \left\vert x_{N_{2k}+R_k+s-1}-\lambda\, x_{N_{2k}+R_k+s}  \right\vert-2\varepsilon_{2k+1}\\
  &\geq \left\vert x_{N_{2k}+R_k+s-1}\right\vert -\vert \lambda\vert \,\left\vert x_{N_{2k}+R_k+s}  \right\vert-2\varepsilon_{2k+1}.
\end{align*}

In other words:
\[ \left\vert x_{N_{2k}+R_k+s}  \right\vert\geq \frac{\left\vert x_{N_{2k}+R_k+s-1}\right\vert-2\varepsilon_{2k+1}}{\vert\lambda\vert}\cdot\]

Since $\varepsilon_{2k+1}=c\, \tau_k\varepsilon_{2k}\leq \frac\eta{16}\, \tau_k\varepsilon_{2k}$ and $\vert\lambda\vert <1-\tau_k$, it follows that 
\[ \left\vert x_{N_{2k}+R_k+s}  \right\vert\geq \frac{\left\vert x_{N_{2k}+R_k+s-1}\right\vert-(\eta/8)\,\tau_k\varepsilon_{2k}}{1-\tau_k}\qquad\hbox{for every $1\leq s<R_k-1$}.\]

Recall that $\vert x_{N_{2k}+R_k}  \vert\geq \frac{\eta\,\varepsilon_{2k}}4\cdot$  With $s:=1$, we get 
\[ \vert x_{N_{2k}+R_k+1}  \vert\geq  \vert x_{N_{2k}+R_k}  \vert \,\frac{1-\tau_k/2}{1-\tau_k}\geq \frac{\eta\varepsilon_{2k}}4\cdot\]
Iterating this $R_k-2$ times, we obtain (keeping in mind that $C>4/\eta$)
\[ \vert x_{N_{2k}+R_k+R_k-2}\vert \geq \frac{\eta\,\varepsilon_{2k}}4\,\left(\frac{1-\tau_k/2}{1-\tau_k}\right)^{R_k-2}\geq \frac{\eta\,\varepsilon_{2k}}4\times \frac{C}{\varepsilon_{2k}}>1,\] which is
a contradiction since $\Vert x\Vert =1$.
\epf

\begin{claim}\label{cl2} If $k\geq 0$ is such that $\vert x_k\vert \geq 8c+\eta$, then one can find $k'>k$ such that $$\vert x_{k'}\vert \geq \frac{1-\alpha_k}{1+4\alpha_k}\, \vert x_k\vert\cdot$$
\end{claim}
\begin{proof}[Proof of Claim \ref{cl2}] By Claim \ref{cl1}, we know that $\vert \lambda\vert\geq 1-\tau_k$. Since $\Lambda_k=\{ \lambda_1,\dots ,\lambda_{L_k}\}$ is a $\tau_k\,$-$\,$net for $\T$, it follows that one can find $1\leq i\leq L_k$ such that $\vert \lambda-\lambda_i\vert <2\tau_k$. Then, we may write
\begin{align*}
0&=\left\vert \pss{e_{N_{2k}+iR_k}^*}{Tx}-\lambda\, x_{N_{2k}+iR_k}\right\vert\\
&\geq \left\vert\pss{e_{N_{2k}+iR_k}^*}{T(x_ke_k+x_{N_{2k}+iR_k}e_{N_{2k}+iR_k})}-\lambda\, x_{N_{2k}+iR_k} \right\vert -\left\Vert e_{N_{2k}+iR_k}^* TP_{\Z_+\setminus\{ k, N_{2k}+iR_k\}}\right\Vert \\
&\geq \left\vert\pss{e_{N_{2k}+iR_k}^*}{A_{2k+1}(x_ke_k+x_{N_{2k}+iR_k}e_{N_{2k}+iR_k})}-\lambda\, x_{N_{2k}+iR_k} \right\vert -\left\Vert (T-A_{2k+1}) x_k e_k\right\Vert\\
&\phantom{e_{N_{2k}+iR_k}^*}-\left\Vert (T-A_{2k+1}) x_{N_{2k}+iR_k} e_{N_{2k}+iR_k}\right\Vert-\left\Vert e_{N_{2k}+iR_k}^* TP_{\Z_+\setminus\{ k, N_{2k}+iR_k\}}\right\Vert\\
&\geq \left\vert \lambda_i\,\frac{\varepsilon_{2k}}2\, x_k+\left( \lambda_i\left(1- \frac{\varepsilon_{2k}}2\right)-\lambda\right) x_{N_{2k}+iR_k}\right\vert -4\varepsilon_{2k+1}.
\end{align*}

So we have $\left\vert \lambda_i\,\frac{\varepsilon_{2k}}2\, x_k+\left( \lambda_i\left(1- \frac{\varepsilon_{2k}}2\right)-\lambda\right)x_{N_{2k}+iR_k}\right\vert \leq 4\varepsilon_{2k+1}=4c\tau_k\varepsilon_{2k}$. Since $\vert x_k\vert \geq 8c+\eta$ and $\left\vert  \lambda_i\left(1- \frac{\varepsilon_{2k}}2\right)-\lambda\right\vert\leq \frac{\varepsilon_{2k}}2+2\tau_k$, it follows that 
\[ \vert x_{N_{2k}+iR_k}\vert \geq \frac{\varepsilon_{2k}/2-4c\tau_k\varepsilon_{2k}/(8c+\eta)}{\varepsilon_{2k}/2 +2\tau_k}\, \vert x_k\vert =\frac{1-8c\tau_k/(8c+\eta)}{1+4\tau_k/\varepsilon_{2k}}\, \vert x_k\vert.\]
Since $\tau_k=\alpha_k\varepsilon_{2k}$ and $\varepsilon_{2k}\leq 1$, this implies that 
\[ \vert x_{N_{2k}+iR_k}\vert \geq\frac{1-\alpha_k}{1+4\alpha_k}\, \vert x_k\vert.\]

So we may set $k':= N_{2k}+iR_k$.
\epf

It is now easy to conclude the proof of Lemma \ref{winning}. Since $\Vert x\Vert=1$, we can choose $k_0\geq 0$ such that $\vert x_{k_0}\vert =1$. Then, since we are assuming that $\prod_{k=0}^\infty \frac{1-\alpha_k}{1+4\alpha_k}\geq 8c+\eta$, we may use Claim \ref{cl2} to find an increasing sequence of integers $(k_l)_{l\geq 0}$ such that $\vert x_{k_l}\vert \geq 8c+\eta$ for all $l\geq 0$. Since $x\in c_0$, this is the required contradiction.
\epf

\smallskip By Lemma \ref{winning}, the proof of Theorem \ref{Valeurs propres c0} is now complete.

\subsection{Outline of the proof of Theorem \ref{Valeurs propres lp}} 
%We now move over to the proof of Theorem \ref{Valeurs propres lp}, which 
The proof of Theorem \ref{Valeurs propres lp} will now occupy us for a while. The main step will be to prove the following statement, which interrelates in a rather curious way the topologies \sot\ and \sote\ on $\mathcal B_1(\ell_p)$ when $p>2$.

\bth\label{bizarre?} Let $X=\ell_p$, $p>2$. If $\mathcal O\subseteq \bbx$ is \emph{\sote}-$\,$open, then the relative \emph{\sot}$\,$-$\,$interior of $\mathcal O$ in $\bbx$ is \emph{\sote}-$\,$dense in $\mathcal O$, and hence \emph{\sot}$\,$-$\,$dense as well.
\eth

From this, we immediately deduce
\bco\label{bizarre2} Let $X=\ell_p$, $p>2$. If $\mathcal G\subseteq\bbx$ is \emph{\sote}-$\,\gd$ and \emph{\sot}$\,$-$\,$dense in $\bbx$, then $\mathcal G$ is \emph{\sot}$\,$-$\,$comeager in $\bbx$. Therefore, any \emph{\sote}-$\,$comeager subset of $\bbx$ is \emph{\sot}$\,$-$\,$comeager as well.
\eco
\bpf Choose a sequence $(\mathcal O_k)_{k\in\N}$ of \sote-$\,$open subsets of $\bbx$ such that $\bigcap_{k\in\N} \mathcal O_k=\mathcal G$. Then each set $\mathcal O_k$ is \sot$\,$-$\,$dense in $\bbx$. Denoting by $\stackrel{\scriptscriptstyle\circ}{\;{\mathcal{O}_k}}$ the relative \sot$\,$-$\,$interior of $\mathcal O_k$ in $\bbx$, 
it follows from Theorem \ref{bizarre?} that each $\stackrel{\scriptscriptstyle\circ}{\;{\mathcal{O}_k}}$ is (open and) dense in $(\bbx,\sot)$. Hence $\mathcal G$ is {\sot}$\,$-$\,$comeager since it contains $\bigcap_{k\in\N} \stackrel{\scriptscriptstyle\circ}{\;{\mathcal{O}_k}}$.
\epf

\smallskip The other ingredient is the following result. This was proved for $p=2$ in \cite{GMM}, and essentially the same proof would give the result for an arbitrary $p$.
\bpr\label{monsteraide} Let $X=\ell_p$, $1< p<\infty$. For any real number $M>1$, the set $\{ T\in\bbx;\; (MT)^*\;\hbox{is hypercyclic and $\Vert T\Vert=1$}\}$ is a dense $\gd$ subset of $\bigl( \bbx,\emph{\sote}\bigr)$.
\epr
\bpf The proof of \cite[Proposition 2.3] {GMM}, which is given for $p=2$ but works in fact for any $1<p<\infty$, shows that the set  of hypercyclic operators is $G_\delta$ and dense in $\bigl(\mathcal B_M(X), \sote\bigr)$. Since the map $T\mapsto (MT)^*$ is a homeomorphism from $\bigl(\bbx,\sote\bigr)$ onto $\bigl(\mathcal B_M(X^*),\sote\bigr)$, it follows that the set $\{ T\in\bbx;\; (MT)^*\;\hbox{is hypercyclic }\}$ is $G_\delta$ and dense in $\bigl(\bbx,\sote\bigr)$. Moreover, the set $\{ T\in\bbx;\; \Vert T\Vert=1\}$ is also $G_\delta$ and dense in $\bigl(\bbx,\sote\bigr)$; in fact, \sot$\,$-$\, G_\delta$ and \sote-$\,$dense.
\epf

\smallskip
\begin{proof}[Proof of Theorem \ref{Valeurs propres lp}] The result follows immediatey from Corollary \ref{bizarre2} and Proposition \ref{monsteraide}.
\end{proof}

%\subsection{Proof of Proposition \ref{Operateurs spectre T}}
%\subsubsection{Proof of Theorem \ref{Valeurs propres lp} assuming Theorem \ref{Localisation vp lp}}
%Theorem \ref{Valeurs propres lp} will be an immediate consequence of Theorem \ref{Localisation vp lp} and the following proposition:
%\begin{proof}[Proof of Proposition \ref{Operateurs spectre T}]

%Proposition \ref{Operateurs spectre T} combined with Theorem \ref{Localisation vp lp} now yield Theorem \ref{Valeurs propres lp}:

% \end{proof}

\subsection{Outline of the proof of Theorem \ref{bizarre?}} 
%\par\medskip We now turn to the proof of Theorem \ref{Localisation vp lp}. 
Let us recall our notation: the canonical basis of $X=\ell_{p}$ is denoted by $(e_{j})_{j\ge 0}$, its dual basis by $(e_{j}^{*})_{j\ge 0}$, and for each 
$N\ge 0$, 
\[
E_{N}=[\,e_{0},\dots,e_{N}\,]\qquad \textrm{and}\qquad F_{N}=[\,e_{j}\;;\;j> N\,].
\]
The canonical projection of $X$ onto $E_{N}$ is denoted by $P_{N}$. For any operator $T\in\bx$, we identify the operator $P_{N}TP_{N}\in\bx$ with $P_{N}T_{| E_{N}}\in{\mathcal{B}}(E_{N})$. So we consider $P_{N}TP_{N}$ both as an operator on $X$ and as an operator on $E_{N}$. Likewise, we may consider any operator on $A\in{\mathcal{B}}(E_{N})$ as an operator on $X$ by identifying it with $P_{N}AP_{N}$. 
\par\medskip
 In what follows, we will use the following rather \textit{ad hoc} terminology. 
 
 \smallskip Recall that a vector $x$ is said to be \emph{norming} for an operator $A$ if $\Vert x\Vert=1$ and $\Vert Ax\Vert=\Vert A\Vert$. Given $N\ge 0$, we will say that an operator $A\in{\mathcal{B}}(E_{N})$ is \emph{absolutely exposing} if the set of all norming vectors for $A$ is as small as possible, i.e\ consists only of unimodular multiples of a single vector $x_{0}\in S_{E_{N}}$. We denote by ${\mathcal{E}}_1(E_N)$ the set of absolutely exposing operators $A\in{\mathcal B}_1(E_N)$. Moreover, we say that an absolutely exposing operator $A\in{\mathcal{B}}(E_{N})$ is \emph{evenly distributed} if for any norming vector $x\in S_{E_{N}}$ for $A$, we have $\pss{e_{j}^{*}}{x}\neq 0$ and $\pss{e_{j}^{*}}{Ax}\neq 0$ for every $j\in [0,N]$. 

% Let $L$ be a $K_{\sigma }$ subset of $\D$, and let $(L_{k})_{k\ge 1}$ be an increasing sequence of compact subsets of $\D$ such that $L=\bigcup_{k\ge 1}L_{k}$. In the case where $0\in L$, we can suppose without loss of generality that $L_{1}=\{0\}$ and that $0\not\in L_{k}$ for $k\ge 2$. When $0\not\in L_{k}$, we denote by $\rho _{k}$ a positive number such that $|\lambda |\ge \rho _{k}$ for every $\lambda \in L_{k}$. This assumption will be of use in the proof of Proposition \ref{Premiere proposition Localisation vp lp} below. Let also $(Q_{k})_{k\ge 1}$ be a sequence of closed balls of $X$ with the following property:
 %for every $k\ge 1$, $Q_{k}=\ba{B}(u_{k},r_{k})$ where $r_{k}>0$ and $u_{k}\in X$ is a vector with finite support, contained in some interval $[0,p_{k}]$, $p_{k}\ge 0$; also denoting by $\ba{B}_{X}$ the closed unit ball of $X$, we suppose that the balls $Q_{k}$ are such that $\bigcup_{k\ge 1}\stackrel{\scriptscriptstyle\circ}{\;{Q}_{k}}=\ba{B}_{X}\setminus \{0\}$ and moreover, that each vector $x\in\ba{B}_{X}\setminus\{0\}$ belongs to $\stackrel{\scriptscriptstyle\circ}{\;{Q}_{k}}$ 
% for infinitely many of the integers $k$.
% \par\medskip

\medskip
The key steps in the proof of Theorem \ref{bizarre?} are the next two propositions, whose proofs are postponed to a later section.
\begin{proposition}\label{Premiere proposition Localisation vp lp}
 Let $X=\ell_{p}$, $1< p<\infty$. Let also $N$ be a non-negative integer, and let $A\in {\mathcal{B}}_{1}(E_{N})$. For any $\varepsilon >0$, there exists an operator $B\in{\mathcal{B}}_{1}(E_{2N+1})$
such that 
\begin{enumerate}
\item[\emph{(i)}]\ $\Vert B\Vert=1$;
\item[\emph{(ii)}]\ $B$ is absolutely exposing;
\item[\emph{(iii)}]\ $B$ is evenly distributed;
%\item[\emph{(iv)}]\ $B$ has property $(\ast)_{\Lambda, Q,2N+1}$;
\item[\emph{(iv)}]\ $\Vert BP_{N}-A\Vert<\varepsilon $.
\end{enumerate}

\noindent Moreover, if $\Vert A\Vert$ is sufficiently close to $1$, then one may require that in fact 
\begin{enumerate}
\item[\emph{(iv')}] $\Vert B-A\Vert <\varepsilon$.
\end{enumerate}

 \end{proposition}
This first proposition is valid on any $\ell_{p}\,$-$\,$space with $p>1$. Our assumption that $p>2$ will be needed only in the next statement:
\begin{proposition}\label{Seconde proposition Localisation vp lp}
 Let $X=\ell_{p}$, $p>2$. Fix an integer $M\ge 0$, and suppose that $B\in {\mathcal{B}}_{1}(E_{M})$ is such that $\Vert B\Vert=1$, $B$ is absolutely exposing, and $B$ is evenly distributed. Then, for any $\varepsilon >0$, there exists $\delta >0$ such that the following holds true:
 \[
\textrm{if}\ T\in {\mathcal{B}}_{1}(X)\ \textrm{satisfies}\ \Vert P_{M}TP_{M}-B\Vert<\delta ,\ \textrm{then}\ \Vert P_{M}T(I-P_{M})\Vert<\varepsilon .
\]
\end{proposition}

%Let us now go back to the proof of Theorem \ref{Localisation vp lp}.
Taking the above two propositions for granted, we can now give the

\begin{proof}[Proof of Theorem \ref{bizarre?}] For any $\mathcal A\subseteq\bbx$, we denote by $\stackrel{\;\;\scriptscriptstyle\circ}{\mathcal A}$ the relative \sot$\,$-$\,$interior of $\mathcal A$ in $\bbx$. 
It is in fact enough to show that if 
$\mathcal O\subseteq \bbx$ is {${\sote}$}-$\,$open and $\mathcal O\neq\emptyset$, then $\stackrel{\;\;\scriptscriptstyle\circ}{\mathcal O}\neq\emptyset$.

\smallskip 
Since the unit sphere of $\mathcal B(X)$ is $\sote$-$\,$dense in $\bbx$, one can find $T_0\in\mathcal O$ with $\Vert T\Vert=1$. Choose $K\in\N$ and $\varepsilon$ such that for any $T\in\bbx$, the following implication holds true:
\[ \bigl(\, \Vert (T-T_0)P_K\Vert <\varepsilon\quad{\rm and}\quad \Vert (T^*-T_0^*)P_K^*\Vert <\varepsilon\,\bigr)\implies T\in\mathcal O.\]

\smallskip  Next, choose $L\geq K$ such that, setting $A_N:=P_NT_0P_N$, we have
\[ \forall N\geq L\;:\;  \Vert (A_N-T_0)P_K\Vert <\varepsilon/3\quad{\rm and}\quad \Vert (A_N^*-T_0^*)P_K^*\Vert <\varepsilon/3.\]

\smallskip
Now, let $\alpha>0$ be very small, and fix $N\geq L$ such that $\Vert A_N\Vert>1-\alpha$.

\smallskip  If $\alpha$ is small enough then, by Proposition \ref{Premiere proposition Localisation vp lp}, one can find $M>N$ and $B\in\mathcal B_1(E_M)$ such that $\Vert B\Vert=1$, $B$ is absolutely exposing and evenly distributed, and $\Vert B- A_N\Vert<\varepsilon/3$. %(en consid\'erant $B$ et $A_N$ comme op\'erateurs sur 
%$\ell_p$). Est-ce que c'est bien vrai?

\smallskip  By Proposition \ref{Seconde proposition Localisation vp lp}, one can find $\gamma>0$ such that for any $T\in\bbx$, the following implication holds true:
\[ \Vert P_MTP_M-B\Vert<\gamma\implies \Vert P_MT(I-P_M)\Vert <\varepsilon/6.\]
We assume without loss of generality that $\gamma<\varepsilon/6$. Then $ \Vert P_MT-B\Vert <\gamma+\varepsilon/6<\varepsilon/3$ if $\Vert P_MTP_M-B\Vert<\gamma$; and since $P_MTP_M-B=P_M(TP_M-B)$, we see that
\[ \Vert TP_M-B\Vert<\gamma\implies \Vert P_MT-B\Vert <\varepsilon/3.\]

\smallskip Now, set
\[ \mathcal U:=\bigl\{ T\in\bbx;\; \Vert TP_M-B\Vert<\gamma\bigr\}.\]

\noindent The set $\mathcal U$ is \sot$\,$-$\,$open, and $\mathcal U\neq\emptyset$ since $B\in\mathcal U$. Let us show that $\mathcal U\subseteq \mathcal O$. It is enough to check that if $T\in\mathcal U$, then $\Vert (T-T_0)P_K\Vert <\varepsilon$ and $ \Vert (T^*-T_0^*)P_K^*\Vert <\varepsilon$. 

\smallskip On the one hand, we have \[\Vert (T-T_0)P_K\Vert \leq \Vert (T-A_N)P_K\Vert +\Vert (A_N-T_0)P_K\Vert<\Vert (T-A_N)P_K\Vert+\varepsilon/3.\] Moreover, $\Vert (T-A_N)P_K\Vert \leq\Vert (T-B)P_K\Vert +\Vert (B-A_N)P_K\Vert<\Vert (T-B)P_K\Vert +\varepsilon/3$, and $\Vert (T-B)P_K\Vert =\Vert (TP_M-B)P_K\Vert <\gamma<\varepsilon/3$. Hence $\Vert (T-T_0)P_K\Vert<\varepsilon$.

\smallskip
On the the other hand, $\Vert (T^*-T_0^*)P_K^*\Vert < \Vert (T^*-A_N^*)P_K^*\Vert +\varepsilon/3< \Vert (T^*-B^*)P_K^*\Vert +2\varepsilon/3$. Moreover, since $\Vert TP_M-B\Vert<\gamma$, we have $\Vert P_MT-B\Vert <\varepsilon/3$ and hence $\Vert T^*P_M^*-B^*\Vert <\varepsilon/3$. So we get
$\Vert (T^*-B^*)P_K^*\Vert=\Vert (T^*P_M^*-B^*)P_K^*\Vert <\varepsilon/3$, and hence $\Vert (T^*-T_0^*)P_K^*\Vert <\varepsilon$.
\epf

\subsection{Proofs of Propositions \ref{Premiere proposition Localisation vp lp} and \ref{Seconde proposition Localisation vp lp}} Since it is the shortest, we start with the proof of Proposition \ref{Seconde proposition Localisation vp lp}.

\begin{proof}[Proof of Proposition \ref{Seconde proposition Localisation vp lp}]
Let $\delta >0$ (how small $\delta $ needs to be will be determined during the proof) and let $T\in\b_{1}(X)$ be such that $\Vert P_{M}TP_{M}-B\Vert<\delta $. Let 
$x_1\in S_{E_{M}}$ be a norming vector for $B$. For every non-zero vector $y\in F_{M}=[\,e_{i}\,;\,i>M]$ and every scalar $\mu >0$, we have by Lemma \ref{Lemme Kan} that
\begin{align*}
\vert \pss{e_{j}^{*}}{Tx_{1}+\mu Ty}\vert^{p}+\vert\pss{e_{j}^{*}}{Tx_{1}-\mu Ty}\vert^{p}\ge 2\,\vert&\pss{e_{j}^{*}}{Tx_{1}}\vert^{p}\\&+p\,\mu ^{2}\,\vert \pss{e_{j}^{*}}{Tx_{1}}\,\vert ^{p-2}\vert\pss{e_{j}^{*}}{Ty}\,\vert ^{2}
\end{align*}
for every $j\ge 0$, which can be rewritten as
\begin{align*}
 p\mu ^{2}\, \vert\pss{e_{j}^{*}}{Tx_{1}}\vert^{p-2}\,\vert\pss{e_{j}^{*}}{Ty}\vert^{2}\le
 \vert \,\pss{e_{j}^{*}}{Tx_{1}+\mu Ty}\,\vert ^{p}+\vert \,&\pss{e_{j}^{*}}{Tx_{1}-\mu Ty}\,\vert ^{p}\\
 &-2\,\vert \,\pss{e_{j}^{*}}{Tx_{1}}\,\vert ^{p}\!.
\end{align*}
Summing over $j\ge 0$ yields that 
\begin{align*}
 p\mu ^{2}\,\sum_{j\ge 0}\,\vert\pss{e_{j}^{*}}{Tx_{1}}\vert^{p-2}\,\vert\,\pss{e_{j}^{*}}{Ty}\,\vert ^{2}&\le\Vert T(x_{1}+\mu y)\Vert^{p}+\Vert T(x_{1}-\mu y)\Vert^{p}-2\,\Vert Tx_{1}\Vert^{p}\\
&\le\Vert x_{1}+\mu y\Vert^{p}+\Vert x_{1}-\mu y\Vert^{p}-2\Vert Tx_{1}\Vert^{p}.
\end{align*}
Since $\Vert Tx_{1}\Vert>1-\delta $ and since the vectors $x_{1}$ and $y$ have disjoint supports, it follows that 
\[
 p\,\mu ^{2}\sum_{j\ge 0}\,\vert \pss{e_{j}^{*}}{Tx_{1}}\vert^{p-2}\, \vert\pss{e_{j}^{*}}{Ty}\vert^{2}\le 2\,\bigl(1+\mu ^{p}\,\Vert y\Vert^{p}-(1-\delta )^{p}\bigr).
\]
Let us temporarily write 
\[
\alpha :=\sum_{j\ge 0}\,\vert\pss{e_{j}^{*}}{Tx_{1}}\vert^{p-2}\, \vert\pss{e_{j}^{*}}{Ty}\vert^{2}.
\]
Then $\Vert y\Vert^{p}\mu ^{p}-(p\alpha /2)\mu ^{2}+1-(1-\delta )^{p}\ge 0$ for every $\mu >0$. Optimizing  this inequality with respect to $\mu$, \mbox{\it i.e.} taking $\mu :=\left(\frac{\alpha }{\Vert y\Vert^{p}}\right)^{1/(p-2)}$, we obtain
\[
\dfrac{\alpha ^{\,p/(p-2)}}{\Vert y\Vert^{\,2p/(p-2)}}\,\biggl (1-\dfrac{p}{2} \biggr)+\bigl (1-(1-\delta )^{p} \bigr)\ge 0;  
\]
which can be rewritten as
\[
\dfrac{p-2}{2}\,\Bigl (\, \dfrac{\alpha }{\Vert y\Vert^{2}}\,\Bigr)^{p/(p-2)}\le1-(1-\delta )^{p}\le p\delta \]
\textit{i.e.}
\begin{equation}\label{trucmachin} 0\le \alpha \le \Vert y\Vert^{2}\cdot K_{p}\cdot\delta ^{(p-2)/p},\quad \textrm{where}\quad 
K_{p}:={\biggl (\dfrac{2p}{p-2} \biggr)^{(p-2)/p}}\cdot
\end{equation}
Our assumption on the operator $B$ implies that 
\[
\gamma :=\min_{j\in[0,M]}\vert \pss{e_{j}^{*}}{Bx_{1}}\vert>0.
\]
Let $\delta $ be so small that every operator $T\in\b_{1}(E_{M})$ with $\Vert P_{M}TP_{M}-B\Vert<\delta $ satisfies 
$\min_{j\in[0,M]}\vert\pss{e_{j}^{*}}{Tx_{1}}\vert>\gamma /2$. For every $j\in[0,M]$, we have by (\ref{trucmachin}) that
\[
\vert\pss{e_{j}^{*}}{Ty}\vert^{2}\le\Vert y\Vert^{2}\cdot K_{p}\cdot\biggl (\dfrac{2}{\gamma } \biggr)^{p-2}\cdot\,\delta ^{(p-2)/p} ,\]
\textit{i.e.}
\[ \vert \pss{e_{j}^{*}}{Ty}\vert^{p}\le\Vert y\Vert^{p}\cdot K_{p}^{p/2}\cdot\biggl (\dfrac{2}{\gamma } \biggr)^{\frac{p(p-2)}{2}}\cdot\,\delta ^{(p-2)/2}. \]
 {Summing over $j\in[0,M]$, we obtain that}
 \[ \Vert P_{M}Ty\Vert\le\Vert y\Vert\cdot K_{p}^{1/2}\cdot M^{1/p}\cdot \biggl (\dfrac{2}{\gamma } \biggr)^{\frac{p-2}{2}}\cdot\,\delta ^{\frac{p-2}{2p}}.
\]
If we choose $\delta >0$ so small that $\ds K_{p}^{1/2}\cdot M^{1/p}\cdot(2/\gamma  )^{\frac{p-2}{2}}\cdot\,\delta ^{\frac{p-2}{2p}}<\varepsilon $ (this condition depends on $p$, $M$, and $B$, but not on $T$), we obtain that $\Vert P_{M}T(I-P_{M})\Vert<\varepsilon $, which is exactly the inequality we were looking for.
\end{proof}
\par\medskip 
It remains to prove Proposition \ref{Premiere proposition Localisation vp lp}, and this will be a bit longer. Recall that for each $N\ge 1$, we denote by ${\mathcal{E}}_1(\en)$ the set of absolutely exposing operators of
$\b_{1}(\en)$. %, and, for each $k\ge 1$, by $\a_{1,k}(\en)$ the set of operators $A\in\mathfrak{E}_1(\en)$ satisfying Property $(\ast)_{k,N}$. 
We begin by a series of rather elementary lemmas which shed some light on the behaviour of absolutely exposing contractions.
\begin{lemma}\label{Lemma A bis} 
The set ${\mathcal{E}}_1(\en)$ is dense in $\b_{1}(\en)$.
 \end{lemma}
\begin{proof}
 Let $A\in\b_{1}(\en)$ with $A\neq 0$ and $\Vert A\Vert<1$. Let $x_{0}\in S_{\en}$ be such that $\Vert Ax_{0}\Vert=\Vert A\Vert$, and let $x_{0}^{*}\in\en$ be such that $\Vert x_{0}\Vert=\pss{x_{0}^{*}}{x_{0}}=1$. Let $R_{0}$ be the rank 1  operator on $\en$ defined by $R_{0}(x):=\pss{x_{0}^{*}}{x}\,Ax_{0}$, $x\in\en$;  and for any $\delta >0$, let 
 %$A_{\delta }$ be the operator 
 $A_{\delta }:=A+\delta R_{0}$. Then $A_{\delta }x_{0}=(1+\delta )Ax_{0}$, and $\Vert A_{\delta }x_{0}\Vert=(1+\delta )\Vert Ax_{0}\Vert=(1+\delta )\Vert A\Vert$. Since $\Vert R_{0}\Vert=\Vert A\Vert$, $\Vert A_{\delta }\Vert\le (1+\delta )\Vert A\Vert$, and hence $\Vert A_{\delta }\Vert=(1+\delta )\Vert A\Vert$. So $\Vert A_{\delta }x_{0}\Vert=\Vert A_{\delta }\Vert$. Moreover, if $x\in S_{\en}$ is such that $\Vert A_{\delta }x\Vert=\Vert A_{\delta }\Vert$, then 
 $\Vert Ax+\delta \,\pss{x_{0}^{*}}{x}\,Ax_{0}\Vert=(1+\delta )\Vert A\Vert$. So $\pss{x_{0}^{*}}{x}\neq 0$ and 
\begin{align*}
 (1+\delta )\Vert A\Vert=\Vert Ax+\delta \pss{x_{0}^{*}}{x}\,Ax_{0}\Vert&\le \Vert A\Vert\,\Vert x+\delta\, \pss{x_{0}^{*}}{x}\, x_{0}\Vert\\
 &\le\Vert A\Vert\,\bigl (\Vert x\Vert+\delta \,\vert\pss{x_{0}^{*}}{x}\vert\,\Vert x_{0}\Vert \bigr)\le \Vert A\Vert\,(1+\delta ). 
\end{align*}
Since $\Vert A\Vert>0$, it follows that $\Vert x+\delta\, \pss{x_{0}^{*}}{x}\,x_{0}\Vert=\Vert x\Vert+\Vert\delta\, \pss{x_{0}^{*}}{x}\, x_{0}\Vert$, and since $\pss{x_{0}^{*}}{x}\neq 0$, this implies that $x$ is colinear to $x_{0}$, by strict convexity of the $\ell_p\,$-$\,$norm. So $A_{\delta }$ is absolutely exposing. Given $\varepsilon >0$, one can choose $\delta >0$ so small that $\Vert A_{\delta }\Vert<1$ and $\Vert A-A_{\delta }\Vert<\varepsilon $, and this proves that ${\mathcal{E}}_1(\en)$ is dense in $\b_{1}(\en)$.
 \end{proof}
\begin{lemma}\label{Lemma A}
 Let $A\in{\mathcal{E}}_1(\en)$, and let $x_{0}\in S_{\en}$ be a norming vector for $A$. For every $\varepsilon >0$, there exists $\delta >0$ such that any vector $x\in S_{\en}$ such that 
 $\Vert Ax\Vert>(1-\delta )\Vert A\Vert$ satisfies $\emph{dist}(x,\T x_{0})<\varepsilon $.
\end{lemma}
\begin{proof}
 Fix $\varepsilon >0$ and, towards contradiction, suppose that for every $n\ge 1$, there exists  a vector $x_{n}\in S_{\en}$ such that $\Vert Ax_{n}\Vert>(1-2^{-n})\Vert A\Vert$ and $\textrm{dist}(x_{n},\T x_{0})\ge \varepsilon $. Without loss of generality, we can suppose that the sequence $(x_{n})_{n\ge 1}$ converges to a vector $x_{\infty}\in S_{\en}$. Then $\Vert Ax_{\infty}\Vert=\Vert A\Vert$ and thus $x_{\infty}\in \T x_{0}$, which is a contradiction.
\end{proof}
\begin{lemma}\label{Lemma B}
 Let $A\in{\mathcal{E}}_1(\en)$, and let $x_{0}\in S_{\en}$ be a norming vector for $A$. For every $\varepsilon >0$, there exists $\gamma >0$ with the following property:
 for every operator $B\in\b_{1}(\en)$ with $\Vert A-B\Vert<\gamma $ and any norming vector $x\in S_{\en}$ for $B$, we have $\emph{dist}(x,\T x_{0})<\varepsilon $.
\end{lemma}
\begin{proof}
 Given \(A\in\b_{1}(E_{N})\) and $\varepsilon >0$, let $\delta >0$ be given by Lemma \ref{Lemma A}: if $x\in S_{\en}$ is such that $\Vert Ax\Vert>(1-\delta )\Vert A\Vert$, then $\textrm{dist}(x,\T x_{0})<\varepsilon $. Let $\gamma >0$ be so small that $\Vert A\Vert-2\gamma \ge(1-\delta )\Vert A\Vert$. If $B\in\b_{1}(E_{N})$ is such that $\Vert B-A\Vert<\gamma $, and if $x\in S_{\en}$ is a norming vector for $B$, then $\bigl \vert \, \Vert B\Vert-\Vert Ax\Vert\,\bigr\vert <\gamma$  , and since $\bigl \vert \, \Vert B\Vert-\Vert A\Vert\,\bigr \vert <\gamma $, one has $\bigl \vert \, \Vert A\Vert-\Vert Ax\Vert\,\bigr\vert <2\gamma  $. Hence
 $\Vert Ax\Vert>\Vert A\Vert-2\gamma \ge(1-\delta )\Vert A\Vert$, and so $\textrm{dist}(x,\T x_{0})<\varepsilon $. 
\end{proof}

Some of the arguments in the proofs to come will involve a duality mapping 
$\mathbf J:\ell_{p}\to\ell_{p'}$, where $1/p+1/p'=1$. Recall that we are assuming that $1<p<\infty$. The map $\mathbf J$ is defined as follows: if $x\in X=\ell_p$, then $\mathbf J(x)$ is the unique element of $X^*=\ell_{p'}$ such that  $\pss{\mathbf J(x)}{x}=\Vert \mathbf J(x)\Vert \,\Vert x\Vert=\Vert x\Vert^{p}$ (there are several natural ways of ``normalizing'' the duality mapping when defining it on the whole space rather than just on the unit sphere). %, and $\Vert Jx\Vert=\Vert x\Vert^{p/p'}$. 
Explicitely, we have:
\[
\mathbf J(x)=\sum_{j\ge 1}\ba{\pss{e_{j}^{*}}{x}}\cdot \bigl\vert\pss{e_{j}^{*}}{x}\bigr\vert^{p-2}\, e_{j}^{*},\quad x\in X. 
\]
Recall also that an absolutely exposing operator $B\in{\mathcal{E}}_1(E_N)$ is said to be evenly distributed if for every norming vector $x_{1}\in S_{\en}$ for $B$, we have $\pss{e_{j}^{*}}{x_{1}}\neq 0$ and $\pss{e_{j}^{*}}{Bx_{1}}\neq 0$ for all $j\in [0,N]$.
\begin{proposition}\label{Proposition E}
 Let $A\in{\mathcal{E}}_1(\en)$ be such that $A\neq 0$ and $\Vert A\Vert<1$. For every $\varepsilon >0$, there exists an operator $B\in{\mathcal{E}}_1(\en)$ such that $\Vert B-A\Vert<\varepsilon $ and $B$ is evenly distributed. %(\mbox{\it i.e.} for every norming vector $x_{1}\in S_{\en}$ for $B$, and for every $i\in[0,N]$, $\pss{e_{i}^{*}}{x_{1}}\neq 0$ and $\pss{e_{i}^{*}}{Ax_{1}}\neq 0$).
\end{proposition}
\begin{proof}
 The proof will proceed in two steps.
 \par\medskip
 \noindent
 \textbf{Step 1.} \emph{We prove that given $A\in{\mathcal{E}}_1(\en)$ with $A\neq 0$ and $\Vert A\Vert<1$, and given $\varepsilon >0$, there exists $C\in{\mathcal{E}}_1(\en)$ with $\Vert C-A\Vert<\varepsilon $ such that for any norming vector $x_{1}\in S_{\en}$ for $C$, and for every $j\in [0,N]$, we have $\pss{e_{j}^{*}}{x_{1}}\neq 0$.}
 \par\medskip 
 Let $x_{0}\in S_{\en}$ be any norming vector for $A$, and let $J:=\{j\in[0,N]\;;\;\pss{e_{j}^{*}}{x_{0}}\neq 0\}$. If $J=[0,N]$, there is nothing to prove, so suppose that $J\neq[0,N]$. Observe that since $x_{0}$ is supported on $J$ and $\Vert Ax_{0}\Vert=\Vert A\Vert$, we have $\Vert\,A_{| E_{J}}\,\Vert=\Vert A\Vert$, where $E_{J}=[\,e_{j}\;;\;j\in J]$. Consider now the sets
 \begin{align*}
  J'&:=\bigl\{\,j\in [0,N]\;;\;\forall\,l\in[0,N],\; \exists\,s\in\C\setminus\{0\}\ \textrm{such that}\ \Vert A+s\, e_{l}\otimes e_{j}^{*}\Vert\le\Vert A\Vert\bigr\}
  \intertext{and}
  K&:=[0,N]\setminus(J\cup J').
 \end{align*}
We first observe that $J'$ is actually included in $J$:
\begin{lemma}\label{Lemma EE}
 For every $j\in J'$, we have $\pss{e_{j}^{*}}{x_{0}}\neq 0$.
\end{lemma}
\begin{proof}[Proof of Lemma \ref{Lemma EE}]
Let $j\in J'$, and suppose that $\pss{e_{j}^{*}}{x_{0}}= 0$. Since $x_{0}\in S_{\en}$ is a norming vector for $A$, the function $\phi :\R\to\R$ defined by 
\[
\phi (t)=\dfrac{\Vert A(x_{0}+te_{j})\Vert^{p}}{\Vert x_{0}+te_{j}\Vert^{p}}
\]
attains its maximum at $t=0$. As $\pss{e_{j}^{*}}{x_{0}}= 0$, $\Vert x_{0}+te_{j}\Vert^{p}=1+t^{p}$, and so
\[
\phi (t)=\dfrac{\Vert A(x_{0}+te_{j})\Vert^{p}}{1+t^{p}},\quad t\in\R.
\]
Since $p>1$, $\phi $ is differentiable on $\R$, and thus $\phi '(0)=0$. Using the fact that for any complex numbers $a$ and $b$, the function \(\varphi _{a,b}\) defined by $\varphi _{a,b}(t)=\vert a+tb\vert ^{p}$, \(t\in\R\), is differentiable on $\R$ 
 with $\varphi '_{a,b}(0)=p\cdot\vert a\vert^{p-2}\cdot\,\mathfrak{Re} (a\,\ba{b})$, we obtain that 
 \[
\phi (0)=p\cdot\sum_{k=0}^{N}\vert \,\pss{e_{k}^{*}}{Ax_{0}}\,\vert ^{p-2}\cdot\,\mathfrak{Re}\bigl (\pss{e_{k}^{*}}{Ax_{0}}\ \ba{\pss{e_{k}^{*}}{Ae_{j}}}\, \bigr)=p\cdot\mathfrak{Re}\bigl(\pss{\mathbf J(Ax_{0})}{Ae_{j}}\bigr)=0. 
\]
Applying the same argument to the function $\psi :\R\to\R$ defined by 
\[
\psi (t)=\dfrac{\Vert A(x_{0}+ite_{j})\Vert^{p}}{1+t^{p}},\quad t\in\R,
\]
we obtain that $\psi '(0)=p\cdot\mathfrak{Im}\,\bigl(\pss{\mathbf J(Ax_{0})}{Ae_{j}}\bigr)=0$. Hence $\pss{\mathbf J(A{x_{0}})}{Ae_{j}}=0$. In other words, the vector $Ae_{j}$ belongs to the kernel of the functional $\mathbf J(Ax_{0})_{| E_N}$ , which we denote by $H$. Since $Ax_{0}\neq 0$, $H$ is a hyperplane of $\en$.
\par\smallskip
Fix now $l\in[0,N]$. Since $j\in J'$, there exists $s\neq 0$ such that $\Vert A+s\,e_{l}\otimes e_{j}^{*}\Vert\le \Vert A\Vert$. Setting $B=A+s\,e_{l}\otimes e_{j}^{*}$, we have $\Vert B\Vert\le \Vert A\Vert$ and $Bx_{0}=
Ax_{0}+s\,\pss{e_{j}^{*}}{x_{0}}\,e_{l}=Ax_{0}$ since we have supposed that \(\pss{e_{j}^{*}}{x_{0}}=0\). Hence $\Vert B\Vert=\Vert A\Vert$ and $x_{0}$ is a norming vector  for the operator $B$. Reasoning as above, we obtain that the vector $Be_{j}$ belongs to the kernel of the functional $\mathbf J(B x_{0})_{| E_N}$, which is $H$. So $B e_{j}=Ae_{j}+s\,e_{l}\in H$. Since $Ae_{j}\in H$ and $s\neq 0$, it follows that $e_{l}$ belongs to $H$. This being true for every $l\in[0,N]$, we have $H=\en$, which is a contradiction. So $\pss{e_{j}^{*}}{x_{0}}\neq 0$, and Lemma \ref{Lemma EE} is proved.
\end{proof}

We have thus shown that $J'\subseteq J$, and hence we can write the set $[0,N]$ as the disjoint union of the sets $J$ and $K$.

\begin{lemma}\label{Lemma EEE}
 Suppose that $K\neq\emptyset$. Let $k\in K$ and $\gamma >0$. There exists an operator $D\in{\mathcal{E}}_1(\en)$ with $D\neq 0$, $\Vert D\Vert<1$ and $\Vert D-A\Vert<\gamma $ such that, for any norming vector $x_{2}\in S_{\en}$ for $D$, we have $\pss{e_{j}^{*}}{x_{2}}\neq 0$
for all $j\in J\cup\{k\}$.
 \end{lemma}
\begin{proof}[Proof of Lemma \ref{Lemma EEE}]
Since $k$ belongs to $K$, it does not belong to $J'$, and hence there exists $l\in[0,N]$ such that $\Vert A+s\, e_{l}\otimes e_{k}^{*}\Vert>\Vert A\Vert$ for every $s >0$. For each $s > 0$, set $B_{s }:=A+s \,e_{l}\otimes e_{k}^{*}$. Then $\Vert B_{s }-A\Vert=s $, and the operator  induced by $B_{s }$ on $E_{\,[0,N]\setminus \{k\}}$ coincides with $A$ on $E_{\,[0,N]\setminus \{k\}}$. Since $\Vert B_{s }\Vert>\Vert A\Vert$, it follows that if $x\in S_{\en}$ is a norming vector for $B_{s }$, then $\pss{e_{k}^{*}}{x}\neq 0$.
\par\smallskip
On the other hand, since $\pss{e_{j}^{*}}{x_{0}}\neq 0$ for every $j\in J$, Lemma \ref{Lemma B} implies that if $s $ is sufficiently small, any norming vector $x\in S_{\en}$ for $B_{s }$ is such that $\pss{e_{j}^{*}}{x}\neq 0$ for every $j\in J$. Combining this with the observation above, we deduce that if $s \neq 0$ is sufficiently small, then any norming vector \(x\in S_{E_{N}}\) for \(B_{s}\) is such that $\pss{e_{j}^{*}}{x}\neq 0$ for every $j\in J\cup \{k\}$. By Lemmas \ref{Lemma A bis} and \ref{Lemma B}, we can now choose $B\in{\mathcal{E}}_1(\en)$ so close to $B_{s}$ (with $s $ small enough chosen first) that any norming vector $x_{2}\in S_{\en}$ for $B$ satisfies $\pss{e_{j}^{*}}{x_{2}}\neq 0$ for every $j\in J\cup\{k\}$.
\end{proof}
We now have all the tools handy to finish the proof of Step 1 of Proposition \ref{Proposition E}. Suppose that $K$ has cardinality $r\ge 1$, and write $K=\{k_{1},\dots,k_{r}\}$. Choose $0<\gamma <\varepsilon /r$, and apply Lemma \ref{Lemma EEE} $r$ times successively, to obtain operators $D_{i}\in{\mathcal{E}}_{1}(\en)$, $i\in[1,r]$, with $D_{i}\neq 0$, $\Vert D_{i}\Vert=1$, $\Vert D_{i+1}-D_{i}\Vert<\gamma $ for every $i\in [1,r-1]$, and such that for every norming vector $y_{i}\in S_{\en}$ for the operator $D_{i}$, one has that $\pss{e_{j}^{*}}{y_{i}}\neq 0$ for every $j\in J\cup\{k_{1},\dots,k_{i}\}$. Then, the operator $C:=D_{r}$ belongs to ${\mathcal{E}}_1(\en)$, satisfies $\Vert C-A\Vert<\varepsilon $, and whenever $x\in S_{\en}$ is a norming vector for $C$, we have $\pss{e_{j}^{*}}{x}\neq 0$ for every $j\in J\cup K$. Since $J\cup K=[0,N]$ by Lemma \ref{Lemma EE}, this proves the statement we were looking for.
\par\medskip
\noindent\textbf{Step 2.} \emph{We now prove the statement of Proposition \ref{Proposition E}, namely that given $A\in{\mathcal{E}}_1(\en)$ with $A\neq 0$ and $\Vert A\Vert<1$, and given $\varepsilon >0$, there exists $B\in{\mathcal{E}}_1(\en)$ with $\Vert A-B\Vert<\varepsilon $ such that for any norming vector $x_{2}\in S_{\en}$ for $B$, and for every $j\in [0,N]$, we have $\pss{e_{j}^{*}}{x_{2}}\neq 0$ and $\pss{e_{j}^{*}}{Bx_{2}}\neq 0$.} 
\par\medskip
Let $C\in{\mathcal{E}}_1(\en) $ be given by Step 1, with $C\neq 0$, $\Vert C\Vert<1$, $\Vert C-A\Vert<\varepsilon /2$, and $\pss{e_{j}^{*}}{x}\neq 0$ for every norming vector $x\in S_{\en}$ for $C$ and all $j\in[0,N]$. We also fix a norming vector $x_1\in S_{E_N}$ for $C$.

\par\smallskip
Our strategy  is now to apply the result proved in Step 1 to the operator $C^{*}$ acting on $\en^{*}=[e_0^*,\dots ,e_N^*]\subseteq \ell_{p'}(\Z_+)$, where $1/p+1/p'=1$. Since $1<p<\infty$, we have $1<p'<\infty$. We first observe that the operator $C^{*}$ belongs to ${\mathcal{E}}_1(\en^{*})$, and that its norming vectors are the unimodular multiples of the vector $y_{1}^{*}:=\mathbf J(Cx_{1})\,/\,\Vert C\Vert^{p/p'}$. Indeed, if $x^{*}\in S_{\en^{*}}$ is such that $\Vert C^{*}x^{*}\Vert=\Vert C^{*}\Vert$, there exists $x\in S_{\en}$ such that 
$\pss{C^{*}x^{*}}{x}=\Vert C^{*}\Vert$, \mbox{\it i.e.}  $\pss{x^{*}}{Cx}=\Vert Cx\Vert$. Hence $x$ is a norming vector for $C$, and thus $x$ belongs to $\T x_{1}$, and $\vert \pss{x^{*}}{Cx_{1}}\vert =\Vert Cx_{1}\Vert=\Vert C\Vert$. It follows that $x^{*}$ is a unimodular multiple of the vector $\mathbf J(Cx_{1})\,/\,\Vert C\Vert^{p/p'}=y_{1}^{*}$, and so $C^{*}$ is absolutely exposing. Since moreover $\Vert C^{*}\Vert=1$ and $1<p'<\infty$, it is legitimate to apply the result proved in Step 1: there exists an operator $D\in{\mathcal{E}}_1(\en^{*})$ with $\Vert D-C^{*}\Vert<\varepsilon /2$ such that $\pss{y^{*}}{e_{j}}\neq 0$ for every norming vector $y^{*}\in S_{\en^{*}}$ for the operator $D$ and for every $j\in[0,N]$. Let us fix a norming vector $y_0^*$ for $D$.
\par\smallskip
Set $B:=D^{*}\in\b_{1}(\en)$. Then $\Vert B-A\Vert<\varepsilon $, and the same argument as above shows that $B\in{\mathcal{E}}_1(\en)$, and that for any norming vector $x_{2}\in S_{\en}$ for $B$, the vector $\mathbf J(Bx_{2})\,/\,\Vert B\Vert^{p/p'}$ is a unimodular multiple of $y_{0}^{*}$. So $\mathbf J(Bx_{2})$ is a non-zero multiple of $y_{0}^{*}$, and $\pss{\mathbf J(Bx_{2})}{e_{j}}\neq 0$ for every $j\in [0,N]$.
It follows from the expression of \(\mathbf{J}(Bx_{2})\) that
$\pss{e_{j}^{*}}{\mathbf J(Bx_{2})}\neq 0$ for all $j\in[0,N]$.
Moreover, Lemma \ref{Lemma B} implies that if $\varepsilon >0$ is sufficiently small, the operator $B$ is so close to $C$ that $\pss{e_{j}^{*}}{x_{2}}\neq 0$ for all $j\in[0,N]$. Hence $B$ has all the required properties, and Proposition \ref{Proposition E} is proved.
\end{proof}
We are now ready to prove Proposition \ref{Premiere proposition Localisation vp lp}.
\begin{proof}[Proof of Proposition \ref{Premiere proposition Localisation vp lp}]
Let $A\in\b_{1}(\en)$. By Lemma \ref{Lemma A bis} and Proposition \ref{Proposition E}, we can find an operator $A'$ extremely close to it  which is absolutely exposing, evenly distributed, and is such that $\Vert A'\Vert<1$. So we may and do assume that $A$ satisfies these additional properties.
\par\medskip
Write $E_{2N+1}$ as $E_{2N+1}=\en\oplus G_{N}$, where $G_{N}:=[\,e_{j}\;;\;N+1\le j\le 2N+1\,]$. Define $S_{N}:\en\to G_{N}$ by setting 
\[
S_{N}x:=\sum_{i=0}^{N}e_{i}^{*}(x)e_{N+1+i}\quad \textrm{for every}\ x\in\en.
\]
For every parameters $\eta ,\delta >0$, we define an operator $B_{\eta ,\delta }\in\b(E_{2N+1})$ in the following way:
\[
B_{\eta ,\delta }(x+S_{N}y):=A(x+\delta y)+\eta \,S_{N}\,A(x+\delta y),\quad x,y\in\en.
\]
In matrix form with respect to the decomposition 
$E_{2N+1}=\en\oplus G_{N}$,
\[
B_{\eta ,\delta }=
\begin{pmatrix}
 A&\delta A\\ \eta A&\eta \delta A
\end{pmatrix}\cdot
\]
For every $x,y\in \en$, we have
\begin{align*}
 \Vert B_{\eta ,\delta }(x+S_{N}y)\Vert&%=\Bigl (\,\Vert A(x+\delta y)\Vert^{p}+\eta ^{p}\,\Vert A(x+\delta y)\Vert^{p} \,\Bigr)^{\,1/p} 
 =(1+\eta ^{p})^{\,1/p}\,\Vert A(x+\delta y)\Vert\\
 &\leq (1+\eta ^{p})^{\,1/p}\, \Vert A\Vert\, \bigl(\Vert x\Vert +\delta\,\Vert y\Vert\bigr)\\
 &\le (1+\eta ^{p})^{\,1/p}\,(1+\delta ^{p'})^{\,1/p'}\,\Vert A\Vert\cdot\Bigl (\,\Vert x\Vert^{p}+\Vert y\Vert^{p} \,\Bigr)^{\,1/p}\\
 &= (1+\eta ^{p})^{\,1/p}\,(1+\delta ^{p'})^{\,1/p'}\,\Vert A\Vert\cdot \Vert x+S_N y\Vert.
\end{align*}
Hence $\Vert B_{\eta ,\delta }\Vert\le (1+\eta ^{p})^{\,1/p}\,(1+\delta ^{p'})^{\,1/p'}\,\Vert A\Vert$. We now claim that $\Vert B_{\eta ,\delta }\Vert$ is exactly equal to $(1+\eta ^{p})^{\,1/p}\,(1+\delta ^{p'})^{\,1/p'}\,\Vert A\Vert$. Indeed, let $x_{0}\in S_{\en}$ be a norming vector for $A$, and let $u$ be the vector of 
$E_{2N+1}$ defined as $u:=x_{0}+\delta ^{p'-1}S_{N}x_{0}$. We set also $u_{0}:=u\,/\,\Vert u\Vert$. We have
\begin{align*}
 \Vert B_{\eta ,\delta }(u)\Vert&=(1+\eta ^{p})^{\,1/p}\,\Vert A(x_{0}+\delta ^{p'}x_{0})\Vert=(1+\eta ^{p})^{\,1/p}\,(1+\delta ^{p'})\,\Vert Ax_{0}\Vert\\
 &=(1+\eta ^{p})^{\,1/p}\,(1+\delta ^{p'})\,\Vert A\Vert
\end{align*}
since $x_{0}\in S_{\en}$ is norming for $A$. Also 
$\Vert u\Vert=(1+\delta ^{(p'-1)p})^{\,1/p}=(1+\delta ^{p'})^{\,1/p}$, and hence
\[
\Vert B_{\eta ,\delta }(u_{0})\Vert=(1+\eta ^{p})^{\,1/p}\,(1+\delta ^{p'})^{\,1/p'}\,\Vert A\Vert.%=\Vert B_{\eta ,\delta }\Vert.
\]
This proves our claim. Moreover, we have also shown that 
\[ u_0=\frac{x_{0}+\delta ^{p'-1}S_{N}x_{0}}{(1+\delta^{p'})^{\, 1/p}}\]
is a norming vector for $B_{\eta,\delta}$.
\par\medskip 
Let us now check that $B_{\eta ,\delta }$ is absolutely exposing. Let $v\in S_{E_{2N+1}}$ be such that 
$\Vert B_{\eta ,\delta }v\Vert=\Vert B_{\eta ,\delta }\Vert$. Write $v=x+S_{N}y$ with $x,y\in\en$ and \(\Vert x\Vert^{\,p}+\Vert y\Vert^{\,p}=1\).  Then  $\Vert B_{\eta ,\delta }v\Vert=(1+\eta ^{p})^{\,1/p}\,\Vert A(x+\delta y)\Vert$, so we have
\[
\Vert A(x+\delta y)\Vert=\bigl(\,1+\delta ^{p'}\,\bigr)^{1/p'}\,\Vert A\Vert.
\]
It follows from this equality that $x$ and $y$ are necessarily non-zero. Since 
\[
\Vert x+\delta y\Vert\le \Vert x\Vert+\delta \Vert y\Vert\le \bigl (\,1+\delta ^{p'}\, \bigr)^{1/p'}\,\bigl (\,\Vert x\Vert^{p}+\Vert y\Vert^{p} \,\bigr)^{1/p}=  \bigl (\,1+\delta ^{p'}\, \bigr)^{1/p'},
\]
it also follows that $\Vert x+\delta y\Vert=\Vert x\Vert+\delta \Vert y\Vert=(1+\delta ^{p'})^{\,1/p'}$. So there exists $\alpha >0$ such that $y=\alpha x$. Finally, observe that $\frac{x+\delta y}{\Vert x+\delta y\Vert}$ is a norming vector for $A$, and hence is a unimodular multiple of $x_{0}$. We thus have $(1+\alpha \delta )x=(1+\delta ^{p'})^{\,1/p'}\lambda x_{0}$ for some $\lambda \in\T$, so that
\[
v=\dfrac{(1+\delta ^{p'})^{\,1/p'}}{1+\alpha \delta }\,\lambda \,(x_{0}+\alpha S_{N}x_{0}).
\]
Since $\Vert v\Vert=1$, we have $1+\alpha \delta =(1+\delta ^{p'})^{\,1/p'}\,(1+\alpha ^{p})^{\,1/p}$. By the equality case in H\"{o}lder's inequality, it follows that the vectors 
$\binom{1}{\delta ^{p'}}$ and $\binom{1}{\alpha ^{p}}$ are colinear in $\R^{2}$, which implies that $\alpha =\delta ^{p'-1}$. Looking back at the definition of $u_0$, we thus see that $v$ is a unimodular multiple of  $u_{0}$. Hence $B_{\eta ,\delta }$ is absolutely exposing, whatever the choices of the parameters $\eta $ and $\delta $.
\par\medskip
The next step is to verify that $B_{\eta,\delta}$ is evenly distributed. With $u= x_0+\delta^{p'-1} S_N x_0$ as above, it is enough to show that $\pss{e_{j}^{*}}{u}\neq 0$ and $\pss{e_{j}^{*}}{B_{\eta,\delta}u }\neq 0$ for all $j\in [0, 2N+1]$. We have $\pss{e_{j}^{*}}{u}=\pss{e_{j}^{*}}{x_{0}}$ if $j\in[0,N]$ and $\pss{e_{j}^{*}}{u}=\delta ^{p-1}\pss{e_{j-N-1}^{*}}{x_{0}}$ if $j\in[N+1,2N+1]$; so $\pss{e_{j}^{*}}{u}\neq 0$ in both cases since $x_0$ is a norming vector for $A$ and $A$ is assumed to be evenly distributed. Likewise, since  $B_{\eta ,\delta }(u)=A(x_{0}+\delta ^{p'}x_{0})+\eta \,S_{N}A(x_{0}+\delta ^{p'}x_0)$ we have $\pss{e_{j}^{*}}{B_{\eta ,\delta }u}=(1+\delta ^{p'})\,\pss{e_{i}^{*}}{Ax_{0}}$ if $j\in [0,N]$ and $\pss{e_{j}^{*}}{B_{\eta ,\delta }u}=\eta \,(1+\delta ^{p'})\,\pss{e_{j-N-1}^{*}}{Ax_{0}}$ if $j\in [N+1, 2N+1]$; so $\pss{e_{j}^{*}}{B_{\eta ,\delta }u}\neq 0$.
%\begin{align*}
% \intertext{and}
%  e_{j}^{*}(B_{\eta ,\delta }u)&=\eta \,(1+\delta ^{p'})\,e_{j}^{*}(Ax_{0})\quad \textrm{for every}\ i\in[N+1,2N]
%\end{align*}
%Since $A$ has Property (P), $B_{\eta ,\delta }$ has also Property (P).
\par\medskip
It remains to choose the parameters $\eta $ and $\delta $ in such a way that $\Vert B_{\eta ,\delta }P_{N}-A\Vert<\varepsilon$ and $\Vert B_{\eta ,\delta }\Vert=1$; and, moreover, to show that $\Vert B_{\eta,\delta}-A\Vert<\varepsilon$ if $\Vert A\Vert$ is close enough to $1$. If $\eta >0$ is chosen sufficiently small, then
it follows from the definition of $B_{\eta,\delta}$ that $\Vert B_{\eta ,\delta }P_{N}-A\Vert<\varepsilon $ whatever the choice of $\delta $. Moreover, once \(\eta \) is chosen and is small enough, we can fix \(\delta >0\) in such a way that \(\Vert B_{\eta ,\delta }\Vert=1\). This is possible since \(\Vert A\Vert<1\). Indeed, if $\eta >0$ is such that $(1+\eta^p)^{1/p}\Vert A\Vert <1$, then $\delta >0$ is uniquely determined by the equation
\[(1+\eta ^{p})^{1/p}\,(1+\delta ^{p'})^{1/p'}\,\Vert A\Vert=1,\] namely
\[
\delta =\left[ \biggl ( \dfrac{1}{\Vert A\Vert\,(1+\eta ^{p})^{1/p}}\biggr)^{p'} -1\right]^{1/p'}\le\biggl ( \dfrac{1}{\Vert A\Vert^{p'}}-1\biggr)^{1/p'}=:\delta _{A}.  
\]

\noindent Finally, if $\Vert A\Vert$ is very close to $1$ then $\delta_A$ is very small; so, looking at the definition of $B_{\eta,\delta}$, we see that $\Vert B_{\eta,\delta}-A\Vert <\varepsilon$ provided that $\Vert A\Vert$ is close enough to $1$.
\end{proof}
\begin{remark}\label{Localisation valeurs propres lp}
 The proof of Theorem \ref{Valeurs propres lp} does not extend in a straightforward manner to the case where $X=c_{0}$. The fact that we are working on an $\ell_{p}\,$-$\,$space with $p>1$ is used in an important manner in the proof of Proposition \ref{Proposition E} (both in the differentiability argument underlying the proof of Lemma \ref{Lemma EE}, and in the duality argument used in Step 2 of the proof of Proposition \ref{Proposition E}), as well as in the proof of Proposition \ref{Seconde proposition Localisation vp lp}. Although it seems likely, in view of Proposition \ref{Injection lp}, that the statement of Theorem \ref{Valeurs propres lp} also holds for $X=c_{0}$, we are not able to provide a proof of it.
\end{remark}

\subsection{Additional results} 
Theorem \ref{Valeurs propres lp} rules out one trivial reason why a typical operator $T\in\b_{1}(\ell_{p},\sot)$, $p>2$ should have a non-trivial invariant subspace, but it does not bring us any nearer to an answer to this question. %whether a typical contraction on $\ell_{p}$ for the \sot\ has a non-trivial invariant subspace. 
\par\medskip
In the Hilbertian case, here is a highly non-elementary proof of the fact that a typical $T\in(\bbh,\sot)$ has non-trivial invariant subspaces: by Theorem \ref{EM}, we know that a typical $T\in(\bbh,\sot)$ is such that $\sigma (T)=\overline{\,\D}$; and by the Brown-Chevreau-Pearcy Theorem from \cite{BCP2}, any $T\in\bbh$ whose spectrum contains the unit circle $\T$ has a non-trivial invariant subspace. 
\par\smallskip
It is tempting to try to generalize this argument on more general Banach spaces. The trouble is that there is no full analogue of the Brown-Chevreau-Pearcy Theorem in a Banach space setting. A deep result due to Ambrozie and M\"uller \cite{AM} states that if $T$ is a \emph{polynomially bounded} operator on $X$ such that $\sigma(T)$ contains $\T$ and $\Vert T^nx\Vert\to 0$ for all $x\in X$, then $T$ has a non-trivial invariant subspace. When $X=\ell_{p}$, $1<p<\infty$, or $X=c_{0}$, we do know, thanks to Propositions \ref{Tnto0} and \ref{aameliorer}, that a typical $T\in(\bbx,\sot)$ is such that $\sigma(T)$ contains $\T$ and that $\Vert T^nx\Vert\to 0$ for all $x\in X$. However, whereas any Hilbert space contraction is polynomially bounded by von Neumann's inequality, the next proposition shows that on $X=\ell_{p}$ with $1\le p<\infty$, 
$p\neq 2$, as well as on $c_{0}$, polynomial boundedness is \emph{not} typical.

% \par\smallskip
\bpr\label{polybound} Let $X=c_0$ or $X=\ell_p$ with $1\leq p<\infty$ and $p\neq 2$. Then, the set of all polynomially bounded contractions is $\fs$ and meager in $(\bbx,\emph{\sot})$. In particular, a typical $T\in (\bbx,\emph{\sot})$ is not polynomially bounded.
\epr

\bpf The key point of the proof is the following well known fact, of which we give a proof for convenience of the reader.
\begin{claim}\label{Peller} Let $S$ be the unweighted forward shift on $X$. Then $S$ is not polynomially bounded.
\end{claim}
\begin{proof}[Proof of Claim \ref{Peller}]
We first consider the case $X=\ell_p$, $1\leq p<2$. For any polynomial $p(z)=\sum_{j=0}^d a_j z^j$, we have 
\[ \Vert p(S) e_0\Vert=\left( \sum_{j=0}^d \vert a_j\vert^p\right)^{1/p}.\]
For every $d\ge 0$, let $p_d$ be the classical \emph{Rudin-Shapiro polynomial} of degree $d$ (see \cite{R}). This polynomial has the following properties :
\[ p_d(z)=\sum_{j=0}^d a_{j,d} z^j,\quad \textrm{with}\ a_{j,d}\in\{-1,1\}\ \textrm{for every}\ 0\le j\le d,\quad\hbox{and}\quad \Vert p_d\Vert_{\infty,\D}\leq\sqrt{2(d+1)}.\]
So we have
\[ \Vert p_d(S)\Vert\geq \Vert p_d(S)e_0\Vert= (d+1)^{1/p}\geq \frac1{\sqrt{2}}\, (d+1)^{\frac1p-\frac12} \Vert p_d\Vert_{\infty,\D},\]
which shows that $S$ is not polynomially bounded since $p<2$. 
\par\medskip
Essentially the same proof (looking at the vector $e_{d+1}$ rather than at $e_0$) shows that the {backward} shift $B$ is not polynomially bounded on $\ell_p$, $1\leq p<2$. So the cases where $X=c_0$ and 
$X=\ell_p$, with $2<p<\infty$ can be handled by a duality argument.
\epf 

\smallskip Let us denote by $\g$ the set of all $T\in\bbx$ which are not polynomially bounded. We have to show that $\g$ is a dense $\gd$ subset of $(\bbx,\sot)$.

The fact that $\g$ is \sot$\,$-$\,\gd$ is straightforward: indeed, if $T\in\bbx$ then 
\[ T\in\g\iff \forall K\in\N\;\exists p\in\C[z]\; \exists x\in S_X\;;\; \Vert p(T) x\Vert > K \Vert p\Vert_{\infty,\D}.\]

Let us now show that $\g$ is \sot$\,$-$\,$dense in $\bbx$. Let $A\in\bbx$ be arbitrary, and set $T_N:= P_NA_{| E_N}\oplus S_N$, where $S_N$ is the forward shift acting on $F_N$. Then $\Vert T_N\Vert\leq 1$, $T_N\xrightarrow{\sot} A$, and $T_N\in\g$ because $S_N$ is not polynomially bounded on \(X\) by Claim \ref{Peller}. 
\epf
% 
% \smallskip
% Proposition \ref{polybound} shows that the Ambrozie-M\"uller theorem does not apply to a typical contraction on $X$ when $X=c_0$ or $\ell_p$ with $p\neq 2$. In fact, the question of whether a typical contraction on $\ell_p$ has a non-trivial invariant subspace remains obstinately open for $1<p<\infty$ and $p\neq 2$.
\par\medskip
In the same circle of ideas, we mention the work of V. M\"uller, who proved in \cite{M2} the following result: Assume that $X$ is a Banach space not containing an isomorphic copy of $c_{0}$, and let $T\in\bx$ be a power-bounded operator on $X$ such that $1\in\sigma(T)$. Then, there exist $x_0\neq 0$ in X and $x_0^*\neq 0$ in $X^*$ such that ${\rm Re}\,\langle x_0^* ,T^nx_0\rangle\geq 0$ for all $n\in\Z_+$. Thus, $T$ admits a non-trivial invariant closed convex cone, namely the set \[L:=\{ x\in X;\; \forall n\geq 0\;:\; {\rm Re}\,\langle x_0^* ,T^nx\rangle\geq 0\}.\] It follows that $T$ has a (non-zero) non-supercyclic vector. Indeed, by a result of \cite{LM}, any vector $x\in L$ is non-supercyclic if 
$\sigma_p(T^*)=\emptyset$; whereas if  $T^*$ has an eigenvalue, then $T$ has a non-trivial invariant subspace and hence non-cyclic (non-zero) vectors. Since a typical $T\in(\bbx,\sot)$ is such that $1\in\sigma(T)$ when $X=\ell_p$, $1\le p<\infty$ or $X=c_0$ (by Proposition \ref{aameliorer}), we may therefore state
\begin{proposition}\label{Mullererie}
  Assume that $X=\ell_p$, $1<p<\infty$. Then a typical operator $T\in(\bbx,\emph{\sot})$ admits a non-trivial invariant closed cone and {\rm (}hence{\rm )} a non-zero non-supercyclic vector.
\end{proposition}
% Assume that $X=\ell_p$, $1<p<\infty$. Then a typical operator $T\in(\bbx,\emph{\sot})$ admits a non-trivial invariant closed cone and {\rm (}hence{\rm )} a non-zero non-supercyclic vector.

\smallskip M\"uller's result from \cite{M2} cannot be applied when $X=c_0$. However, we are able to prove the existence of non-supercyclic vectors %a non-trivial invariant closed cone (not necessarily convex) 
in the $c_0\,$-$\,$case as well.
\bpr\label{conec0} A typical $T\in (\mathcal B_1(c_0),\emph{\sot})$ admits a non-zero  non-supercyclic vector.%, and hence a non-trivial invariant closed cone.
\epr
\bpf %Note that if $x\in X$ is a non-zero non-supercyclic vector for an operator $T\in\mathcal B(X)$, then the set $C:=\overline{\bigl\{ \lambda T^nx;\; n\geq 0,\; \lambda\in\C\bigr\}}$ is a non-trivial invariant closed cone for $T$. So we just need to prove the first part of Proposition \ref{conec0}.
%\medskip 
As in the proof of Theorem \ref{Valeurs propres c0}, we use the Banach-Mazur game. Let us denote by $\mathcal A\subseteq (\mathcal B_1(c_0),\sot)$ the set of all $T\in\mathcal B_1(c_0)$ admitting a non-zero non-supercyclic vector. We will describe a strategy for player II in the Banach-Mazur game $\mathbf G(\mathcal A)$, and show that this strategy is winning under a suitable assumption.

\medskip With the notations of the proof of Theorem \ref{Valeurs propres c0}, assume that player I has just played an open set $\mathcal U_{2k}=\mathcal U(N_{2k}, A_{2k},\varepsilon_{2k})$. Then, player II plays the open set $\mathcal U_{2k+1}=\mathcal U(N_{2k+1}, A_{2k+1},\varepsilon_{2k+1})$, where $N_{2k+1}:= N_{2k}+1$,   $0<\varepsilon_{2k+1}<\varepsilon_{2k}/2$, and $A_{2k+1}\in\mathcal B_1(E_{N_{2k+1}})$ is defined as follows:
\[\left\{ \begin{array}{ll}
 A_{2k+1}e_n:=\bigl(1-\frac{\varepsilon_{2k}}{2}\bigr)A_{2k}e_n&\hbox{ for all $0\le n\le N_{2k}$};\\
 A_{2k+1}e_{N_{2k}+1}:=e_{N_{2k}+1}. & %\hbox{for every $1\le l< L_{2k+1}$;}
 \\
%A_{2k+1}e_{N_{2k+1}}:=0.&
\end{array}
\right.
\]

\noindent This is indeed a legal move for II since $\varepsilon_{2k+1}<\varepsilon_{2k}/2$ and $\Vert A_{2k+1}e_n-A_{2k}e_n\Vert\leq \varepsilon_{2k}/2$ for all $0\leq n\leq N_{2k}$ (which implies that $\mathcal U_{2k+1}\subseteq \mathcal U_{2k}$). \smallskip Note that, by the definition of $A_{2k+1}$, we have
\begin{equation}\label{coeff1} \pss{e_{N_{2k}+1}^*}{A_{2k+1}e_{N_{2k}+1}}=1.
\end{equation}

\smallskip 
The actual choice of $\varepsilon_{2k+1}$ will be determined by some large positive integer $L_{2k+1}$, that player II selects in such a way that
\begin{equation}\label{Lbig} \left(1-\frac{\varepsilon_{2k}}{4}\right)^{L_{2k+1}}\le \frac{1}{2^{k+1}}\cdot
\end{equation}

%\smallskip Without telling this to player I, player II also chooses a positive integer $L_{2k+1}$ such that 

\medskip The following fact will be useful.

%\begin{fact}\label{triv} For any $k\geq 0$, we have 
%\[ \pss{e_{N_{2k}+1}^*}{A_{2k+1}e_{N_{2k}+1}}=1.% \qquad\hbox{for all $N_{2k}+1<j\le N_{2k+1}$}.
%\]
%\end{fact}
%\begin{proof}[\it Proof of Fact \ref{triv}] This is clear by the definition of $A_{2k+1}$.
%\end{proof}

\begin{fact}\label{estimation} Let $(\mathcal U_n)_{n\geq 0}$ be a run in the Banach-Mazur game where player {\rm II} has followed the strategy described above. If $T\in\bigcap_{n\geq 0} \mathcal U_n$ then, for any $k\geq 0$, we have
\[ \Vert e_{N_{2k}+1}^* T P_{\mathbb{N}\backslash\{N_{2k}+1\}}\Vert \le \varepsilon_{2k+1}.\]
%\qquad\hbox{for all $N_{2k}+1<j\le N_{2k+1}$}.\]
 \end{fact}
 \begin{proof}[\it Proof of Fact \ref{estimation}] %It is clear that %\[%\sum_{n\le N_{2k+1}}\vert \pss{e_j^*}{A_{2k+1}e_n}\vert=
%$ \pss{e_j^*}{A_{2k+1}e_{j-1}}=1.$ 
%En effet, l'égalité ci-dessus est clair pour tout $N_{2k}< j\le N_{2k+1}$ et pour $j\le N_{2k}$, on a 
%\begin{align*}
%\sum_{l\le N_{2k+1}}|P_jA_{2k+1}e_l|&=\sum_{l\le N_{2k}}(1-\frac{\varepsilon_{2k}}{2})\|P_jA_{2k}e_l|+ 1-(1-\frac{\varepsilon_{2k}}{2})\big(\sum_{l=0}^{N_{2k}}|P_jA_{2k}e_l|\big)=1.
%\end{align*}
Let $x\in c_0$ with  $\Vert x\Vert\le 1$. Set $y:=e_{N_{2k}+1}+ e^{i\theta}P_{\mathbb{N}\backslash \{N_{2k}+1\}}x$, where $\theta\in\R$ is such that $e^{i\theta}\pss{e_{N_{2k}+1}^*}{TP_{\mathbb{N}\backslash \{{N_{2k}+1}\}}x}=\vert \pss{e_{N_{2k}+1}^*}{TP_{\mathbb{N}\backslash \{ N_{2k}+1\}}x}\vert$. Then $\|y\|\le 1$. Hence,  
\begin{align*}
1\ge \vert \pss{e_{N_{2k}+1}^*}{ T y}\vert&\ge \vert e^{i\theta}\pss{e_{N_{2k}+1}^*}{ T P_{\mathbb{N}\backslash \{N_{2k}+1\}}x}+ \pss{e_{N_{2k}+1}^*}{ A_{2k+1}e_{N_{2k}+1}}\vert\\
&\qquad\qquad\qquad\qquad-\vert\pss{e_{N_{2k}+1}^*}{ (T-A_{2k+1})e_{{N_{2k}+1}}}\vert\\
&\ge \vert\pss{e_{N_{2k}+1}^*TP_{\mathbb{N}\backslash \{{N_{2k}+1}\}}}{x}\vert+ 1-\varepsilon_{2k+1} \qquad\hbox{by (\ref{coeff1})}.
\end{align*}

\noindent %where we have used Fact \ref{triv}. 
So we get $\vert\pss{e_{N_{2k}+1}^*TP_{\mathbb{N}\backslash \{{N_{2k}+1}\}}}{x}\vert\leq \varepsilon_{2k+1}$ for every $x\in B_{c_0}$.
 \epf

\smallskip We can now prove
\blm\label{lemmecone} Set $L_{-1}:=0$. Assume that at each move, player {\rm II} decides to choose $L_{2k+1}>L_{2k-1}$ satisfying {\rm (\ref{Lbig})}, and $\varepsilon_{2k+1}$  small enough to ensure that 
%\[\bigl(1-\frac{\varepsilon_{2k}}{4}\bigr)^{L_{2k+1}}\le \frac{1}{2^{k+2}},\]
\[%L_{2k+1}(N_{2k}+4)\varepsilon_{2k+1}\le \frac{1}{2^{k+2}}\quad\text{ and}\quad 
L_{2k+1}(N_{2k}+1)\varepsilon_{2k+1}\le \frac{\varepsilon_{2k}}{4}\left(1-\frac{\varepsilon_{2k}}{2}\right)^{L_{2k+1}} .\]
Then, the above strategy is winning for \emph{II}.
\elm
\begin{proof}[\it Proof of Lemma \ref{lemmecone}] Let $(\mathcal U_n)_{n\geq 0}$ be a run in the Banach-Mazur game where player {\rm II} has followed his strategy, and let $T\in\bigcap_{n\geq 0} \mathcal U_n$. We are going to show that the vector 
\[ x:= e_0+\sum_{k=0}^\infty \frac1{2^{k+1}} e_{N_{2k}+1}\]
is a non-supercyclic vector for $T$. The key point is the following fact. %(Here and afterwards, we set $L_{-1}:=0$.)
\begin{fact}\label{claimc0} For any $k\geq 0$ and all $n$ with $L_{2k-1}\leq n<L_{2k+1}$, we have 
\[ \Vert T^nx\Vert\leq 8\, \vert\pss{e^*_{N_{2k}+1}}{T^nx}\vert.\]
\end{fact}

Taking Fact \ref{claimc0} for granted, it is easy to finish the proof of Lemma \ref{lemmecone}. Indeed, let $n\geq 0$, and choose $k$ such that $L_{2k-1}\leq n< L_{2k+1}$. For any $\lambda\in\C$, we have 
\begin{align*}\|\lambda T^n x-e_0\|&\ge \max\Bigl(1-\vert \lambda\vert \|T^nx\|, \vert\lambda\vert \,\vert\pss{e^*_{N_{2k}+1}}{T^nx}\vert \Bigr)\\
&\geq \max\Bigl(1-\vert \lambda\vert \|T^nx\|, \vert\lambda\vert \,\Vert T^nx\Vert/8 \Bigr)\\
&\geq 1/9.
\end{align*}

\noindent So we see that $e_0\not\in\overline{\bigl\{ \lambda T^nx;\; n\geq 0,\; \lambda\in\C\bigr\}}$, and hence that $x$ is not a supercyclic vector for $T$.
\epf

\begin{proof}[\it Proof of Fact \ref{claimc0}] This will follow from the next four claims.
\begin{claim}\label{cl1} For any $k,n\geq 0$, we have 
\[\vert \pss{e_{N_{2k}+1}^*}{T^n x}\vert\ge \frac{1}{2^{k+1}}-2n\,\varepsilon_{2k+1}.\]
\end{claim}
\begin{proof}[\it Proof of Claim \ref{cl1}] Let us fix $k\geq 0$. The result is clear for $n=0$ since $\pss{e_{N_{2k}+1}^*}{x}=\frac1{2^{k+1}}$ by the definition of $x$. Moreover, if $n\geq 1$ then
\begin{align*}
&\vert \pss{e_{N_{2k}+1}^*}{T^n x}\vert\\
&\qquad\geq \vert \pss{e_{N_{2k}+1}^*}{TP_{\{N_{2k}+1\}}T^{n-1} x}\vert - \vert \pss{e_{N_{2k}+1}^*}{TP_{\mathbb{N}\backslash\{N_{2k}+1\}}T^{n-1} x}\vert\\
&\qquad\ge \vert \pss{e_{N_{2k}+1}^*}{A_{2k+1}P_{\{N_{2k}+1\}}T^{n-1} x}\vert-\vert\pss{e_{N_{2k}+1}^*}{(T-A_{2k+1})P_{\{ N_{2k}+1\}}T^{n-1} x}\vert\\
&\qquad\qquad-\vert\pss{e_{N_{2k}+1}^*}{TP_{\mathbb{N}\backslash\{N_{2k}+1\}}T^{n-1} x}\vert\\
&\qquad \ge \vert\pss{e_{N_{2k}+1}^*}{T^{n-1} x}\vert-\varepsilon_{2k+1}\Vert T^{n-1} x\Vert
-\varepsilon_{2k+1}\Vert T^{n-1} x\Vert\qquad\hbox{by (\ref{coeff1}) and Fact \ref{estimation}}\\
&\qquad \ge \vert\pss{e_{N_{2k}+1}^*}{T^{n-1} x}\vert-2\varepsilon_{2k+1}.
\end{align*}
So the claim follows by induction.
\epf

\begin{claim}\label{cl2} For any $k,n\geq 0$, we have 
\[\Vert P_{(N_{2k},\infty)}T^{n}P_{[0,N_{2k}]}x\Vert \le n (N_{2k}+1)\, \varepsilon_{2k+1}.\]
\end{claim}
\begin{proof}[\it Proof of Claim \ref{cl2}] Let us fix $k\geq 0$. The result is clearly true for $n=0$ since $P_{(N_{2k},\infty)}T^{0}P_{[0,N_{2k}]}=0$. Moreover, if $n\geq 1$ then
\begin{align*}
&\|P_{(N_{2k},\infty)}T^{n}P_{[0,N_{2k}]}x\|\\
&\qquad\qquad\le
\|P_{(N_{2k},\infty)}TP_{[0,N_{2k}]}T^{n-1}P_{[0,N_{2k}]}x\|+\|P_{(N_{2k},\infty)}TP_{(N_{2k},\infty)}T^{n-1}P_{[0,N_{2k}]}x\|\\
&\qquad\qquad\le  \|P_{(N_{2k},\infty)}A_{2k+1}P_{[0,N_{2k}]}T^{n-1}P_{[0,N_{2k}]}x\|\\
&\qquad\qquad\qquad+\|P_{(N_{2k},\infty)}(T-A_{2k+1})P_{[0,N_{2k}]}T^{n-1}P_{[0,N_{2k}]}x\|+\|P_{(N_{2k},\infty)}T^{n-1}P_{[0,N_{2k}]}x\|\\
&\qquad\qquad\le 0+ (N_{2k}+1)\varepsilon_{2k+1}+\|P_{(N_{2k},\infty)}T^{n-1}P_{[0,N_{2k}]}x\|;%\\
%&\quad\le n (N_{2k}+1)\varepsilon_{2k+1}. 
\end{align*}
and the claim follows by induction.
\epf

\begin{claim}\label{cl3} For any $k\geq 0$ and all $n\leq L_{2k+1}$, we have
\[\|T^{n}P_{[0,N_{2k}]}x\|\le \left(1-\frac{\varepsilon_{2k}}{4}\right)^{n}.\]
\end{claim}
\begin{proof}[\it Proof of Claim \ref{cl3}] This is clear if $n=0$. Assume that $1\le n\le L_{2k+1}$ and that the inequality has been proved for $n-1$. Then% et % Par induction sur $n$, on a bien pour tout $1+n\le l\le L_{2k+1}$,
%
%\begin{align*}
%&|P_{N_{2k}+l}T^{n}P_{[0,N_{2k}]}x|\\
%&\quad \le 
%|P_{N_{2k}+l}T P_{N_{2k}+l-1}T^{n-1}P_{[0,N_{2k}]}x|+|P_{N_{2k}+l}T P_{\mathbb{N}\backslash\{N_{2k}+l-1\}}T^{n-1}P_{[0,N_{2k}]}x|\\
%&\quad \le |P_{N_{2k}+l}A_{2k+1} P_{N_{2k}+l-1}T^{n-1}P_{[0,N_{2k}]}x|\\
%&\quad\quad+|P_{N_{2k}+l}(T-A_{2k+1}) P_{N_{2k}+l-1}T^{n-1}P_{[0,N_{2k}]}x| +\varepsilon_{2k+1}\\
%&\quad \le |P_{N_{2k}+l-1}T^{n-1}P_{[0,N_{2k}]}x|+ 2 \varepsilon_{2k+1}\\
%&\quad \le 2(n-1)\varepsilon_{2k+1}+2\varepsilon_{2k+1}=2n\varepsilon_{2k+1}
%\end{align*}
%comme $L_{2k+1}(N_{2k+1}+1)\varepsilon_{2k+1}\le \frac{\varepsilon_{2k}}{4}(1-\frac{\varepsilon_{2k}}{2})^{L_{2k+1}}$, on a par induction
\begin{align*}
&\|T^{n}P_{[0,N_{2k}]}x\|\\
&\quad \le 
\|T P_{[0,N_{2k}]}T^{n-1}P_{[0,N_{2k}]}x\|+\|T P_{(N_{2k},\infty)}T^{n-1}P_{[0,N_{2k}]}x\|\\
&\quad \le  \|A_{2k+1} P_{[0,N_{2k}]}T^{n-1}P_{[0,N_{2k}]}x\|\\
&\quad\quad+\|(T-A_{2k+1}) P_{[0,N_{2k}]}T^{n-1}P_{[0,N_{2k}]}x\|  + \|P_{(N_{2k},\infty)}T^{n-1}P_{[0,N_{2k}]}x\|\\
&\quad \le \left(1-\frac{\varepsilon_{2k}}{2}\right)\|A_{2k}P_{[0,N_{2k}]}T^{n-1}P_{[0,N_{2k}]}x\|\\
&\quad\quad+(N_{2k}+1)\varepsilon_{2k+1}+(n-1)(N_{2k}+1)\varepsilon_{2k+1}\qquad\hbox{by Claim \ref{cl2}}\\
&\quad \le \left(1-\frac{\varepsilon_{2k}}{2}\right)\|T^{n-1}P_{[0,N_{2k}]}x\|+L_{2k+1}(N_{2k}+1)\varepsilon_{2k+1}\qquad\hbox{since $n\le L_{2k+1}$}\\
&\quad \le \left(1-\frac{\varepsilon_{2k}}{2}\right)\left(1-\frac{\varepsilon_{2k}}{4}\right)^{n-1}+\frac{\varepsilon_{2k}}{4}\left(1-\frac{\varepsilon_{2k}}{2}\right)^{L_{2k+1}}\qquad\hbox{by assumption on $\varepsilon_{2k+1}$}\\
&\quad \le \left(1-\frac{\varepsilon_{2k}}{2}\right)\left(1-\frac{\varepsilon_{2k}}{4}\right)^{n-1}+\frac{\varepsilon_{2k}}{4}\left(1-\frac{\varepsilon_{2k}}{2}\right)^{n-1}\\
%&\quad \le (1-\frac{\varepsilon_{2k}}{2})(1-\frac{\varepsilon_{2k}}{4})^{n-1}+\frac{\varepsilon_{2k}}{4}(1-\frac{\varepsilon_{2k}}{2})^{n-1}}\\
%&\quad \le \big(1-\frac{\varepsilon_{2k}}{2}+\frac{\varepsilon_{2k}}{4}(1-\frac{\varepsilon_{2k}}{2})^{L_{2k+1}-n+1}\big) (1-\frac{\varepsilon_{2k}}{4})^{n-1}\\
&\quad = \left(1-\frac{\varepsilon_{2k}}{4}\right)^{n}.
\end{align*} 
This proves the claim.
\epf

\begin{claim}\label{cl4} We have $\|T^{L_{2k-1}}x\|\le \frac{1}{2^{k-1}}$ for all $k\geq 0$.
\end{claim}
\begin{proof}[\it Proof of Claim \ref{cl4}] This is true if $k=0$ since $\|T^{L_{-1}}x\|=\|x\|=1\le 2$; so assume that $k\geq 1$, and let $k':=k-1$. Then %we may write
%On a bien $\|T^{L_{-1}}x\|=\|x\|=1\le 2$ et comme $T$ est une contraction, nous avons pour tout $k\ge 1$,
\begin{align*}
\|T^{L_{2k-1}}x\|=\Vert T^{L_{2k'+1}}x\Vert&\le \|T^{L_{2k'+1}}P_{[0,N_{2k'}]}x\|+\|T^{L_{2k'+1}}P_{(N_{2k'},\infty)}x\|\\
&\le \|T^{L_{2k'+1}}P_{[0,N_{2k'}]}x\|+\|P_{(N_{2k'},\infty)}x\|\\
&= \|T^{L_{2k'+1}}P_{[0,N_{2k'}]}x\|+ \frac{1}{2^{k'+1}}\\
&\leq \left(1-\frac{\varepsilon_{2k'}}{4}\right)^{L_{2k'+1}}+ \frac{1}{2^{k'+1}}\qquad\hbox{by Claim \ref{cl3}}\\
&\leq \frac{1}{2^{k'+1}}+\frac1{2^{k'+1}}=\frac{1}{2^{k-1}}\cdot
\end{align*}
%So it is more than enough to show that $\Vert T^{L_{2k-1}}P_{[0,N_{2k-2}]}x\|\le \frac{1}{2^{k}}$; or, writing $k=k'+1$, that%Il reste à vérifier que pour tout $k\ge 0$, 
%\[\|T^{L_{2k'+1}}P_{[0,N_{2k'}]}x\|\le \frac{1}{2^{k'+1}}\cdot\]
%Soit $k\ge 0$. On commence par montrer que pour tout $n\ge 0$,

%\noindent Using Claims \ref{cl2} and \ref{cl3} with $k'$ in place of $k$, we get
%\begin{align*}
%\|T^{L_{2k'+1}}P_{[0,N_{2k'}]}x\|&\le 
%\|P_{[0,N_{2k'}]}T^{L_{2k'+1}}P_{[0,N_{2k'}]}x\|+\|P_{(N_{2k'},\infty)}T^{L_{2k'+1}}P_{[0,N_{2k'}]}x\|\\
%&\le \left(1-\frac{\varepsilon_{2k'}}{4}\right)^{L_{2k'+1}}+ L_{2k'+1}(N_{2k'}+1)\varepsilon_{2k'+1}\\
%&\le \frac{1}{2^{k'+1}}\qquad\hbox{by assumption on $L_{2k'+1}$ and $\varepsilon_{2k'+1}$}.
%\end{align*}
%which concludes the proof.
\epf

We can now conclude the proof of Fact \ref{claimc0}, and hence that of Lemma \ref{lemmecone}. Let $k\geq 0$ and $L_{2k-1}\le n< L_{2k+1}$. On the one hand, we have by Claim \ref{cl1}:
\begin{align*} \vert\pss{e_{N_{2k}+1}^*}{T^n x}\vert &\geq \frac1{2^{k+1}}- 2L_{2k+1}\varepsilon_{2k+1} \\
&\geq \frac{1}{2^{k+2}}\qquad\hbox{since}\quad\hbox{ $2L_{2k+1}\varepsilon_{2k+1}\leq \frac{\varepsilon_{2k}}2 \left(1-\frac{\varepsilon_{2k}}2\right)^{L_{2k+1}}\leq\frac1{2^{k+2}}$};
\end{align*}
and on the other hand, by Claim \ref{cl4} (and since $\Vert T\Vert\leq 1$): 
\[ \Vert T^nx\Vert \leq \|T^{L_{2k-1}}x\|\le \frac{1}{2^{k-1}}\cdot\]
Hence, $\vert\pss{e_{N_{2k}+1}^*}{T^n x}\vert\geq \Vert T^nx\Vert/8$.
\end{proof}

By Lemma \ref{lemmecone}, the proof of Proposition \ref{conec0} is now complete.
\epf
%We do not know if Proposition \ref{Mullererie} holds true as well when $X=c_0$, because M\"uller's result from \cite{M2} cannot be applied in that case. Note that it is also proved in \cite{M2} that without any assumption on the space $X$, if $T\in\bx$ is a power-bounded operator such that $1\in\sigma(T)$, then $T^*$ has a non-trivial invariant closed cone; but this does not seem to be very helpful to determine whether a typical contraction on $X=c_{0}$ admits a non-trivial invariant closed cone.

\section{Typical properties of ``triangular plus 1'' contractions}\label{Section5}
In this section, we study typical properties of contraction operators which are ``triangular plus 1'' with respect to some fixed basis of the underlying Banach space. One of the results proved in this section (namely, Lemma \ref{Lemma 2.1 bis}) will be used in the proof of Theorem \ref{Sixieme th}, to be given in the forthcoming Section \ref{Section6} (Theorem~\ref{Theorem 6.2}).
\par\medskip
In the first part of this section, we place ourselves in the Hilbertian setting. Let $H$ be a Hilbert space (complex, infinite-dimensional and separable), and let $(f_{j})_{j\ge 0}$ be a fixed orthonormal basis of $H$.
We denote by $\tbh$ the set of all the operators $T\in\bbh$ which are ``triangular plus $1$'' with respect to the basis $(f_j)$, with positive entries on the first subdiagonal:
\[ 
{\mathcal T}_1(H)=\bigl\{ T\in\bbh\;;\; \hbox{$Tf_j\in [f_0,\dots ,f_{j+1}]$ for all $j$}\quad\hbox{and}\quad \langle Tf_j, f_{j+1}\rangle >0\bigr\}. 
\]
Since any cyclic operator $T\in\bbh$ is unitarily equivalent to some operator belonging to $\tbh$, this is a rather natural class to consider. The following fact is easy to check.

\begin{lemma} The set $\tbh$ is a $G_\delta$ subset of $(\bbh,\emph{\sot})$, hence of \((\bbh,\emph{\sote})\), and hence a Polish space with respect to both topologies. However, $\tbh$ is nowhere dense in $(\bbh,\emph{\sot})$ as well as in \((\bbh,\emph{\sote})\).
\end{lemma}
\bpf For any fixed $j\in\Z_+$, the condition ``$Tf_j\in [f_0,\dots ,f_{j+1}]$'' defines a closed set in $(\bbh,\sot)$ since it is equivalent to ``$\langle Tf_{j}, f_n\rangle=0 \;\hbox{for all $n>j+1$}$'', and the condition ``$ \langle Tf_j, f_{j+1}\rangle >0$'' defines a $\gd$ set because $(0,\infty)$ is a $\gd$ subset of $\C$; so $\tbh$ is a $\gd$ in $(\bbh,\sot)$, hence in \((\bbh,{\sote})\). Moreover, $\tbh$ is contained in the \sot$\,$-$\,$closed set 
\[
 {\mathcal F}:=\bigl\{ T\in\bbh\;;\; \hbox{$Tf_j\in [f_0,\dots ,f_{j+1}]$ for all $j\in\Z_+$}\bigr\},
\]
whose complement is clearly dense in $(\bbh,\sote)$, hence in \((\bbh,{\sot})\).
\epf
Also we have

\begin{lemma}\label{Lemma 2.1 bis}
 The set of operators $T\in\b_{1}(H)$ which are unitarily equivalent to some operator belonging to $\tbh$ is comeager in $(\bbh,\emph{\sot})$ and in \((\bbh,\emph{\sote})\).
\end{lemma}

\begin{proof}[Proof of Lemma \ref{Lemma 2.1 bis}]
The orbit of the set $\tbh$ under unitary equivalence is exactly the class ${{\mathcal{CY}}}_1(H)$ of all cyclic contractions on $H$. Let $(V_{k})_{k\ge 1}$ be a basis of open sets for $H$, and let $(p_{r})_{r\ge 1}$ be an enumeration of all complex polynomials with coefficients in $\Q+i\,\Q$. 
The class ${{\mathcal{CY}}}_1(H)$ contains the set ${{\mathcal{CY}}}_{d,1}(H)$ of cyclic contractions on $H$ with a dense set of cyclic vectors.
It follows easily from the Baire Category Theorem that
\[
{{\mathcal{CY}}}_{d,1}(H)=\bigcap_{(k,l)\in\N^{2}}\;\bigcup_{r\ge 1}\;\bigl \{ T\in\b_{1}(H)\;;\;p_{r}(T)^{-1}(V_{k})\cap V_{l}\neq\emptyset\bigr\}\cdot 
\]
%It  that any operator $T\in \mathfrak G$ is cyclic. 
From this, we see that the set ${{\mathcal{CY}}}_{d,1}(H)$ is \sot$\,$-$\,G_{\delta }$ in $\bbh$, hence \sot$^{*}$-$\,G_{\delta }$. So we  only need to prove that ${{\mathcal{CY}}}_{d,1}(H)$ is \sot$^{*}$-$\,$dense in 
$\bbh$. But this follows immediately from the fact, proved in \cite[Corollary 2.12]{GMM}, that the \emph{hypercyclic} operators are \sot$^{*}$-$\,$dense in $\b_{M}(H)$ for any $M>1$, since non-zero multiples of hypercyclic operators have a dense set of cyclic vectors.\end{proof}

Since $\tbh$ is nowhere dense in $(\bbh,\sot)$, it does not follow immediately from Theorem \ref{EM} that a typical $T\in(\tbh,\sot)$ is unitarily equivalent to $B_\infty$. However, we will prove that it is indeed the case. This will follow from a more general result on preservation of comeagerness.

\smallskip Let us denote ${{\mathcal U}}(H)$ the unitary group of $H$. This is a Polish group when endowed with \sot. %We have already explained why this is so (see the proof of Proposition \ref{0-1}); but here is a slightly different proof anyway. First, $\mathcal U(H)$ is a $\gd$ subset of $(\bbh, \sot)$. Indeed, an operator $T\in\bbh$ is unitary if and only if $T$ is an isometry and $T^*$ is an isometry. The first condition is \sot$\,$-$\,$closed in 
%$\bbh$, and the second one is \sot$\,$-$\,\gd$ (see the beginning of the proof of Theorem \ref{l1}). The continuity of multiplication is clear. Finally, the continuity of  the map $U\mapsto U^{-1}$ can be proved as follows. If $(U_n)$ is a sequence in ${{\mathcal U}}(H)$ such that $U_n\xrightarrow{\sot} U\in{{\mathcal U}}(H)$, then $U_n^{-1}=U_n^*\to U^*=U^{-1}$ for the \emph{Weak Operator Topology} \texttt{WOT}. But the topologies \texttt{WOT} and \texttt{SOT} coincide on ${{\mathcal U}}(H)$ (in other words, ${{\mathcal U}}(H)$ is a \emph{light group} in the sense of Megrelishvili; see \cite{Me} and \cite{AFGR} for more on this topic), so $U_n^{-1}\xrightarrow{\sot} U^{-1}$. %Therefore, $({{\mathcal U}}(H), \sot)$ is a Polish group.
%\smallskip
 In what follows, we will denote by ${{\mathcal U}}_{f_0}(H)$ the set of all $U\in{{\mathcal U}}(H)$ such that $Uf_0=f_0$. This is an \sot$\,$-$\,$closed subgroup of ${{\mathcal U}}(H)$, hence a Polish group. We say that a set ${\mathcal P}\subseteq \bbh$ is \emph{${{\mathcal U}}_{f_0}\,$-$\,$invariant} if $U\, {\mathcal P} \, U^{-1}\subseteq{\mathcal P}$ for all $U\in{{\mathcal U}}_{f_0}(H)$; equivalently, if $U\, {\mathcal P} \, U^{-1}={\mathcal P}$ for all $U\in{{\mathcal U}}_{f_0}(H)$. The result we intend to prove reads as follows.

\bth\label{invariant} Let ${\mathcal P}\subseteq \bbh$ be ${{\mathcal U}}_{f_0}(H)\,$-$\,$invariant. If the set ${\mathcal P}$ is comeager in $(\bbh,\emph{\sot)}$, then  ${\mathcal P}\cap \tbh$ is comeager in $(\tbh,\emph{\sot})$. %The converse is true if $\mathfrak P$ is
\eth

\smallskip
Note that we do not assume that ${\mathcal P}$ is $\gd$ in $(\bbh,\sot)$. This is, however, not a real complication in view of the following folklore lemma. We thank C. Rosendal for showing us how to prove it. 

\blm\label{Vaught} Let $G$ be a Polish group acting continuously on a Polish space $Z$, and let $P$ be a comeager and $G\,$-$\,$invariant subset of $Z$. Then $P$ contains a dense, $G\,$-$\,$invariant $\gd$ set in $Z$.
\elm
\bpf Since $P$ is comeager, it contains a dense $\gd$ set $B$. Consider the \emph{Vaught transform} $B^*$ of $B$, which is defined as follows:
\[ B^*:=\bigl\{ z\in Z;\; g\cdot z\in B\;\hbox{for a comeager set of $g\in G$}\bigr\}.\]
As shown in \cite{V}, the Vaught transform preserves the multiplicative Borel classes. Hence, $B^*$ is a $\gd$ set. Since $B$ is comeager in $Z$, it follows from the Kuratowski-Ulam Theorem (see \cite{Ke}) that $B^*$ is comeager as well. Moreover, $B^*$ is easily seen to be $G\,$-$\,$invariant. Finally, $B^*$ is contained in $P$ because $P$ is $G\,$-$\,$invariant.
\epf

\smallskip We now give the
\begin{proof}[Proof of Theorem \ref{invariant}] Assume that ${\mathcal P}$ is comeager in $(\bbh,\sot)$. By Lemma \ref{Vaught}, we may assume that ${\mathcal P}$ is also $\gd$. Then ${\mathcal P}\cap\tbh$ is $\gd$ in $(\tbh,\sot)$; so we just need to show that ${\mathcal P}\cap \tbh$ is \sot$\,$-$\,$dense in $\tbh$.
\par
Observe that for any $f\in H\setminus\{ 0\}$, the set 
\[ 
\g_f:=\{ T\in\bbh;\; T\;\hbox{is cyclic with cyclic vector $f$}\}
\] 
is a dense $\gd$ subset of $(\bbh,\sot)$. Indeed, the fact that $\g_{f}$ is $\gd$ is easy to check, and density follows for example from \cite[Lemma 2.23]{GMM}. So the set $\g_0:={\mathcal P}\cap \g_{f_0}$ is a dense $\gd$ subset of $(\bbh,\sot)$; in particular $\g_0$ is dense in $\g_{f_0}$. Note also that $\tbh\subseteq \g_{f_0}$.%, which is also $\mathcal U_{f_0}(H)$-invariant. 
\par 
For any $T\in\g_{f_0}$, let us denote by $(f_j(T))_{j\geq 0}$ the orthonormal basis of $H$ obtained by applying the Gram-Schmidt orthonormalization process to the sequence $(T^jf_0)_{j\geq 0}$ (which is linearly independent and spans a dense linear subspace of $X$ since $f_0$ is a cyclic vector for $T$); and let $U(T):H\to H$ be the associated ``change of basis'' operator, \mbox{\it i.e.} the unitary operator defined by $U(T)f_j=f_j(T)$, $j\geq 0$. Note that in fact $U(T)$ belongs to $\mathcal{U}_{f_0}(H)$. Writing down the orthonormalization process explicitly, one easily checks that the maps $T\mapsto f_j(T)$ are continuous from $(\g_{f_0},\sot)$ into $H$. Hence, the map $T\mapsto U(T)$ is $\sot\,$-$\,$continuous from $\g_{f_0}$ into ${{\mathcal U}}_{f_0}(H)$; and since $({{\mathcal U}}_{f_0}(H),\sot)$ is a topological group, the map $T\mapsto U(T)^{-1}$ is also continuous.

For any $T\in\g_{f_0}$, we set $R(T):=U(T)TU(T)^{-1}$. It is easily checked that $R(T)\in\tbh$. Moreover, the map $T\mapsto R(T)$ is $\sot\,$-$\,$continuous from $\g_0$ into $\tbh$, because the product map \((T,S)\mapsto TS\) is jointly continuous on $(\bbh,\sot)\times(\bbh,\sot)$. Note also that %$S(T)$ is in particular defined for any $T\in\tbh$, because $\tbh\subseteq\g_{f_0}$, and that 
$R(T)=T$ for every $T\in\tbh$, by the very definition of $\tbh$ and of the map $R$. Finally, we have $R(T)\in {\mathcal P}$ for any $T\in \g_0={\mathcal P}\cap\g_{f_0}$ because ${\mathcal P}$ is ${{\mathcal U}}_{f_0}(H)\,$-$\,$invariant. So we have defined an 
\sot$\,$-$\,$continuous retraction $R:\g_{f_0}\to \tbh$ such that $R(\g_0)\subseteq {\mathcal P}$. Since $\g_0$ is dense in $\g_{f_0}$, it follows immediately that ${\mathcal P}\cap \tbh$ is dense in $(\tbh,\sot)$.
%Now, let $T\in\tbh$ be arbitrary. Since $\g_0$ is dense in $\bbh$, one can find a sequence $(T_n)\subseteq \g_0$ such that $T_n\xrightarrow{\sot} T$.
\epf

\begin{remark} Theorem \ref{invariant} admits the following partial converse: if ${\mathcal P}\subseteq \bbh$ is a ${{\mathcal U}}_{f_0}(H)\,$-$\,$invariant set such that ${\mathcal P}\cap \tbh$ is comeager in $(\tbh,\sot)$, and if ${\mathcal P}$ is also \sot$\,$-$\,\gd$, then ${\mathcal P}$ is comeager in $(\bbh,\sot)$. Indeed, in this case we just have to check that ${\mathcal P}$ is \sot$\,$-$\,$dense in $\bbh$; and with the notation of the above proof, it is enough to show that the \sot$\,$-$\,$closure of ${\mathcal P}$ in $\bbh$ contains $\g_{f_0}$ since $\g_{f_0}$ is \sot$\,$-$\,$dense in $\bbh$. Let $T\in\g_{f_0}$ be arbitrary. Since ${\mathcal P}\cap\tbh$ is \sot$\,$-$\,$dense in $\tbh$, one can find a sequence $(S_n)\subseteq {\mathcal P}\cap \tbh$ such that $S_n\xrightarrow{\sot} S(T):=U(T)^{-1}TU(T)$. Then $T_n:=U(T)S_nU(T)^{-1}\in{\mathcal P}$ for all $n$ since ${\mathcal P}$ is ${{\mathcal U}}_{f_0}(H)\,$-$\,$invariant, and $T_n\xrightarrow{\sot} T$.%But this is clear since the continuous retraction $R:\g_{f_0}\to \tbh$ admits a continuous section $S :\tbh\to \g_{f_0}$ such that $S(\mathfrak P)\subseteq \mathfrak P$, namely $S(T):= U(T)^{-1} TU(T)$. 
\end{remark}

\smallskip From Theorem \ref{invariant}, we immediately deduce
\bco\label{T1l2} A typical $T\in(\tbh,\emph{\sot})$ is unitarily equivalent to $B_\infty$, and hence has all properties listed in Corollary \ref{vrac}. In particular, a typical $T\in(\tbh,\emph{\sot})$ has a non-trivial invariant subspace.
\eco

\smallskip If $X=c_0$ or $\ell_p$, with canonical basis $(e_j)_{j\geq 0}$, one can define $\tbx$ in the obvious way as 
\[
\tbx=\bigl\{T\in\mathcal{B}(X)\;;\;Te_{j}\in[\,e_{0},\dots,e_{j+1}\,]\ \textrm{and}\ \pss{e_{j+1}^{*}}{Te_{j}}>0\ \textrm{for every}\ j\ge 0\bigr\}
\]
 and the question of the existence of non-trivial invariant subspaces for a typical operator $T\in(\tbx,\sot)$ makes sense. Not surprisingly, this question has a positive answer when $X=\ell_1$.

\bpr\label{T1l1} Assume that $X=\ell_1$. Then, a typical $T\in(\tbx,\emph{\sot})$ has all the properties listed in Theorem \ref{l1}. In particular, a typical $T\in(\tbx,\emph{\sot})$ has a non-trivial invariant subspace.
\epr
\bpf By the proof of Theorem \ref{l1}, it is enough to check that a typical $T\in(\tbx,\sot)$ has the following properties: 
\begin{enumerate}
 \item [(i)] $T^*$ is an isometry;
 \item[(ii)] $T-\lambda$ has dense range for any $\lambda\in\D$;
 \item[(iii)] $\sigma(T)\supseteq \D$.
\end{enumerate}
\par\medskip
(i) With the notations of the proof of Theorem \ref{l1}, we already know that ${\mathcal I}_*\cap\tbx$ is a $\gd$ subset of $\tbx$; so we just need to check that ${\mathcal I}_*\cap\tbx$ is dense in $(\tbx,\sot)$.
%\[ \mathfrak I_*\cap\tbx=\bigcap_{(\alpha\in\Q} \g_{\alpha},\]
%where
%\[ \g_\alpha:=\bigl\{ T\in\tbx;\; \Vert T^*x^*\Vert \geq \alpha \;\hbox{for every}\; x^*\in B_{X^*}\;\hbox{such that}\; \Vert x^*\Vert>\alpha\; \bigr\}.\]

%\[ \mathfrak I_*\cap\tbx=\bigcap_{(\alpha,\beta)\in \Lambda} \g_{\alpha,\beta},\]
%where $\Lambda:=\{ (\alpha,\beta)\in\Q\times\Q\;:\; \alpha<\beta\}$ and
%\[ \mathfrak G_\alpha:=\bigl\{ T\in\tbx;\; \Vert T^*x^*\Vert >\alpha \;\hbox{for every}\; x^*\in B_{X^*}\;\hbox{such that}\; \Vert x^*\Vert>\beta\; \bigr\}.\]
%By the proof of Theorem \ref{l1}, each set $\g_{\alpha,\beta}$ is $\gd$ in $\tbx$; so it is enough to check that all these sets are dense in $\tbx$ (by the Baire category theorem). 
As usual, denote by $P_N:X\to [e_0,\dots ,e_N]$ the canonical projection map and set $F_N:=[e_j,\; j> N]$.
Let also $\phi:\Z_+\to\Z_+$ be a map taking every value $j\in\Z_+$ infinitely many times and such that $\phi(k)\leq k$ for all $k\geq 0$. 

Given an arbitrary $A\in \tbx$ with $\Vert A\Vert <1$, choose a sequence of positive real numbers $(\varepsilon_n)_{n\geq 0}$ such that $\Vert A\Vert+\varepsilon_n\leq 1$ for all $n$ and $\varepsilon_n\to 0$ as $n\to\infty$, and define for each $N\geq 0$ an operator $T_N$ as follows:
\[ 
T_N:= P_NAP_N+ \varepsilon_0 e_N^*\otimes e_{N+1}+ \widetilde B_{N}(I-P_{N}),
\]
where $\widetilde B_{N}:F_{N}\to X$  is the operator defined by 
\[ 
\widetilde B_{N}e_{N+1+k}=(1-\varepsilon_{1+k}) e_{\phi(k)}+ \varepsilon_{1+k} e_{N+2+k}\quad\hbox{for every}\ k\geq 0.
\]
Then $T_N\in\tbx$. Moreover, by the choice of the map $\phi$ and since we are working on $\ell_1$, it is easy to check that $\Vert T_N^*x^*\Vert \geq\Vert x^*\Vert$ for every $x^*\in X^*=\ell_\infty$. So we have $T_N\in {\mathcal I}_*\cap\tbx$; and since $T_N\xrightarrow{\sot} A$ as $N\to\infty$, it follows that ${\mathcal I}_*\cap\tbx$ is dense in $(\tbx,\sot)$.
\par\medskip 
(ii) The fact that a typical $T\in(\tbx,\sot)$ is such that $T-\lambda$ has dense range for every $\lambda\in\D$ (in fact, for every $\lambda\in\C$) follows from the proof of Proposition \ref{aameliorer}.
\par\medskip
(iii) For any set $E\subseteq\C$, let us denote by $\g_E$ the set of all $T\in\bbx$ such that $T-\lambda$ is not bounded below for any $\lambda\in E$. To prove that a typical $T\in(\tbx,\sot)$ is such that $\sigma(T)\supseteq\D$, it is obviously enough to show that 
$\g_\D\cap\tbx$ is comeager in $(\tbx,\sot)$.% (which is a stronger statelment).

By the proof of Proposition \ref{aameliorer}, we know that $\g_K$ is an \sot$\,$-$\,\gd$ set for any compact set $K$. Therefore, by the Baire Category Theorem and since $\D$ is a countable union of compact sets, we just have to check that $\g_K\cap\tbx$ is \sot$\,$-$\,$dense in $\tbx$ for any compact set  $K\subseteq\D$. So let us fix such a compact set $K$.

For any $N\geq 0$, let us denote by $B_N$ the canonical backward shift acting on $F_N$ with respect to the basis \((e_{j})_{j>N}\). For any $\lambda\in\D$, the operator $B_N-\lambda$ is a Fredholm operator on \(F_{N}\) with $\dim \ker(B_N-\lambda)=1={\rm ind}(B_N-\lambda)$. By the standard  perturbation theory for Fredholm operators (see for instance \cite[Proposition 2.c.9]{LT}) and since $K$ is a compact subset of $\D$, it follow that there exists $\delta>0$ such that, for any $N\geq 0$, the following holds true: any operator $R\in{\mathcal B}(F_N)$ with 
$\Vert R-B_N\Vert\leq\delta$ is such that $R-\lambda$ is Fredholm with ${\rm ind}(R-\lambda)=1$ for every $\lambda\in K$. In particular, any operator $R\in{\mathcal B}(F_N)$ with $\Vert R-B_N\Vert\leq\delta$ is such that $R-\lambda$ is not one-to-one and hence not bounded below for any $\lambda\in K$. 

Now, given an arbitrary $A\in\tbx$ with $\Vert A\Vert <1$, define for each $N\geq 0$ an operator $T_N\in{\mathcal B}(X)$ as follows: 
\[ 
T_N:= P_N AP_N +\eta e_N^*\otimes e_{N+1} + J_N\bigl((1-\delta/2) B_N+\delta/2 \,S_N\bigr) (I- P_N),
\]
where $J_N:F_N\to X$ is the canonical inclusion, $S_N$ is the canonical forward shift acting on $F_N$ and $\eta>0$ is such that $\Vert A\Vert+\eta\leq 1$. Then $T_N\in\tbx$; and by the choice of $\delta$, we also have $T\in\g_K$. Since $T_N\xrightarrow{\sot} A$ as $N\to\infty$, this shows that $\g_K\cap \tbx$ is dense in $(\tbx,\sot)$.% and, for any compact set $K\subseteq \D$, denote by $\g_K$ the set of all $T\
\epf

\smallskip
Corollary \ref{T1l2} and Proposition \ref{T1l1} leave open the question to know if on
$X=c_0$ or $X=\ell_p$, $1< p<\infty$, \(p\neq 2\),
a typical $T\in({\mathcal T}_1(X),\sot)$ has a non-trivial invariant subspace.
It is possible to prove directly that a typical $T\in(\tbx,\sot)$ satisfies $\sigma(T)=\overline{\,\D}$. When $X=\ell_p$, \(1<p<\infty\), the results of \cite{M2} imply again that a typical $T\in({\mathcal T}_1(X),\sot)$ has a non-trivial invariant closed cone.

\section{\sote-$\,$typical contractions on $\ell_{p}$, $1<p<\infty$}\label{Section6}
Let $X=\ell_{p}$, $1<p<\infty$. The properties of \sote-$\,$typical operators of $\bbx$ may be very different from those of $\sot\,$-$\,$typical operators. Basically, the main difference between the two topologies is that the map 
$T\mapsto T^*$ is \sote$\,$-$\,$continuous (because $X$ is reflexive) but not \sot$\,$-$\,$continuous. This allows for example to show in a not too complicated way that for any $M>1$, an \sote$\,$-$\,$typical $T\in{\mathcal B}_M(X)$ is such that $T^*$ is hypercyclic (see \cite[Corollary 2.12]{GMM});  which implies that $T^*$ is not an isometry and $T$ has no eigenvalue. By homogeneity, it follows that a typical $T\in(\bbx,\sote)$ has no eigenvalue. (In the Hilbertian case, this was also proved in \cite[Proposition 6.3]{EM}.) So we may state
\begin{proposition}\label{Proposition 6.0}
 Let $X=\ell_{p}$, $1<p<\infty$. A typical  $T\in(\bbx,\emph{\sote})$ has no eigenvalue and $T^*$ is not an isometry.
\end{proposition}

In contrast, recall what we know for the topology \sot: an \sot$\,$-$\,$typical contraction on $\ell_{2}$ has plenty of eigenvalues; an \sot$\,$-$\,$typical contraction on $\ell_{p}$, $p>2$ has no eigenvalue but this seems to be rather hard to prove; and we have no idea of what is the situation for $1<p<2$. 

\smallskip
Among the results proved in Section \ref{Section2}, what remains valid for every $\ell_p\,$-$\,$space in the \sote \ setting is the following analogue of Proposition \ref{aameliorer}. In the Hilbertian case, this was proved in \cite[Proposition 6.11]{EM}.
\begin{proposition}\label{Proposition 6.1}
 Let $X=\ell_{p}$, $1<p<\infty$. A typical  $T\in(\bbx,\emph{\sote})$ has the following properties:
 $T-\lambda $ has dense range for every $\lambda \in\C$, and $\sigma (T)=\sigma _{ap}(T)=\ba{\, \D}$.
\end{proposition}
\begin{proof}
 Using the same notation as in the proof of Proposition \ref{aameliorer}, we observe that the sets $\g_{1}$ and $\g_{2}$, being \sot$\,$-$\,G_{\delta }$, are also \sote-$\,G_{\delta }$. The only issue is thus to prove that these two sets are \sote-$\,$dense in $\bbx$. Since the set of hypercyclic operators is in fact \sote-$\,$dense in $\bmx$ for any $M>1$ by \cite[Corollary 2.12]{GMM}, the same argument as in the proof of Proposition \ref{aameliorer} shows that $\g_{1}$ is \sote-$\,$dense in $\bbx$. As to the \sote-$\,$density of the set $\g_{2}$, the argument is exactly the same since the sequence of operators $(T_{N})$ associated in the proof to an arbitrary $A\in\bbx$ actually tends to $A$ for the topology \sote.
\end{proof}

%\smallskip
\begin{corollary}\label{invl2} A typical $T\in(\mathcal B_1(\ell_2),\emph{\sote})$ has a non-trivial invariant subspace.
\end{corollary}
\begin{proof} This follows immediately from Proposition \ref{Proposition 6.1} and the Brown-Chevreau-Pearcy Theorem from \cite{BCP2}, which states that any Hilbert space contraction whose spectrum contains the unit circle has a non-trivial invariant subspace.
\end{proof} 

%\smallskip
\begin{corollary}\label{Mul7} For any $1<p<\infty$, a typical $T\in(\mathcal B_1(\ell_p),\emph{\sote})$ has a non-trivial invariant closed cone. 
\end{corollary}
\begin{proof} This follows directly from the fact that an \sote-$\,$typical $T\in\bbx$ is such that $\sigma (T)=\ba{\,\D}$, together with the result of \cite{M2}; see the discussion before Proposition \ref{Mullererie}.
\end{proof}

\smallskip 
We do not know of any substantially simpler way than using the Brown-Chevreau-Pearcy Theorem to prove that an \sote-$\,$typical contraction on the Hilbert space has a non-trivial invariant subspace. Therefore, it is natural to investigate whether an \sote-$\,$typical $T\in\b_{1}(H)$ could enjoy some other properties, which would imply in a more elementary way that $T$ has a non-trivial invariant subspace. Theorem \ref{Sixieme th} answers one of the many questions one can ask in this vein: we show that an \sote-$\,$typical contraction on a (complex, separable, infinite-dimensional) Hilbert space does not commute with any non-zero compact operator on $H$. One can therefore not use this approach, via the Lomonosov Theorem, to show that an \sote-$\,$typical $T\in\bbh$ has a non-trivial invariant subspace. %We will actually prove a stronger result, given by Theorem \ref{Theorem 6.2} below. 

\begin{theorem}\label{Theorem 6.2}
 A typical $T\in(\bbh,\emph{\sote})$ %is such that $\Vert A\Vert_{e}=\Vert A\Vert$ for every operator $A\in\{T\}'$. In particular, $T$ 
 does not commute with any non-zero compact operator.
\end{theorem}

\smallskip Before embarking on the proof of this result, let us recall some standard notations.  
 For every $T\in\bh$, we  denote  by  $\{  T\}'$ the commutant of  $T$:	
\[
\{T\}':=\{A\in\bbh\;;\;AT=TA\}.
\]
%is the commutant of $T$. 
Also, %W
we denote by $\Vert T\Vert_{e}$ the essential norm of $T$, that is the distance of $T$ to the algebra ${\mathcal{K}}(H)$ of compact operators:
\[
\Vert T\Vert_{e}=\inf_{K\in{\mathcal{K}}(H)}\Vert T-K\Vert;
\]
and by $r_{e}(T)$ the essential spectral radius of T:
\[
r_{e}(T)=\max\,\{\vert z\vert \;;\;z\in\sigma _{e}(T)\}.
\]

\smallskip 
Our strategy for proving Theorem \ref{Theorem 6.2} is the following. Let us denote by $\mathcal M$ the set of all $T\in\bbh$ which do not commute with a non-zero compact operator, and by ${\mathcal M}_e$ the set of all $T\in\bbh$ such that $\Vert A\Vert=\Vert A\Vert_e$ for every $A\in\{ T\}'$. Clearly ${\mathcal M}_e\subseteq \mathcal M$. We will show that:
\begin{itemize}
\item[\sbt] ${\mathcal M}_e$ is $\sote$-$\,$dense in $\bbh$,
\item[\sbt] there is an $\sote$-$\,\gd$ set $\g\subseteq\bbh$ such that ${\mathcal M}_e\subseteq \g\subseteq\mathcal M$.
\end{itemize}

\smallskip
Standard examples of operators $T$ such that  $\Vert A\Vert=\Vert A\Vert_e$ for every $A\in\{ T\}'$ are the multiplication operators on the Hardy space $H^2(\D)$. For any function $\phi\in H^\infty(\D)$, let $M_\phi$ be the associated multiplication operator acting on $H^2(\D)$. Denoting by $\mathbf z\in H^\infty(\D)$  the function $z\mapsto z$, the commutant of $M_{\mathbf z}$ is equal to $\{ M_\phi;\; \phi\in H^\infty(\D)\}$. Moreover, standard arguments (see for instance \cite[Section 3.3]{M-AR}) show that  $\Vert M_{\phi }\Vert_{e}=\Vert\phi \Vert_{\infty}=\Vert M_{\phi}\Vert$ for any $\phi\in H^\infty(\D)$.
%$\sigma (M_{\varphi })=\sigma _{e}(M_{\varphi })=\ba{\varphi (\D)}$ and $\Vert M_{\varphi }\Vert_{e}=\Vert M_{\varphi}\Vert=\Vert\varphi \Vert_{\infty}$ for any $\varphi \in L^{\infty}(\T)$. 
With this example in mind, we present in Theorem \ref{Theorem 6.3} below a general condition on an operator $T$ acting on a Hilbert space implying that $\Vert A\Vert_{e}=\Vert A\Vert$ for any $A\in \{T\}'$. The assumptions on $T$ concern the eigenvectors of its adjoint $T^{*}$, and they permit to represent $T$ as a multiplication operator on a certain Hilbert space of holomorphic functions on the unit disk $\D$ (see \cite[Problem 85]{H}).

%Standard examples of operators satisfying the conclusions of Theorem \ref{Theorem 6.2} are multiplication operators. Denote by $M$ the multiplication operator on the space $L^{2}(\T)$ defined by $M:f\longrightarrow e^{i\theta }f$, $f\in L^{2}(\T)$. The commutant of $M$ consists of the operators $M_{\varphi }:f\longrightarrow\varphi f$, where $\varphi \in L^{\infty}(\T)$. Standard arguments (see for instance \cite[Section 3.1]{RR}) show that 
%$\sigma (M_{\varphi })=\sigma _{e}(M_{\varphi })=\ba{\varphi (\D)}$ and $\Vert M_{\varphi }\Vert_{e}=\Vert M_{\varphi}\Vert=\Vert\varphi \Vert_{\infty}$ for any $\varphi \in L^{\infty}(\T)$. We present in Theorem \ref{Theorem 6.3} below a general condition on an operator $T$ acting on a Hilbert space implying that $\Vert A\Vert_{e}=\Vert A\Vert$ for any $A\in \{T\}'$. The assumptions on $T$ concern the eigenvectors of its adjoint $T^{*}$, and they permit to represent $T$ as a multiplication operator on a certain Hilbert space of holomorphic functions on the unit disk $\D$ (see \cite[Problem 85]{H}).
\begin{theorem}\label{Theorem 6.3}
 Let $H$ be a Hilbert space, and let $T\in\bbh$. Suppose that there exists a family $(f_{w})_{w\,\in\D}$ of vectors of $H$ satisfying the following properties:
 \begin{enumerate}
  \item [\emph{(1)}] the map $w\mapsto f_{w}$ is holomorphic on $\D$;
  \item [\emph{(2)}] for every $w\in\D$, we have $T^{*}f_{w}=w\,f_{w}$;
  \item [\emph{(3)}] $[\,f_{w}\;;\;w\in\D\,]=H$;
  \item [\emph{(4)}] there exists a cyclic vector $x_{0}\in H$ for $T$such that $\pss{f_{w}}{x_{0}}=1$ for every $w\in\D$.
 \end{enumerate}
Then every operator $A\in\{T\}'$ is such that $\Vert A\Vert_{e}=\Vert A\Vert$.
\end{theorem}
\begin{proof}[Proof of Theorem \ref{Theorem 6.3}]
To each vector $x\in H$, we associate the map $\ti{x}: \D\to \C$ defined by \[ \ti{x}(w):=\pss{f_{\overline w}\,}{x}\qquad\hbox{for every $w\in \D$.}\] Let $\mathcal{H}:=\{\ti{x}\;;\;x\in H\}$. Then $\mathcal{H}$ is a linear space of holomorphic functions on $\D$, by (1); and by property (3), the map $U:H\to \mathcal H$ defined by $Ux:= \ti{x}$ is a linear isomorphism from $H$ onto $\mathcal H$. Therefore, $\mathcal H$ becomes a Hilbert space when endowed with the scalar product defined by $\pss{\ti{x}}{\ti{y}}:=\pss{x}{y}$, and $U:H\to\mathcal H$ is a unitary operator.% Observe that this scalar product is well-defined by property (3): if $\ti{x_{1}}=\ti{x_{2}}$ as elements of $\mathcal{H}$, then $\pss{f_{w}}{x_{1}-x_{2}}=0$ for every $w\in\D$, and hence $x_{1}=x_{2}$. By the definition of $\mathcal H$, the operator 
%$U:H\longrightarrow\mathcal{H}$ defined by $Ux=\ti{x}$, $x\in H$, is unitary.
\par\medskip
Next, we observe that $\mathcal{H}$ is a reproducing kernel Hilbert space, with reproducing kernels $\ti{{f}_{\overline w}}$, $w\in\D$: for every $w\in\D$ and every function $\ti{x}\in\mathcal{H}$, 
$\pss{\ti{{f}_{\overline w}}\,}{\ti{x}}=\pss{f_{\overline w}\,}{x}=\ti{x}(w)$. Observe also that $\ti{x_{0}}=\mathbf 1$ by property (4), so that the constant function $\mathbf 1$ belongs to $\mathcal{H}$; and that if we denote by $\mathbf w$ the function $ w\mapsto w$, then $\ti{\,T^{n}x_{0}} =\mathbf w^n$ for all $n\in\N$, by property (2). Hence, $\mathcal H$ contains all polynomial functions. Moreover, since $x_0$ is a cyclic vector for $T$ and 
$\pss{\ti{x}}{\mathbf w^n}=\pss{\ti{x}}{\ti{\,T^{n}x_{0}}}=\pss{x}{T^{n}x_{0}}$ for every $\ti{x}\in \mathcal H$ and all $n\ge 0$, the polynomial functions are dense in $\mathcal H$. %because $ Hence the function $[\,w\mapsto w^{n}\,]=\ti{T^{n}x_{0}}$ belongs to $\mathcal{H}$ for every $n\ge 1$, and the polynomial functions $w\mapsto p(w)$, $p\in\C[X]$ form a dense linear space in $\mathcal{H}$. 
%Indeed, if $\ti{x}\in\mathcal{H}$ is such that $\pss{\ti{x}}{w^{n}}=0$ for every $n\ge 0$, $\pss{\ti{x}}{\ti{T^{n}x_{0}}}=\pss{x}{T^{n}x_{0}}=0$ for every $n\ge 0$, and since $x_{0}$ is a cyclic vector for $T$, $x=0$.
\par\medskip
Using property (2), it is easily checked that $\ti{\,Tx}(w)= w\,\ti{x}(w)$ for every $\ti{x}\in\mathcal H$ and all $w\in\D$. This means that the multiplication operator $M_{\mathbf w}$ is well defined and bounded on $\mathcal H$ and that $T=U^{-1}M_{\mathbf w}U$. So we have $\Vert M_{\mathbf w}\Vert\leq 1$, and $T$ is unitarily equivalent to $M_{\mathbf w}$.
%Let $M_{w}$ be the multiplication operator by the independent variable $w$ on $\ti{\mathcal{H}}$: for every $\ti{x}\in\mathcal{H}$, $M_{w}\ti{x}=[\,w\mapsto w\ti{x}(w)\,]$. One easily checks that $M_{w}\ti{x}=\ti{Tx}$ for every $x\in H$, so that $M_{w}$ is a bounded operator on $\mathcal{H}$ with $\Vert M_{w}\Vert\le 1$. Also, $T=U^{-1}M_{w}U$, so that $T$ is unitarily equivalent to $M_{w}$.
\par\medskip
This description of $T$ as a multiplication operator on $\mathcal{H}$ yields a description of the commutant of $T$ as a certain family of multiplication operators. Indeed, let $A\in\{T\}'$. Then $B:=UAU^{-1}$ commutes with $M_{\mathbf w}$ on $\mathcal{H}$. Set $\phi :=B\mathbf 1\in\mathcal{H}$. For every $n\ge 0$, we have 
\[
M_{\mathbf w}^{n}\phi =M_{\mathbf w}^{n} B\mathbf 1=BM_{\mathbf w}^{n}\mathbf 1=B(\mathbf w^{n}).
\]
Hence $B p=\phi p$ for any polynomial function $p\in\mathcal H$.  Since the polynomial functions are dense in $\mathcal{H}$ and since the evaluation functionals $\ti{x}\mapsto \ti{x}(w)$ are continuous on $\mathcal H$, it follows that $B\ti{x}=\phi\ti{x}$ for every $\ti{x}\in\mathcal H$. Accordingly, we now write $B=M_\phi$.%we get that $B=M_{\phi }$, where $M_{\phi }$ acts as $\ti{x}\mapsto \phi \,.\,\ti{x}$ on $\mathcal{H}$ and is bounded on $\mathcal{H}$.
\par\medskip
Our goal is now to show that the function $\phi $ is bounded on $\D$, and that we have $\Vert M_{\phi }\Vert=\Vert\phi \Vert_{\infty}=\Vert M_{\phi }\Vert_{e}$. 

On the one hand, the operator $M_{\phi }$ is bounded on $\mathcal{H}$, so that $\Vert M_{\phi }^{*}\ti{f_{\overline w}}\Vert\le\Vert M_{\phi }\Vert\,\Vert\ti{f_{\overline w}}\Vert$ for every $w \in\D$. But $M_{\phi }^{*}\ti{f_{\overline w}}=\phi (w)\,\ti{f_{\overline w}}$ because $\ti{f_{\overline w}}$ is the reproducing kernel for $\mathcal H$ at $w$. So we get $\vert \phi (w)\vert \le\Vert M_{\phi }\Vert$ for every $w\in\D$; and hence $\phi $ is bounded with $\Vert\phi \Vert_{\infty}\leq \Vert M_\phi\Vert$.

\par\smallskip
On the other hand, observe that $M_{\mathbf w}$ is a completely non-unitary contraction on $\mathcal{H}$. Indeed, for every $w\in\D$, $\Vert M_{\mathbf w}^{*n}\ti{f_{w}}\Vert=\vert w\vert ^{n}\,\Vert\ti{f_{w}}\Vert$ tends to $0$ as $n$ tends to infinity; and since $[\,\ti{f_{w}}\;;\;w\in\D\,]=\mathcal{H}$, it follows that $\Vert M_{w}^{*n}\ti{x}\Vert\to 0$ for every $\ti{x}\in\mathcal{H}$. Being completely non-unitary, $M_{\mathbf w}$ admits a continuous functional calculus from $(H^{\infty}(\D),w^{*})$ into $(\b(\mathcal{H}),\sot)$, and the von Neumann inequality extends to functions of $H^{\infty}(\D)$ (see \cite{NF} or \cite{Ba-Ca}). Hence $\Vert\phi (M_{\mathbf w})\Vert\le \Vert\phi \Vert_{\infty}$. Now $\psi (M_{\mathbf w})=M_{\psi }$ for every $\psi\in H^\infty(\D)$, \mbox{\it i.e.} $\psi(M_{\mathbf w})\, \widetilde x=\psi\, \widetilde x$ for all $\ti{x}\in\mathcal H$. Indeed, this is clear if $\psi$ is a polynomial function; and the general case follows from the continuity of the functional calculus, the $w^*$-$\,$density of the polynomial functions in $H^\infty(\D)$, and the continuity of the evaluation functionals on $(H^{\infty}(\D),w^{*})$ and on the space $\mathcal H$. In particular $\phi(M_{\mathbf w})=M_\phi$, and hence $\Vert M_\phi\Vert \leq \Vert \phi\Vert_\infty$. %If we denote, for every $0<r<1$, by $\phi _{r}$ the function $w\mapsto \phi (rw)$, then $\phi _{r}\to \phi $ in $H^{\infty}(\D)$ for the $w^{*}$-topology, so that $\phi _{r}(M_{w})\to \phi(M_{w}) $ for the \sot. Now, $\phi _{r}(M_{w})=M_{\phi _{r}}$, and thus $\phi _{r}\ti{x}$ tends to $\phi (M_{w})\ti{x}$ in $\mathcal{H}$ as $r$ tends to $1$. Since $\phi _{r}(w)\to \phi (w)$ for every $w\in\D$, we deduce that 
%\[
%\pss{f_{w}}{\phi (M_{w})\ti{x}}=\phi (w)\,\pss{f_{w}}{\ti{x}}=\pss{f_{w}}{M_{\phi }\ti{x}}
%\]
%for every $w\in\D$, and thus $\phi (M_{w})=M_{\phi }$. Hence $\Vert M_{\phi }\Vert\le \Vert\phi \Vert_{\infty}$ by the von Neumann inequality, and so $\Vert M_{\phi }\Vert=\Vert\phi \Vert_{\infty}$.
\par\smallskip
The last step is to prove that $\Vert M_{\phi }\Vert_{e}\ge \Vert\phi \Vert_{\infty}$. Since $M_{\phi }^{*}f_{\,\overline w}=\phi (w)f_{\,\overline w}$ for every $w\in\D$, we see that $\phi (\D)\subseteq \sigma (M_{\phi }^{*})$, and hence $\ba{\phi (\D)}\subseteq \sigma (M_{\phi }^{*})$.  Choose $\lambda_0\in \overline{\phi(\D)}$ such that $\vert \lambda_0\vert=\Vert \phi\Vert_\infty$. Since $\Vert M_{\phi }^{*}\Vert=\Vert M_\phi\Vert=\Vert\phi \Vert_{\infty}$, $\lambda_0$ is a boundary point of the spectrum $\sigma(M_\phi^*)$ of $M_\phi^*$.
Moreover, since $\phi(\D)$ is an open subset of $\C$, $\lambda_0$ is not an isolated point of $\sigma(M_\phi^*)$.
% 
% 
% 
% 
% But $\Vert M_{\phi }^{*}\Vert=\Vert M_\phi\Vert=\Vert\phi \Vert_{\infty}$, so we have in fact $\sigma (M_{\phi }^{*})=\ba{\phi (\D)}$. Since the inequality $\Vert M_{\phi }\Vert_{e}\ge \Vert\phi \Vert_{\infty}$ is trivially true when $\phi $ is a constant function, we can suppose without loss of generality that $\phi $ is non-constant. Choose $\lambda_0\in \overline{\phi(\D)}$ such that $\vert \lambda_0\vert=\Vert \phi\Vert_\infty$. Clearly, $\lambda_0$ is a boundary point of $\overline{\phi(\D)}$, \mbox{\it i.e.} $\lambda_0\in\partial\sigma(M_\phi^*)$ since we have just seen that  $\overline{\phi(\D)}=\sigma(M_\phi^*)$. Moreover, $\sigma(M_\phi^*)=\overline{\phi(\D)}$ has no isolated point since $\phi(\D)$ is an open subset of $\C$. 
Thus, we see that $\lambda_0\in\partial\sigma(M_\phi^*)$ and $\lambda_0$ is not an isolated point of $\sigma(M_\phi^*)$. These two properties imply that $\lambda_0\in\sigma_e(M_\phi^*)$ (see \mbox{e.g} \cite[Theorem XI.6.8]{Con}). So we have $r_{e}(M_{\phi }^{*})\ge \Vert\phi \Vert_{\infty}$, and hence $\Vert M_{\phi }\Vert_{e}=\Vert M_{\phi }^*\Vert_{e}\ge\Vert\phi \Vert_{\infty}$.%, and so $\Vert M_{\phi }\Vert_{e}=\Vert M_{\phi }\Vert$. 
\par\smallskip
Since the operator $A\in\{ T\}'$ we started with is unitarily equivalent to $B=M_{\phi }$, we can now conclude that $\Vert A\Vert_{e}=\Vert A\Vert$.
\end{proof}
\begin{proof}[Proof of Theorem \ref{Theorem 6.2}] %Recall that we denote by $\g$ the set of all $T\in\bbh$ which do not commute with a non-zero compact operator, and by ${\mathcal M}_e$ the set of all $T\in\bbh$ such that $\Vert A\Vert=\Vert A\Vert_e$ for every $A\in\{ T\}'$
Let $(K_{q})_{q\ge 1}$ be a sequence of non-zero compact operators which is dense in ${\mathcal{K}}(H)$ for the operator norm topology. For each $q\ge 1$, set %$\rho _{q}:=\Vert K_{q}\Vert$ and  
\[ \mathfrak B_{q}:=\ba{B}\,(K_{q},\Vert K_{q}\Vert/3)\subseteq\b(H).\]

 We first observe that
\[
{\mathcal{K}}(H)\setminus\{0\}\subseteq\bigcup_{q\ge 1}\mathfrak B_{q}.
\]
Indeed, let $K$ be a non-zero compact operator on $H$, and let $\varepsilon >0$ be such that $\Vert K\Vert>4\varepsilon $. Choose $q\ge 1$ such that $\Vert K-K_{q}\Vert<\varepsilon $. Then 
$\Vert K_{q}\Vert>3\varepsilon $, so that 
\[
K\in B(K_{q},\varepsilon )\subseteq B(K_{q},\Vert K_{q}\Vert/3)\subseteq \mathfrak B_{q}.
\]

Next, for each $q\ge 1$, consider the set
\[
{\mathcal F}_{q}:=\{T\in\bbh\;;\;\exists\,A\in \mathfrak B_{q}\;;\;AT=TA\}\cdot	
\]
We prove the following fact, where we use in a crucial way the topology \sote\ (the argument would break down if we were to consider the  topology \sot\  instead).
\begin{fact}\label{Fact 6.4}
 For each $q\ge 1$, the set $\mathcal F_{q}$ is closed in $(\bbh,\sote)$.
\end{fact}
\begin{proof}[Proof of Fact \ref{Fact 6.4}]
Let 
\[
{\mathfrak{F}}_{q}:=\{(T,A)\in\bbh\times \mathfrak B_{q}\;;\;AT=TA\},
\]
so that \(\mathcal F_{q}\) is the projection of \({\mathfrak{F}}_{q}\) along the second coordinate. Since \((\mathfrak B_{q},\texttt{WOT})\) is compact, it is enough to show that ${\mathfrak{F}}_{q}$ is closed in $(\bbh,\sote)\times(\mathfrak B_{q},\texttt{WOT})$. 

Let $(T_{n},A_{n})$ belong to ${\mathfrak{F}}_{q}$ for every $n\in\N$, and suppose that 
$T_{n}\xrightarrow{\sote} T$ and $A_{n}\xrightarrow{\wot} A$, with $(T,A)\in\bbh\times \mathfrak B_{q}$. For every vectors $x,y\in H$ and every $n\ge 1$, we have
\[
\pss{y}{A_{n}T_{n}x}=\pss{y}{T_{n}A_{n}x},\quad \textrm{i.e.}\quad \pss{A_{n}^{*}y}{T_{n}x}=\pss{T_{n}^{*}y}{A_{n}x}.
\]
On the one hand, $A_{n}^{*}y\to A^{*}y$ and $A_{n}x\to Ax$
weakly, and on the other hand, $T_{n}x\to Tx$ and $T_{n}^{*}y\to T^{*}y$ in norm. It follows that
\[
 \pss{A_{n}^{*}y}{T_{n}x}\to \pss{A^{*}y}{Tx}\quad \textrm{and}\quad 
 \pss{T_{n}^{*}y}{A_{n}x}\to \pss{T^{*}y}{Ax},
\]
so that $\pss{y}{ATx}=\pss{y}{TAx}$. Thus $AT=TA$, which proves that ${\mathfrak{F}}_{q}$ is indeed closed in $(\bbh,\sote)\times (\mathfrak B_{q},\texttt{WOT})$. \end{proof}
Let us now define 
\[ \g:= \bbh\setminus \bigcup_{q\geq 1} \mathcal F_q,\]
which is a $\gd$ subset of $(\bbh,\sote)$ by Fact \ref{Fact 6.4}. 

\smallskip Recall also that we denote by by $\mathcal M$ the set of all $T\in\bbh$ which do not commute with a non-zero compact operator -- that is, the set which we want to prove is $\sote$-$\,$comeager in $\bbh$ -- and that
\[{\mathcal M}_{e}=\{T\in\bbh\;;\;\forall\,A\in\{T\}',\ \Vert A\Vert_{e}=\Vert A\Vert\}.\] 

%The set ${\mathcal F}:=\bigcup_{q\ge 1}{\mathcal F}_{q}$ is thus $F_{\sigma }$ in $(\bbh,\sote)$ by Fact \ref{Fact 6.4}, so that $\mathcal G:=\bbh\setminus{\mathcal F}$ is $G_{\delta }$. Now, an operator $T\in\bbh$ belongs to ${\mathcal G}$ if and only if any operator $A\in\{T\}'$ satisfies $\Vert A-K_{q}\Vert>\Vert K_{q}\Vert/3$ for every $q\ge 1$. The link between this observation and Theorem \ref{Theorem 6.3} is given by
\begin{fact}\label{Fact 6.5} We have $\mathcal M_e\subseteq \g\subseteq\mathcal M$.
 %If $A\in\bbh$ is such that $\Vert A\Vert=\Vert A\Vert_{e}$, then $\Vert A-K_{q}\Vert>\Vert K_{q}\Vert/3$.
\end{fact}
\begin{proof}[Proof of Fact \ref{Fact 6.5}] By the very definition of $\g$ and since ${\mathcal{K}}(H)\setminus\{0\}\subseteq\bigcup_{q\ge 1}\mathfrak B_{q}$, it is clear that $\g\subseteq\mathcal M$. Moreover, an operator $T\in\bbh$ belongs to ${\mathcal G}$ if and only if any operator $A\in\{T\}'$ satisfies $\Vert A-K_{q}\Vert>\Vert K_{q}\Vert/3$ for every $q\ge 1$. So, proving that $\mathcal M_e\subseteq\g$ amounts to showing that if $A\in\bh$ is such that $\Vert A\Vert_e=\Vert A\Vert$, then $\Vert A-K_{q}\Vert>\Vert K_{q}\Vert/3$ for every $q\ge 1$. 
Suppose that $\Vert A-K_{q}\Vert\le\Vert K_{q}\Vert/3$ for some $q\ge 1$. If $\Vert A\Vert>\Vert K_{q}\Vert/3$, then 
$\Vert A\Vert_{e}>\Vert K_{q}\Vert/3$, and hence $\Vert A-K_{q}\Vert>\Vert K_{q}\Vert/3$, which is impossible. If $\Vert A\Vert\le\Vert K_{q}\Vert/3$, then $\Vert A-K_{q}\Vert\ge 2\Vert K_{q}\Vert/3$; hence $\Vert K_{q}\Vert/3\ge 2\Vert K_{q}\Vert/3$, which is impossible since $K_{q}\neq 0$.
\end{proof}
%Let us now set
%\[{\mathcal G}_{e}:=\{T\in\bbh\;;\;\forall\,A\in\{T\}',\ \Vert A\Vert_{e}=\Vert A\Vert\}.\] Then ${\mathcal G}_{e}\subseteq{\mathcal G}$ by Fact \ref{Fact 6.5}, and since operators from ${\mathcal G}_{e}$ do not commute with any non-zero compact operator, 
By Fact \ref{Fact 6.5}, it now suffices to prove the following proposition in order to terminate the proof of Theorem \ref{Theorem 6.2}.
\begin{proposition}\label{Proposition 6.6}
 The set ${\mathcal M}_{e}$ is dense in $(\bbh,\emph{\sote})$.
\end{proposition}
\begin{proof}[Proof of Proposition \ref{Proposition 6.6}]
We are going to prove that the set ${\mathcal{C}}$ of operators $T\in\bbh$ satisfying the assumptions of Theorem \ref{Theorem 6.3} is \sote-$\,$dense in $\bbh$. Fix an orthonormal basis $(e_{j})_{j\ge 0}$ of $H$, and consider the associated class $\mathcal T_1(H)$. Since ${\mathcal{C}}$ is stable under unitary equivalence, and since the orbit of ${\mathcal{T}}_{1}(H)$ under unitary equivalence is \sote-$\,$dense in $\bbh$ by Lemma \ref{Lemma 2.1 bis}, it suffices to show that 
${\mathcal{C}}\cap{\mathcal{T}}_{1}(H)$ is \sote-$\,$dense in ${\mathcal{T}}_{1}(H)$.
\par\medskip
Let $A\in{\mathcal{T}}_{1}(H)$. We need to find $T\in \mathcal C\cap \mathcal T_1(H)$ such that $T$ is \sote-$\,$close to $A$. 

\smallskip Let $N\in\N$, let $\eta>0$, and let $B\in \mathcal T_1(H)$ be such that  $\Vert (B-A)P_N\Vert <\eta$. We define an operator $T\in\mathcal B(H)$ as follows:
\[ T:= BP_N+S(I-P_N)\,,\]
where $S$ is the canonical forward shift.
%where $S_{N,\eta}$ is the operator defined by 
%\[ S_{N,\,\eta}e_n:=\begin{cases}
%0&{\rm if}\ 0\leq N-1,\\
%\eta\,  e_{N+1}&{\rm if} \ n=N,\\
%e_{n+1}&{\rm if} \ n\geq N+1.
%\end{cases}
%\]

Since $B\in\mathcal T_1(H)$, we see that $T\in\mathcal T_1(H)$. Moreover, it is not hard to check that if $N$ is large enough and $\eta$ is small enough then, regardless of the choice of $B$ such that $\Vert (B-A)P_N\Vert <\eta$, the operator $T$ is \sote-$\,$close to $A$. So we fix $N$ large enough and $\eta$ small enough, and our goal is now to choose $B$ in such a way that $T\in\mathcal C$, \mbox{\it i.e.} there exists a family $(f_{w})_{w\in\D}$ of vectors of $H$ satisfying properties (1)\,--\,(4) of Theorem \ref{Theorem 6.3}.% with respect to the operator $A$.
\par\medskip
Let us denote by $B_N:= P_NBP_N$ the compression of the operator $B$ to the space $E_N=[e_0,\dots ,e_N]$, and set $b_N:= \pss{Be_N}{e_{N+1}}$. Then
%Let us denote by $B_{0}^{*}:=P_{N_{0}+1}A^{*}P_{N_{0}+1}$  the compression of the operator $A^{*}$ to the space $[e_n \;;\; 0\le n\le N_0 +1]$. Then
\[
T^{*}e_{j}=
\begin{cases}
 B_{N}^{*}e_{j}&\textrm{if}\ 0\le j\le N,\\
 b_N\, e_{N}&\textrm{if}\ j=N+1\\
 e_{j-1}&\textrm{if}\ j>N+1.
\end{cases}
\]

We now choose $B$ in such a way that the following properties hold true.
\begin{enumerate}
\item[-] The eigenvalues of the matrix $B_{N}^{*}$ are all distinct.
\item[-] Write $\sigma (B_{N}^{*})=\{\lambda_{0},\dots,\lambda_{N}\}$ and, for each $0\le n\le N$, let  $v_{n}\in E_N$ be an eigenvector of $B_{N}^{*}$ associated to the eigenvalue $\lambda_{n}$ (so that $(v_0,\dots ,v_N)$ is a basis of $E_N$). Then $\pss{v_{n}}{e_{0}}\neq 0$ for $n=0,\dots ,N$, and the vector $e_{N}$ can be written as a linear combination $e_{N}=\sum_{n=0}^{N}\beta _{n}v_{n}$ where all the coefficients $\beta _{n}$ are non-zero;
%\item[-] $\pss{v_{n}}{e_{0}}\neq 0$ for $n=0,\dots ,N$.
\end{enumerate}

\par\medskip
For every $w\in\D\setminus \sigma (B_{N}^{*})$, we define
 \begin{align*}
g_{w}:=&b_N\,(w-B_{N}^{*})^{-1}e_{N}+\sum_{j\geq N+1}w^{j-(N+1)}e_{j}\\
&=b_N\sum_{n=0}^{N}\dfrac{\beta _{n}}{w-w_{n}}\, v_{n}+\sum_{j\geq N+1}w^{j-(N+1)}e_{j}.
\end{align*}
It is not difficult to check that $T^{*}g_{w}=w\,u_{w}$ for every $w\in\D\setminus\sigma (B_{N}^{*})$. Now, we set
 \[f_w:= p(w) g_w\qquad\hbox{where}\quad p(w):=\prod_{n=0}^{N}(w-\lambda_{n}).\]
 Then the map $w\mapsto f_w$ extends holomorphically to the whole disk $\D$, and  we have $T^{*}f_{w}=w\,f_{w}$ for every $w\in\D$. So conditions (1) and (2) from Theorem~\ref{Theorem 6.3} are satisfied. In order to check condition (3), suppose that $x\in H$, written as $x=\sum_{j\ge 0}x_{j}e_{j}$, is such that $\pss{f_{w}}{x}=0$ for every $w\in\D$. Then $\pss{g_{w}}{x}=0$ for every $w\in\D\setminus\{w_{0},\dots,w_{N}\}$, \mbox{\it i.e.} 
\[
b_N\sum_{n=0}^{N}\dfrac{\beta _{n}\,\pss{v_{n}}{x}}{w-\lambda_{n}}=-\sum_{j\geq N+1}\ba{\,x_{j}}\,w^{j-(N+1)}\quad \textrm{for every}\ w\in\D\setminus\{\lambda_{0},\dots,\lambda_{N}\}.
\]
Since the right hand side of this identity  defines a holomorphic function on $\D$ %The function $ w\mapsto b\sum_{n=0}^{N_{0}+1}\frac{\beta _{n}\,\pss{v_{n}}{x}}{w-w_{n}}$ can thus be extended to a holomorphic function on $\D$, 
and since $b_N>0$ and $\beta _{n}\neq 0$ for every $0\le n\le N$, it follows that $\pss{v_{n}}{x}=0$ for every $0\le n\le N$. This implies on the one hand that  $x_{j}=0$ for every $0\le j\le N$ (since the vectors $v_{n}$ form a basis of $E_N$); and, on the other hand, that the holomorphic function 
$w\mapsto \sum_{j\geq N+1}\ba{\,x_{j}}\,w^{j-(N+1)}$ is identically zero on $\D$, and so $x_{j}=0$ for every $j\geq N+1$. We have proved that $x=0$, which yields condition (3).
\par\medskip
Lastly, we have to prove condition (4), namely that there exists $x_{0}\in H$ which is a cyclic vector for $T$ and such that $\pss{f_{w}}{x_{0}}=1$ for every $w\in\D$. To do this, we note that 
\[\pss{f_{w}}{e_{N+1}}=p(w)\qquad{\rm and}\qquad  \pss{f_{w}}{e_{0}}=p(w)\sum_{n=0}^{N}\frac{\beta _{n}\,\pss{v_{n}}{e_{0}}}{w-\lambda_{n}}:= q(w).\]%On the one hand, we have $$; and on the other hand, $$. The function 
%$w\mapsto \pss{f_{w}}{e_{0}}$ is a polynomial in the variable $w$, which we denote by $q$. 
By the choice of $p$, we see that $q$ is a polynomial. Explicitely:
\[
q(w)=\sum_{n=0}^{N}\beta _{n}\,\pss{v_{n}}{e_{0}}\,\prod_{\genfrac{}{}{0pt}{1}{k=0}{k\neq n}}^{N}(w-\lambda_{k}),\quad w\in\D.
\]
For every $n=0,\dots,N$, we  have
\[
q(\lambda_{n})=\beta _{n}\,\pss{v_{n}}{e_{0}}\,\prod_{\genfrac{}{}{0pt}{1}{k=0}{k\neq n}}^{N}(\lambda_{n}-\lambda_{k}).
\]
Since $\beta _{n}\neq 0$, $\pss{v_{n}}{e_{0}}\neq 0$ and $\lambda_0,\dots ,\lambda_N$ are pairwise distinct, it follows that $q(\lambda_{n})\neq 0$ for every $0\le n\le N$. Since the roots of the polynomial $p$ are exactly the numbers $\lambda_{0},\dots,\lambda_{N}$, we thus see that the polynomials $p$ and $q$ have no common zeros; so there exist two polynomials $r$ and $s$ such that $rp+sq=1$. Let $x_{0}:=r(T)e_{N+1}+s(T)e_{0}$. Then, we have for every $w\in\D$:
\begin{align*}
 \pss{f_{w}}{x_{0}}&=\pss{r(T)^{*}f_{w}}{e_{N+1}}+\pss{s(T)^{*}f_{w}}{e_{0}}\\
 &=r(w)\,\pss{f_{w}}{e_{N+1}}+s(w)\,\pss{f_{w}}{e_{0}}\\
 &=(rp+sq)(w)=1.
\end{align*}
Also $\pss{f_{w}}{T^{n}x_{0}}=w^{n}$ for all $n\ge 0$, and thus $\pss{f_{w}}{q(T)x_{0}}=q(w)=\pss{f_{w}}{e_{0}}$, \mbox{\it i.e.}  $\pss{f_{w}}{q(T)x_{0}-e_{0}}=0$ for every $w\in\D$. So $q(T)x_{0}=e_{0}$ by (3), and since $e_{0}$ is cyclic for $T$, it follows that 
$x_{0}$ is cyclic as well. This proves condition (4), and terminates the proof that $T$ satisfies the assumptions of Theorem \ref{Theorem 6.3}.
\end{proof}
Proposition \ref{Proposition 6.6} is thus proved, and Theorem \ref{Theorem 6.2} follows.
\end{proof}
\begin{remark}\label{Remark 6.7}
 It is also true that a typical $T\in(\bbh,\sot)$ does not commute with any non-zero compact operator, but this is much easier to prove thanks to the Eisner-M\'{a}trai Theorem. Indeed, the operator $B_{\infty}$
on $\ell_{2}(\Z_{+},\ell_{2})$ being a co-isometry whose powers $B_{\infty}^{n}$ tend to zero for \sot, it does not commute with any non-zero compact operator. %Since a typical $T\in(\bbh,\sot)$ is unitarily equivalent to $B_{\infty}$, it does not commute either with a non-zero compact operator.
 \end{remark}

\section{Further remarks and questions}\label{Questions}
We collect in this final section some of the many questions which arise naturally in connection with the results presented above. 
% The first one that comes to mind may be the following.
% 
% \begin{question} Let $X$ be an arbitrary (separable) Banach space. Is it true that either a typical $T\in(\bbx,\sot)$ has a non-trivial invariant subspace, or a typical $T\in\bbx$ does not have a non-trivial invariant subspace?
% \end{question} 
\par\smallskip
We know \cite{EM} that the orbit of the backward shift of infinite multiplicity $B_{\infty}$ under unitary equivalence is comeager in $(\bbh,\sot)$, where $H$ is a Hilbert space. Since many of the properties of \sot$\,$-$\,$typical contractions on a Hilbert space are also true on $X=\ell_{1}$, it is  natural to ask:
\begin{question}\label{Question 1}
 Let $X=\ell_{1}$. Does there exist an operator $T_{0}\in\bbx$ whose similarity orbit intersected with $\bbx$ is \sot$\,$-$\,$comeager in $\bbx$?
\end{question}
One can ask a similar question when $X=\ell_{p}$ for $p\neq 1$ and $p\neq 2$, both for the topology \sot\ and for the topology \sote, although a positive answer does not seem very likely:
\begin{question}\label{Question 2}
 Let $X=\ell_{p}$, $1<p<\infty$, $p\neq 2$. Does there exist an operator $T_{0}\in\bbx$ whose similarity orbit intersected with $\bbx$ is \sot$\,$-$\,$comeager (resp. \sote-$\,$comeager) in $\bbx$? When $p=2$, does there exist an operator $T_{0}\in\bbh$ whose orbit under unitary equivalence is \sote-$\,$comeager in $\bbh$?
\end{question}
% It makes sense to ask a similar question for  \sote. Again, the answer should be negative:
% \begin{question}\label{Question 3}
%  Let $X=\ell_{p}$, $1<p<\infty$. Does there exist an operator $T_{0}\in\bbx$ whose similarity orbit intersected with $\bbx$ is \sote-$\,$dense in $\bbx$ (when $p=2$, one can ask the same question for the orbit of $T_{0}$ under unitary equivalence)?
% \end{question}

In another direction, the proof of Theorem \ref{l1} suggests the following question.
\begin{question} Let $X$ be a Banach space. Assume that a typical $T\in(\bbx,\sot)$ is not one-to-one. Does it follow that a typical $T\in\bbx$ is such that $\dim \ker(T)=\infty$?
\end{question}

%Another natural question is the following.
%\begin{question} Let $X$ be a Banach space. Is it true that either a typical $T\in(\bbx,\sot)$ is one-to-one, or a typical $T\in(\bbx,\sot)$ is not one-to-one? 
%\end{question}

The results of Sections \ref{Section3} and \ref{Section4} suggest of course the following questions:
\begin{question}\label{Question 4}
 Let $X=\ell_{p}$, $1<p<\infty$, $p\neq 2$ or $X=c_{0}$. Does a typical $T\in(\bbx,\sot)$ have a non-trivial invariant subspace? What about a typical $T\in(\bbx,\sote)$?
\end{question}

\begin{question}\label{Question 5}
 Let $X=\ell_{p}$, $1<p<2$ or $X=c_0$. Is it true that for any $M>1$, a typical $T\in(\bbx,\sot)$ is such that $(MT)^*$ is hypercyclic? In the $\ell_p\,$-$\,$case, is it true at least that a typical $T\in\bbx$ has no eigenvalue?%have an eigenvalue? What if $X=c_{0}$?
\end{question}
%Note that an operator acting on a reflexive space does not have any eigenvalue as soon as some multiple of its adjoint is hypercyclic. So the following question seems rather natural.
%\begin{question} Let $X=\ell_p$ with $1<p<\infty$ and $p\neq 2$. Is it true that for any $M>1$, a typical $T\in (\bmx,\sot)$ is such that $T^*$ is hypercyclic?
%\end{question}

\begin{question} Let $X=\ell_{p}$, $1<p<2$. Is it true that every \sote-$\,$comeager subset of $\bbx$ is also \sot$\,$-$\,$comeager?
\end{question} 
\begin{question}\label{Question 6}
 Let $X=c_{0}$. Does a typical $T\in(\bbx,\sot)$ have a non-trivial invariant closed cone?
\end{question}

% 
% Regarding invariant subspaces, the situation for the topology \sote\ is not better understood than for  \sot.
% \begin{question}\label{Question 7}
%  Let $X=\ell_{p}$, $1<p<\infty$, $p\neq 2$ or $X=c_{0}$. Does a typical $T\in(\bbx,\sote)$ have a non-trivial invariant subspace? 
% \end{question}
%Regarding the generic absence of eigenvalues on $\ell_p$, $p>2$, a crucial step in the proof is Theorem \ref{Localisation vp lp}. This suggests
%\begin{question}\label{sansinteret?} For which Banach spaces $X$ is the following statement true: if $L$ is a $K_{\sigma }$ subset of $\ba{\,\D}$ such that the set $\{T\in\bbx\;;\;\sigma (T)\cap L=\emptyset\}$ is dense in $(\bbx,{\sot})$, then the set 
%$
%\{T\in\bbx\;;\;\sigma_p(T)\cap L=\emptyset\}
%$ is comeager in $(\bbx,{\sot})$?
%\end{question}

We have observed in Remark \ref{Remark 6.7} (respectively proved in Theorem \ref{Theorem 6.2}) that a typical $T\in(\bbh,\sot)$ (resp. $T\in(\bbh,\sote)$) does not commute with a non-zero compact operator. This motivates several questions.
\begin{question}\label{Question 8}
 Let $X=\ell_{p}$, $1<p<\infty$, $p\neq 2$. Is it true that a typical $T\in(\bbx,\sot)$ does not commute with any non-zero compact operator? What about a typical $T\in(\bbx,\sote)$?
\end{question}
The Lomonosov Theorem implies that an operator commuting with a non-zero compact operator has a non-trivial invariant subspace, but the full statement of the Lomonosov Theorem is actually much stronger: 
\par\smallskip
\emph{If $T\in\bbx$ is such that its commutant $\{T\}'$ contains an operator $A$ which is not a multiple of the identity operator and which commutes with a non-zero compact operator, then $T$ has a non-trivial invariant subspace}. 
\par\smallskip
We call (LH) the hypothesis of the Lomonosov Theorem. 
It was proved by 
Hadwin, Nordgren, Radjavi, and Rosenthal in \cite{HNRR} that there exists an operator $T$ on a complex separable Hilbert space $H$ which does not satisfy (LH): for every $A\in\{T\}'\setminus \C Id$, one has $\{A\}'\cap {\mathcal{K}}(H)=\{0\}$. 
% One can show, using a blend of the techniques from \cite{HNRR} and from the proof of Theorem \ref{Theorem 6.3}, that operators $T\in\bbh$ not satisfying (LH) are \sote-$\,$dense in $\bbh$. 
We do not know whether such operators are typical in $(\bbh,\sote)$:
\begin{question}\label{Question 19}
 Let $H$ be a Hilbert space. Is it true that a typical $T\in(\bbh,\sote)$ does not satisfy (LH)?
\end{question}
Observe that $B_{\infty}\in\b_{1}(\ell_{2}(\Z_{+},\ell_{2}))$ does satisfy (LH): this follows from a result of \cite{Cow}, which shows that the unweighted forward shift $S$ on $\ell_2$ commutes with a non-zero compact operator; as a consequence, a typical operator $T\in(\bbh,\sot)$ satisfies (LH).
\begin{question}\label{Question 20}
 Let $X=\ell_{p}$, $1<p<\infty$, $p\neq 2$. Is it true that a typical operator $T\in(\bbx,\sot)$ satisfies (LH)? What about a typical $T\in(\bbx,\sote)$?
\end{question}

\end{document}